\documentclass[
    11pt,
    a4paper
]{article}

\usepackage[utf8]{inputenc}
\usepackage[english]{babel}
\usepackage{amsmath}
\usepackage{amssymb}
\usepackage{amsthm}
\usepackage{bbm}
\usepackage{mathtools}
\usepackage{enumitem}
\usepackage[numbers]{natbib}
\usepackage{graphicx}
\usepackage{subcaption}
\usepackage{xcolor}
\usepackage{hyperref}
\usepackage[normalem]{ulem} 

\usepackage[top=3cm,bottom=3cm,left=3cm,right=3cm]{geometry}

\theoremstyle{definition}
\newtheorem{definition}{Definition}[section]
\newtheorem{remark}[definition]{Remark}

\newtheorem{assumption}{Assumption}

\theoremstyle{plain}
\newtheorem{theorem}[definition]{Theorem}
\newtheorem{proposition}[definition]{Proposition}
\newtheorem{lemma}[definition]{Lemma}
\newtheorem{corollary}[definition]{Corollary}

\newcommand\norm[1]{\left\lVert#1\right\rVert}

\definecolor{forestgreen}{RGB}{34,139,34}

\newcommand{\NN}{\mathbb{N}}
\newcommand{\ZZ}{\mathbb{Z}}

\newcommand{\RR}{\mathbb{R}}

\newcommand{\ind}{\mathbbm{1}}

\newcommand{\PP}{\mathbf{P}}
\newcommand{\EE}{\mathbf{E}}

\title{Central limit theory for Peaks-over-Threshold partial sums of long memory linear time series}
\author{Ioan Scheffel$^{a}$, Marco Oesting$^{a,b}$, Gilles Stupfler$^{c}$}
\date{$^{a}$ {\small Institute for Stochastics and Applications, University of Stuttgart, D-70563 Stuttgart, Germany} \\[1ex]
$^{b}$ {\small Stuttgart Center for Simulation Science (SC SimTech), University of Stuttgart, D-70569 Stuttgart, Germany} \\[1ex]
$^{c}$ {\small Univ Angers, CNRS, LAREMA, SFR MATHSTIC, F-49000 Angers, France}}

\begin{document}

\maketitle

\begin{abstract}
  Over the last 30 years, extensive work has been devoted to developing central limit theory for partial sums of subordinated long memory linear time series.
A much less studied problem, motivated by questions that are ubiquitous in extreme value theory, is the asymptotic behavior of such partial sums when the subordination mechanism has a threshold depending on sample size, so as to focus on the right tail of the time series.
This article substantially extends longstanding asymptotic techniques by allowing the subordination mechanism to depend on the sample size in this way and to grow at a polynomial rate, while permitting the innovation process to have infinite variance.
The cornerstone of our theoretical approach is a tailored
reduction principle, which enables the use of classical results on partial sums of long memory linear processes.
In this way we obtain asymptotic theory for certain Peaks-over-Threshold estimators with deterministic or random thresholds. Applications cover both heavy- and light-tailed regimes, yielding unexpected results which, to the best of our knowledge, are new to the literature. A simulation study illustrates the relevance of our findings in finite samples.
\end{abstract}


\section{Introduction}
\label{sec:intro}
\subsection{Background and literature review}
\label{sec:intro:lit}

A causal linear (or MA($\infty$)) time series $X_t=\sum_{j=0}^{\infty} a_j \varepsilon_{t-j}$, $t\in \mathbb{Z}$, built upon i.i.d.~innovations $(\varepsilon_t)$ with at least a fractional moment and real-valued coefficients $(a_j)$, is (roughly speaking) said to have long memory if the decay of the coefficients to 0 is slow enough to yield a divergent series, yet fast enough to accommodate tail heaviness of the innovations. Long memory linear time series models
include the ARFIMA model introduced by~\cite{grajoy1980}, which is arguably the simplest example of a stochastic process having long memory. The ARFIMA model has been successfully applied to various fields, such as finance and economics~\cite{bhaswa2006} and medicine~\cite{wanetal2024}, and is the subject of active mathematical research, for example on fundamental aspects linked to model selection~\cite{huaetal2022}.

Central limit theorems for partial sums of long memory linear time series are well-known; see among others~\cite[Chapter~4]{beranLongMemoryProcessesProbabilistic2013}. Instead of the linear time series $(X_t)$, however, in many applications one is interested in a subordinated time series, that is, a transformation $G(X_t)$. Typical examples are $G(X_t)=|X_t|^p$, yielding the empirical $p$th moment, or $G(X_t)=\ind\{X_t\le x\}$, providing an estimator of the distribution function of $X_0$ at the point $x$. This problem is discussed by \cite{hoLimitTheoremsFunctionals1997} (innovations with finite eighth moment), \cite{hoAsymptoticExpansionEmpirical1996} (innovations with finite fourth moment and $G(X_t)=\ind\{X_t\le x\}$), \cite{surgailisStableLimitsSums2004} (innovations with finite second moment but no fourth moment) and \cite{koulAsymptoticsEmpiricalProcesses2001} (innovations with finite first moment but no second moment); see also~\cite[Section~III]{dehlingEmpiricalProcessTechniques2002a} and~\cite[Section~16.1]{kulikHeavyTailedTimeSeries2020}. The cornerstone of a standard proof is to show a (weak) reduction principle of the form
\begin{align} \label{eq:reduction-intro}
  n^{-d-1/\min(\nu,2)}
  \sum_{t=1}^n
  \left(
  G(X_t)
  -
  \EE[G(X_0)]
  \right)
  =
  n^{-d-1/\min(\nu,2)}
    G_{\infty}'(0)
  \sum_{t=1}^n
  X_t
  + \operatorname{o}_{\PP}(1)
  \,,
\end{align}
where $\nu = \sup \{ p>1\mid \norm{\varepsilon}_{L^p(\PP)} < \infty \}$ is the (algebraic) number of finite moments of the innovations, assumed to be larger than 1,
$d$
is a long memory parameter, reminiscent of the fractional differencing parameter in ARFIMA models,
that satisfies
$d\in (0,1-1/\min(\nu,2))$ and is
such that $(n^{1-d} a_n)$ converges to a finite, nonzero limit,
and
$G_{\infty}'(0) = \mathrm{d} (\EE[G(X_0+x)]) / \mathrm{d}x \big|_{x=0}$ is assumed not to vanish.
By~\cite[Theorem~4.6 ($\nu \geq 2$), Theorem~4.17 ($1<\nu<2$)]{beranLongMemoryProcessesProbabilistic2013}, the weak limit of the left-hand side of \eqref{eq:reduction-intro}
is then a stable distribution, which is non-Gaussian for $\nu\in(1,2)$, and whose scale depends on the asymptotic behavior of the coefficients $(a_j)$, tail heaviness of the innovations, and the derivative $G'_{\infty}(0)$.

Fitting extreme value analysis into this framework requires shifting from fixed transformations $G$ to a sequence of sample-size-dependent transformations $(G_n)$.
Typical examples arise from counting threshold exceedances -- the simplest example among a broader class of estimators called Peaks-over-Threshold (PoT) estimators --, that is,
\[
\frac{1}{n}\sum_{t=1}^n\ind\{X_t>u_n\}, \qquad \mbox{ with } G_n(x)=\ind\{x>u_n\}
\]
for a threshold sequence $(u_n)$ with $u_n\to\infty$ as $n\to\infty$, or
\[
\frac{1}{n}\sum_{t=1}^n \log(X_t/u_n) \ind\{X_t>u_n\}, \qquad \mbox{ with } G_n(x)=(\log(x) - \log(u_n))\ind\{x>u_n\}
\]
which is reminiscent of the Hill estimator~\cite{hil1975} of the extreme value index for heavy-tailed data. In practice the threshold sequence is often chosen as an order statistic from the sample, that is, $u_n=X_{n-k:n}$ for an intermediate sequence $k=k(n)$ with $k\to\infty$ and $k/n\to 0$ for $n\to\infty$, where $X_{1:n}\le \ldots \le X_{n:n}$ is the ordered sample. This corresponds to the realistic setting where the chosen threshold estimates an unknown quantile threshold of the form $q_X(1-k/n)$. The asymptotic behavior of these estimators, whether with deterministic or random thresholds, is well-studied under various notions of short-range dependence; see for instance~\cite{hsingTailIndexEstimation1991} in a strong mixing setting, as well as later work by \cite{dreesExtremeQuantileEstimation2003} for extensions to a wider class of estimators in the $\beta-$mixing framework under an anti-clustering condition (see also \cite[Section~V]{dehlingEmpiricalProcessTechniques2002a}). Recent contributions, such as \cite{oestingLongMemoryMaxstable2023a} (based on the notion of dependence introduced in~\cite{kulspo2021}), or \cite{baiTailAdversarialStability2024} and \cite{caoTailIndexEstimation}, have focused on weakening the mixing requirement, which may be difficult to check in a given model.

The question of handling long-range dependence in PoT extreme value analysis has remained much less explored. Works by~\cite{kulikTailEmpiricalProcess2011a},~\cite{bilivakul2019} and \cite[Section~16.3]{kulikHeavyTailedTimeSeries2020} study the tail empirical process with random thresholds, that is,
    $(\frac{1}{k}
      \sum_{t=1}^n
      \ind\{
      X_t
      >
      X_{n-k:n}s
      \})_{s\geq 1}$,
and show that it does not suffer from long memory in certain stochastic volatility models.~\cite{kulsou2013} discusses the existence of multivariate versions of the extremogram in the same class of processes. This is, to the best of our knowledge, the only class of processes that has been studied in this context; in particular, long memory linear time series have been left untouched.

\subsection{Contribution of the paper}
\label{sec:intro:contribution}

The main contribution of our article is to derive central limit theorems for subordinated long memory linear time series, in an extreme value framework where the subordination mechanism depends on sample size.
This is challenging because the mixing properties that are typically assumed in the extreme value literature (in particular, strong mixing, $\beta-$mixing or anti-clustering conditions) do not hold in this context. To be more specific, under mild smoothness assumptions on the innovation sequence, and provided each member of the sequence of transformations $(G_n)$ grows at most like a (suitable) power function and is supported on an interval of the form $(u_n,\infty)$ with $u_n\to\infty$, we show that the centered and scaled subordinated linear time series
\[
n^{-d-1/\min(\nu,2)}\sum_{t=1}^n
(
G_n(X_t)
-
\EE[G_n(X_0)]
)
\]
has the same asymptotic behavior as
\[
\left.
\frac{\mathrm{d}}{\mathrm{d}x}
\EE[G_n(X_0+x)]
\right|_{x=0}
n^{-d-1/\min(\nu,2)}\sum_{t=1}^n
X_t
\,.
\]
This new, tailored reduction principle is achieved by explicitly deriving an upper bound for
an appropriate moment of the difference between the two random sums.
This technical result, specific to the long memory framework, leads to our main central limit theorem for partial sums of subordinated long memory linear time series, yielding the same limiting distribution as the partial sums of the original time series,
but with a different rate of convergence due to focusing only on the extremes of $(X_t)$.

Our results stand in contrast to what one would expect from the available i.i.d.~or short-range-dependent theory of extreme values. For example, with heavy-tailed innovations, we observe that the speed of convergence is \textit{faster} than the classical long memory speed arising from a setting with fixed transformation. This is markedly different from independent or short-range dependent settings, where it is well-known that the speed of convergence of PoT extreme value estimators based on the $k$ largest observations only is $\sqrt{k}$, which is \textit{slower} than the standard speed $\sqrt{n}$. A further contribution of our work is to allow the replacement of the deterministic threshold $(u_n)$ with its random counterpart; in doing so, we obtain a second set of asymptotic results for the extreme value estimators under consideration, showing that, unlike in the i.i.d.~and short memory settings~\cite{stu2019}, Hill estimators with deterministic and random thresholds have different asymptotic distributions. A similar phenomenon was observed by~\cite{kulikTailEmpiricalProcess2011a} in stochastic volatility models with long memory.

We organize the presentation of these theoretical findings as follows. In Section~\ref{sec:model}, we describe the model and assumptions in detail. Section~\ref{sec:main} contains the announced central limit results. In addition, we assess the practical relevance of these theoretical results in a simulation study, presented in Section~\ref{sec:simstudy}. While confirming the expected rate of convergence, the simulations show that in finite samples convergence to the actual form of the asymptotic distribution is slow, even in relatively simple settings.
This is partly due to the extreme value setting and also due to the intrinsic difficulty of working with long memory time series, which is pronounced already at central levels. Section~\ref{sec:discussion} finally discusses our findings and perspectives for future work. The proofs of all the results are deferred to
an online Supplementary Material document.

\section{Model and main assumptions}
\label{sec:model}

Throughout this article, $(\Omega,\mathcal{A}, \PP)$ denotes a probability space rich enough to support all random variables under consideration. For $p\ge 1$, we write $L^p(\PP)$ for the space of random variables with finite $p$-th moment.
For a generic, continuous random variable $Z$, we let $F_Z$ (resp.~$f_Z$, $q_Z$) denote its cumulative distribution function (resp.~probability density function, quantile function).
We write $a\land b = \min\{a,b\}$ and $a\lor b = \max\{a,b\}$ for $a,b\in\mathbb{R}$, and use the symbol $\sim$ to denote asymptotic equivalence of sequences and functions.
We write $\lesssim$ to mean inequality up to a generic multiplicative constant $C>0$ that can change from place to place. If the dependence of this constant on the surrounding variables is relevant to the argument, we will make this explicit, by saying, for example, that $C=C_n$ depends on $n$, or that $C$ is independent of $n$.
The symbol $C^p(\RR)$ denotes the vector space of real-valued functions on $\RR$ with continuous derivatives up to order $p$.

Let $(\varepsilon_t)_{t\in \ZZ}$ be a sequence of independent and identically distributed (i.i.d.)~copies of a random variable $\varepsilon$. Set
\[
\nu = \nu_{\varepsilon} := \sup \{ p>1  \mid \| \varepsilon \|_{L^p(\PP)} < \infty \}.
\]
We have $\nu=\infty$, for example, if $\varepsilon$ is standard Gaussian, and if the complementary distribution function $x\mapsto \PP[|\varepsilon|>x]$ is regularly varying with index $-1/\xi<0$, then $\nu = 1/\xi$. In this work, we focus on linear time series $X_t=\sum_{j=0}^{\infty} a_j \varepsilon_{t-j}$, where $(a_j)$ is a sequence of real-valued coefficients, featuring long memory in the following sense.

\begin{assumption}[Long memory]
  \label{asu:coef}
  It holds that $\nu>1$, and that the sequence
  $(n^{1-d} a_n)$ converges to a finite, nonzero limit $c_a$
  for some $d\in (0,1-1/\alpha)$, where $\alpha:=2\land \nu$.
\end{assumption}

\begin{remark}[On Assumption~\ref{asu:coef}]
    \label{rmk:asu:coef}
    Assumption~\ref{asu:coef} guarantees that the process $(X_t)$ is well-defined and stationary; see~\cite[Section~15.3]{kulikHeavyTailedTimeSeries2020}. It imposes that the innovations have at least a finite moment of order larger than 1, and that the coefficient sequence $(a_j)$, while forming a divergent series and thus granting the long memory property to $(X_t)$, is able to accommodate potentially heavy tails of the innovations.
\end{remark}

Our main goal is to derive a central limit theorem for partial sums of subordinated long memory linear time series, in the sense of Section 4.2.5 in \cite{beranLongMemoryProcessesProbabilistic2013}, taking the form
\[
v_n \sum_{t=1}^n (G_n(X_t) - \EE[G_n(X_0)]) \stackrel{\mathrm{d}}{\longrightarrow} Z
\,,
\]
where $(G_n)$ is a given sequence of real-valued functions, and $(v_n)$, $Z$ are respectively a nonrandom sequence and a nondegenerate stable limiting random variable to be determined. This is tailored to extreme value theory within the Peaks-over-Threshold (PoT) framework, where $G_n$ may for instance be $G_n(x)
    =
    \ind\{x>u_n\}
    $
(for counting the number of exceedances above a high threshold $u_n\to\infty$) or
    $
    G_n(x)
    =
    \left(
    \log(x)
    -
    \log(u_n)
    \right)
    \ind\{x>u_n\}
    $,
resulting in the Hill estimator~\cite{hil1975}. The first of the regularity assumptions we require regards the sequence of functions $(G_n)$.

\begin{assumption}[Regularity of $G_n$]
  \label{asu:G_n}
  It holds that
  \[
    |G_n(x)|
    \ \lesssim \
    (1+|x|)^{\gamma_G} \ind\{x>u_n\}
    \,,
  \]
  for some $\gamma_G\geq 0$ and a deterministic sequence $(u_n)$ of thresholds with $u_n\to\infty$ as $n\to\infty$.
  Here, $\lesssim$ denotes inequality up to a multiplicative constant $C>0$ independent of $n$ and $x$.
\end{assumption}

\begin{remark}[On Assumption~\ref{asu:G_n}]
    \label{rmk:asu:G_n}
    The functions $G_n$ appearing in the simple threshold exceedance estimator and the Hill estimator~\cite{hil1975} satisfy Assumption~\ref{asu:G_n}, with $\gamma_G=0$ and any $\gamma_G>0$, respectively.
    In Section~\ref{sec:main}, we introduce further
    technical restrictions on $\gamma_G$ that prevent $G_n$ from growing too fast, depending on the tail heaviness of the innovations.
\end{remark}

Our second main regularity assumption focuses on the innovation sequence $(\varepsilon_t)$ and is tailored to the two main regimes of interest in our work. On the one hand, we will consider the (easier) case when $\varepsilon$ has a finite second moment with sufficiently fast decay of the characteristic function. On the other hand, we shall be interested in the (harder) case of heavy-tailed time series whose innovations have regularly varying densities and a finite first moment, but no finite second moment.

\begin{assumption}[Regularity of the innovations]
  \label{asu:f}
  Suppose that $F_\varepsilon\in C^2(\RR)$, and that $\varepsilon$ has a symmetric distribution. Furthermore, assume that one of the following alternatives holds:
  \begin{enumerate}[label=(\roman*)]
    \item (Case $\alpha=2$) One has $\EE[\varepsilon^2]<\infty$, and, for some $\delta>0$, the characteristic function of $\varepsilon$ is such that $(1+|s|)^{\delta} \EE[\exp(\mathrm{i}s \varepsilon)]$ is uniformly bounded in $s\in \mathbb{R}$.
    \item (Case $\alpha\in(1,2)$) Assume that there is a finite positive constant $A$ such that $\lim_{x\to\infty} x^{\alpha} \PP[\varepsilon > x] = \lim_{x\to\infty} x^{\alpha} \PP[\varepsilon < -x] = A/2$, and suppose that the probability density function $f_\varepsilon \in C^1(\mathbb{R})$ satisfies:
    \begin{align*}
    |f'_\varepsilon(x)|
    &
    \ \lesssim \
    (1+|x|)^{-\alpha}
    \qquad\text{for}\ x\in\RR\,,
    \\
    |
    f'_\varepsilon(x)
    -
    f'_\varepsilon(y)
    |
    &
    \ \lesssim \
    |x-y|
    \cdot
    (1+|x|)^{-\alpha}
    \qquad\text{for}\ x,y\in\RR\,,|x-y|<1\,,
    \end{align*}
    where $\lesssim$ denotes inequality up to a multiplicative constant $C_\varepsilon>0$ depending only on the distribution of $\varepsilon$.
  \end{enumerate}
\end{assumption}

Our main interest in this work is Case (ii) in Assumption~\ref{asu:f}, that is, the setting $\alpha\in (1,2)$, which should be viewed as the harder case of the two. Let us point out that our assumptions in Case (i) might be weakened further using dedicated techniques from \cite{hoAsymptoticExpansionEmpirical1996}. The rationale for the precise form of our Assumption~\ref{asu:f}(i) is to allow a comparison of asymptotic behavior between the infinite and finite variance contexts in a unified fashion.

\begin{remark}[On Assumption~\ref{asu:f}(i)]
    \label{rmk:asu:f1}
Verifying Assumption~\ref{asu:f}(i) typically boils down to checking integrability properties of the density $f_{\varepsilon}$ and its derivative. Take, for example, any distribution of $\varepsilon$ having finite second moment with density vanishing asymptotically at $\pm\infty$ and integrable first derivative. Then on the one hand, for $|s|\le 1$, it holds that
\[
|\EE[\exp(\mathrm{i}s\varepsilon)]| \leq 1 \leq \frac{1}{1+|s|}
    \,,
\]
and on the other hand, for $|s|> 1$, integration by parts yields
\[
    \left| \EE[\exp(\mathrm{i}s\varepsilon)] \right|
    \ = \
    \left| \frac{1}{\mathrm{i}s}
    \int_\RR
    \exp(\mathrm{i}s x)
    f'_\varepsilon
    (x)
    \,\mathrm{d}x \right|
    \ \leq \
    \frac{1}{|s|}
    \norm{f_{\varepsilon}'}_{L^1(\RR)}
    \ \lesssim \
    \frac{1}{1+|s|}
\]
due to the inequality $1/|s|\le 2/(1+|s|)$ for $|s|>1$. Then Assumption~\ref{asu:f}(i) is satisfied with $\delta=1$. Therefore, it covers a large class of distributions, in particular centered Gaussian distributions, symmetric Beta distributions with parameter larger than 2, and Student distributions with $\nu>2$ degrees of freedom. Note that this nonetheless ensures that $X_0$ has an infinitely differentiable distribution function; see~\cite[Lemma~1]{giraitisAsymptoticNormalityRegression1996}.
\end{remark}

\begin{remark}[On Assumption~\ref{asu:f}(ii)]
    \label{rmk:asu:f2}
  Assumption~\ref{asu:f}(ii) is satisfied by symmetric $\alpha$-stable ($S\alpha S$) distributed innovations; see~\cite[after (2.2)]{koulAsymptoticsEmpiricalProcesses2001}. More generally, the two inequalities there hold if $f'_\varepsilon$ and $f''_\varepsilon$ are continuous and regularly varying (with index $\leq -\alpha$). In particular, the Student $t-$distribution satisfies this stronger criterion. We note that the symmetry property is assumed mainly to simplify the application of the central limit theorem for partial sums when $\alpha\in (1,2)$; it would suffice for the innovations to be centered. Without symmetry, the limiting distribution becomes asymmetric, but the convergence rates remain the same, see \cite[Theorem~8.3.5]{kulikHeavyTailedTimeSeries2020}.
\end{remark}

\section{Main results}
\label{sec:main}
This section is split into four subsections. In Section~\ref{sec:main:reduction}, we develop our reduction principle, which is the key technical tool behind our main results. The reader primarily interested in central limit theory can skip this subsection and go directly to Section~\ref{sec:main:CLT}, where we provide and discuss a general central limit theorem arising as a corollary of this reduction principle. This result is then applied to the cases when the stationary distribution of $(X_t)$ is heavy-tailed (resp.~light-tailed) in Section~\ref{sec:main:hvy} (resp.~Section~\ref{sec:main:light}).

\subsection{Technical tool: Reduction principle}
\label{sec:main:reduction}
At a conceptual level, the reduction principle works as follows.
For $k\in \mathbb{Z}$, let $\mathcal{F}_k=\sigma(\varepsilon_{k},\varepsilon_{k-1},\ldots)$ denote the past $\sigma$-algebra generated by the sequence $(\varepsilon_{t})$ up to index $k$. For $k\ge 0$, let $X_{t,k}:=\sum_{j=0}^k a_j\varepsilon_{t-j}$ be the $k$-truncated time series, and write
\begin{align*}
  G_{k,n}(y)
  \ := \
  \EE[G_n(X_{0,k}+y)]
  \qquad\text{and}\qquad
  G_{k,n}'(y)
  \ = \
  \frac{\mathrm{d}}{\mathrm{d}y}
  G_{k,n}(y)
  \,,
\end{align*}
where we set $G_{\infty,n}(y):=\EE[G_n(X_{0,\infty} + y)]=\EE[G_n(X_0 + y)]$ for $k=\infty$ in the natural way.
The reduction principle consists in showing that
\begin{align}
  \label{eq:heuristic}
  \begin{split}
  &
  \sum_{t=1}^n
  \left(
  G_n(X_t)
  \ - \
  \EE[G_n(X_0)]
  \right)
  \\&
  \ = \
  \sum_{t=1}^n
  \sum_{k=0}^{\infty}
  \left(
  \EE
  \left[
  G_n(X_t)
  \mid \mathcal{F}_{t-k}
  \right]
  \ - \
  \EE
  \left[
  G_n(X_t)
  \mid \mathcal{F}_{t-(k+1)}
  \right]
  \right)
  \\&
  \ = \
  \sum_{t=1}^n
  \sum_{k=0}^{\infty}
  \left(
  G_{k-1,n}(X_t - X_{t,k-1})
  \ - \
  G_{k,n}(X_t - X_{t,k})
  \right)
  \\&
  \ \approx\
  \sum_{t=1}^n
  \sum_{k=0}^{\infty}
  \left(
  G_{k,n}(X_t - X_{t,k-1})
  \ - \
  G_{k,n}(X_t - X_{t,k})
  \right)
  \\&
  \ \approx\
  \sum_{t=1}^n
  \sum_{k=0}^{\infty}
  a_k \varepsilon_{t-k}
  G'_{k,n}(X_t-X_{t,k})
  \\&
    \ = \
  G'_{\infty,n}(0)
  \sum_{t=1}^n
  X_t
  \ + \
  \sum_{t=1}^n
  \sum_{k=0}^{\infty}
  a_k \varepsilon_{t-k}
  \left(
  G_{k,n}'(X_t - X_{t,k})
  \ - \
  G_{\infty,n}'(0)
  \right)
     \,.  \end{split}
\end{align}
Here, the asymptotic profile turns out to be determined only by the linear term,
for which central limit theory is readily available.
Before we make the argument in \eqref{eq:heuristic} rigorous,
we provide central limit theory for partial sums of long memory linear time series.
For a proof of the latter we refer to \cite[Theorem~4.6 ($\alpha = 2$) and Theorem~4.17 ($\alpha\in(1,2)$)]{beranLongMemoryProcessesProbabilistic2013}.
   \begin{theorem}[Central limit theorems for partial sums]
     \label{thm:clt_partsum}
     Let Assumptions~\ref{asu:coef} and~\ref{asu:f} hold. Then
   \begin{align*}
     n^{-d-1/\alpha}
    \sum_{t=1}^n
    X_t
    \stackrel{\mathrm{d}}{\longrightarrow}
    Z_\alpha
    \qquad\text{as}\ n\to\infty\,,
  \end{align*}
  where
  \begin{align*}
      Z_\alpha
    &
    \ \
    \text{has distribution}
    \ \
    \begin{cases}
      \mathcal{N}(0,\sigma^2)
      &\ \ \text{if}\ \ \alpha = 2 \ \text{with $\EE[\varepsilon^2]<\infty$,}
      \\
      S\alpha S(\eta)
      &\ \ \text{if}\ \
      \alpha\in(1,2)\,,
    \end{cases}
  \end{align*}
  with
  variance
\begin{align*}
    \sigma^2
    &:=
      c_a^2\EE[\varepsilon^2]
      \frac
      {
      B(1-2d,d)}
      {d(2d+1)} \ \mbox{ when } \EE[\varepsilon^2]<\infty,
  \intertext{or scale}
  \eta
  &:=
\frac{c_a}{d}
    \cdot
    \left(
    A
    \cdot
      \frac
      {\Gamma(2-\alpha)\cos(\pi\alpha/2)}
      {1-\alpha}
      \right)^{1/\alpha}
    \left(
    \int_{-\infty}^{1}
    \left(
    (1-v)^{d}_+
    -
    (-v)^d_+
    \right)^{\alpha}
    \,\mathrm{d}v
    \right)^{1/\alpha}
      \,.
  \end{align*}
   \end{theorem}
   \begin{remark}[On the limiting distribution in Theorem~\ref{thm:clt_partsum}]
     We define the distribution $S\alpha S(\eta)$ in the usual sense, having characteristic function $x \mapsto  \exp(-\eta^\alpha|x|^\alpha)$.
     Also, it may not be immediately obvious that the integral in the definition of $\eta$ is finite in our setting. To clarify, using Jensen's inequality and the condition $1<(1-d)\alpha<2$, we have
\begin{align*}
  \int_1^\infty
  \left(
  (1+v)^d
  \ - \
  v^d
  \right)^{\alpha}
  \,\mathrm{d}v
  &=
  \int_1^\infty
  \left(
  d
  \int_0^1
  (v+s)^{-(1-d)}
  \,\mathrm{d}s
  \right)^{\alpha}
  \,\mathrm{d}v
  \\&
  \le
  \int_1^\infty
  \left(
  \int_0^1
  (v+s)^{-(1-d)\alpha}
  \,\mathrm{d}s
  \right)
  \,\mathrm{d}v
  \\&
  =
  \int_0^1
  \left(
  \int_1^\infty
  (v+s)^{-(1-d)\alpha}
  \,\mathrm{d}v
  \right)
  \,\mathrm{d}s
  \\&
  \le
  \int_1^\infty
  v^{-(1-d)\alpha}
  \,\mathrm{d}v
     <
  \infty\,.
\end{align*}
Therefore,
\begin{align*}
  &
  \int_{-\infty}^1
    \left(
    (1-v)^{d}_+
    -
    (-v)^d_+
    \right)^{\alpha}
    \,\mathrm{d}v
    \\&\
    =\
 \int_{-1}^1
    \left(
    (1-v)^{d}_+
    -
    (-v)^d_+
    \right)^{\alpha}
    \,\mathrm{d}v
    +
    \int_{1}^\infty
    \left(
    (1+v)^d
    -
    v^d
    \right)^\alpha
    \,\mathrm{d}v
    <
    \infty\,.
\end{align*}
   \end{remark}
   The next result makes the argument in \eqref{eq:heuristic} rigorous. In the proof we adapt techniques of
\cite{koulAsymptoticsEmpiricalProcesses2001} to our extreme value setting.
\begin{theorem}[$L^r-$moment bound]
  \label{thm:approx}
  Let Assumptions~\ref{asu:coef},~\ref{asu:G_n} and~\ref{asu:f} hold. Then
  \[
  G_{\infty,n}'(0)
  \ := \
  \left.
  \frac{\mathrm{d}}{\mathrm{d}x}
  \EE[G_n(X_0+x)]
  \right|_{x=0}
  \]
is well-defined. Moreover, if $\gamma_G$ is such that there are $\gamma,r\in\RR$ satisfying
       \begin{align}
         \label{cond:r_gamma}
         \gamma_G
         \ < \
         \gamma
         \ < \
  \min \left\{
  \frac{d}{1-d}
  \,,
         1
          -
         \frac{1}{r(1-d)}
         \,,
         \frac{\alpha}{r}
          -
         1
  \right\}
           \qquad\text{and}\qquad
         \frac{1}{1-d}
         \ < \
         r
         \ < \
         \alpha
         \,,
       \end{align}
       we have
       \begin{align*}
         &
         \norm{
         \sum_{t=1}^n
         (G_n(X_t)
         \ - \
         \EE[G_n(X_0)])
         \ - \
         G_{\infty,n}'(0)
         \sum_{t=1}^n X_t
         }_{L^r(\PP)}
         \\&
         \ \lesssim \
         u_n^{(\gamma_G-\gamma)}
         n^{1+1/r-(1-d)(1+\gamma)}
         \ + \
         n^{1/r}
         \left(
         \norm{G_n(X_0)}_{L^r(\PP)}
         \lor
         |G_{\infty,n}'(0)|
         \right)
         \\&
         \ \lesssim\
         n^{1+1/r - (1-d)(1+\gamma)}
         \,,
       \end{align*}
       where $\lesssim$ denotes inequality up to a multiplicative constant $C>0$ independent of $n$.
\end{theorem}
It follows from Theorem~\ref{thm:approx} that if $G_{\infty,n}'(0)\neq 0$ for sufficiently large $n$, then
\[
         \norm{
         \frac{n^{-d-1/\alpha}}{G_{\infty,n}'(0)} \sum_{t=1}^n
         (G_n(X_t)
         \ - \
         \EE[G_n(X_0)])
         \ - \
         n^{-d-1/\alpha}
         \sum_{t=1}^n X_t
         }_{L^r(\PP)}
         \ \lesssim\
         \frac{n^{\kappa(\gamma,r) - d - 1/\alpha}}{|G_{\infty,n}'(0)|}
         \,,
\]
where
$\kappa(\gamma,r):=1+1/r-(1-d)(1+\gamma)$.
Therefore, if $n^{\kappa(\gamma,r) - d - 1/\alpha}/G_{\infty,n}'(0) \to 0$ as $n\to\infty$, then the following $L^r(\PP)-$reduction principle applies:
\[
  \frac{n^{-d-1/\alpha}}{G_{\infty,n}'(0)} \sum_{t=1}^n
         (G_n(X_t)
         \ - \
         \EE[G_n(X_0)])
         \ - \
         n^{-d-1/\alpha}
         \sum_{t=1}^n X_t \stackrel{L^r(\PP)}{\longrightarrow} 0
         \qquad\text{as $n\to\infty$.}
\]
\begin{remark}[On the existence of $\gamma$ and $r$ in Theorem~\ref{thm:approx}]
\label{rmk:cond1}
Note that the upper bound for $\gamma$ in \eqref{cond:r_gamma} gets maximal for $$r =r^{\star} := \frac{1}{2} \left(\alpha + \frac 1 {1-d}\right),$$ that is,~for $r$ being equal to the midpoint of its domain. In this case, the upper bound is given by
\[
\gamma^{\star} = \gamma(r^{\star}) = \min\left\{\frac{d}{1-d}, \frac{\alpha(1-d)-1}{\alpha(1-d)+1} \right\} = \begin{cases}
    \dfrac{d}{1-d}
    &
    \ \text{if}\
    \dfrac{1}{(1-d)(1-2d)}
    < \alpha\,,\\[10pt]
    \dfrac{\alpha(1-d)-1}{\alpha(1-d)+1}
    &
    \
    \text{otherwise.}
  \end{cases}
\]
Thus, $\gamma_G < \gamma^{\star}$ is necessary for the existence of a pair $(\gamma,r)$
satisfying~\eqref{cond:r_gamma}. Conversely, if $\gamma_G < \gamma^{\star}$, the pair $(\gamma^{\star}-\delta,r^{\star})$ satisfies Condition~\eqref{cond:r_gamma} for sufficiently small $\delta>0$. Consequently, Condition~\eqref{cond:r_gamma} is nonempty if and only if $\gamma_G < \gamma^{\star}$. Simple calculations show that
\[
\max_{d\in (0,1-1/\alpha)} \min\left\{\frac{d}{1-d}, \frac{\alpha(1-d)-1}{\alpha(1-d)+1} \right\} = \frac{8(1-1/\alpha)}{(1+\sqrt{9-8(1-1/\alpha)})(3+\sqrt{9-8(1-1/\alpha)})}
\]
with the maximum attained at $d(\alpha)=(3-\sqrt{9-8(1-1/\alpha)})/4$. This constitutes a hard upper bound for $\gamma_G$ in our framework. Note that $d(1)=0$ and $d(2)=(3-\sqrt{5})/4\approx 0.191$.
\end{remark}

As illustrated by the example $G_n(x) = \ind\{x > u_n\}$, where $G_{\infty,n}'(0) = f_{X_0}(u_n)$, the condition $n^{\kappa(\gamma,r)-d-1/\alpha}/G_{\infty,n}'(0)\to 0 $ restricts the growth of $(u_n)$ to infinity. Consequently, the choice of $(\gamma,r)$ has a significant impact on the admissible growth rates of the threshold sequence $(u_n)$.
Therefore, we aim to relax this condition as much as possible, which amounts to minimizing the function $(\gamma,r) \mapsto \kappa(\gamma,r)$. It turns out (see Proposition~\ref{prop:optim} in the supplement) that $\kappa$ is minimized over the admissible values of $\gamma$ and $r$ at
       \begin{equation}
       \label{eq:gamma0r0}
 (\gamma_0, r_0) := \left\{ \begin{array}{l}
    \left( \dfrac{d}{1-d}, \ \alpha(1-d) \right)
    \ \text{ if } \
    \dfrac{1}{(1-d)(1-2d)}
    < \alpha\,,\\[10pt]
    \left( \dfrac{\alpha(1-d)-1}{\alpha(1-d)+1}, \ \dfrac{1}{2}
    \left(
    \alpha + \dfrac{1}{1-d}
    \right) \right)
     \ \text{ otherwise,}
    \end{array} \right.
  \end{equation}
  yielding the optimal value $\kappa_0 := \kappa(\gamma_0,r_0)$
  as
  \begin{align}
  \label{eq:kappa_0}
  \kappa_0
    =
  \left\{ \begin{array}{l}
    \dfrac{1}{\alpha(1-d)}
    \ \text{ if } \
    \dfrac{1}{(1-d)(1-2d)}
    < \alpha\,,\\[10pt]
    \dfrac{2(1-d) + (1-\alpha(1-d)(1-2d))}{\alpha(1-d)+1} = d
    +
    (1-d)
    \dfrac{3-\alpha(1-d)}{\alpha(1-d)+1}
    \ \text{ otherwise.}
  \end{array} \right.
\end{align}
With the notation of Remark~\ref{rmk:cond1}, $\gamma_0=\gamma^{\star}$, and $r_0=r^{\star}$ when $\alpha(1-d)(1-2d)\geq 1$. Noting that $\kappa_0-d-1/\alpha<0$, we obtain the following optimized moment bound.
\begin{theorem}[Optimal $L^r-$moment bound]
  \label{thm:approx_optimal}
  Let Assumptions~\ref{asu:coef},~\ref{asu:G_n} and~\ref{asu:f} hold. If $\gamma_G<\gamma_0$, then, for any $\delta>0$,
  \[
         \norm{
         \sum_{t=1}^n
         (G_n(X_t)
         \ - \
         \EE[G_n(X_0)])
         \ - \
         G_{\infty,n}'(0)
         \sum_{t=1}^n X_t
         }_{L^{r_0}(\PP)}
         \ \lesssim\
         n^{\kappa_0+\delta}
         \,,
  \]
  where $\lesssim$ denotes inequality up to a multiplicative constant $C>0$ independent of $n$.
  Consequently, if one has $G_{\infty,n}'(0) \neq 0$ for sufficiently large $n$, and if for some $\delta>0$ one has $n^{\kappa_0+\delta - d - 1/\alpha}/G_{\infty,n}'(0) \to 0$ as $n\to\infty$, then
  \[
  \frac{n^{-d-1/\alpha}}{G_{\infty,n}'(0)} \sum_{t=1}^n
         (G_n(X_t)
         \ - \
         \EE[G_n(X_0)])
         \ - \
         n^{-d-1/\alpha}
         \sum_{t=1}^n X_t \stackrel{L^r(\PP)}{\longrightarrow} 0.
   \]
\end{theorem}
In the next section, we apply these moment bounds in order to obtain general central limit theorems.

\subsection{General central limit theory}
\label{sec:main:CLT}

Recall, from Section~\ref{sec:main:reduction}, that if  $n^{\kappa(\gamma,r) - d - 1/\alpha}/G_{\infty,n}'(0) \to 0$ as $n\to\infty$, then Theorem~\ref{thm:approx} yields
\[
  \frac{n^{-d-1/\alpha}}{G_{\infty,n}'(0)} \sum_{t=1}^n
         (G_n(X_t)
         \ - \
         \EE[G_n(X_0)])
         \ - \
         n^{-d-1/\alpha}
         \sum_{t=1}^n X_t \stackrel{\PP}{\longrightarrow} 0
         \qquad\text{as $n\to\infty$.}
\]
In particular, by Slutsky's theorem, the asymptotic behavior of the subordinated long memory linear time series reduces to that of $n^{-d-1/\alpha} \sum_{t=1}^n X_t$. The same conclusion holds if, with the notation of~\eqref{eq:gamma0r0}, $\gamma_G<\gamma_0$ and $n^{\kappa_0+\delta - d - 1/\alpha}/G_{\infty,n}'(0) \to 0$ as $n\to\infty$, by Theorem~\ref{thm:approx_optimal}. This leads to the following central limit theorem.

\begin{corollary}[Central limit theorem]
\label{cor:sub_ord_univariateI_optim}
  Assume that at least one of the following sets of conditions is met:
  \begin{enumerate}
  \item
  \begin{enumerate}[label=(\roman*)]
  \item
  the assumptions of
  Theorem~\ref{thm:approx} hold
  \item
  $\EE[G_n(X_0)]\neq 0$ and $G_{\infty,n}'(0)\neq 0$,
  for sufficiently large $n$
  \item
  there exist $\gamma,r$ satisfying Condition~\eqref{cond:r_gamma} such that $n^{\kappa(\gamma,r) - d - 1/\alpha}/G_{\infty,n}'(0) \to 0$ as $n\to\infty$, where $\kappa(\gamma,r)=1+1/r-(1-d)(1+\gamma)$.
  \end{enumerate}
  \item
  \begin{enumerate}[label=(\roman*)]
  \item the assumptions of Theorem~\ref{thm:approx_optimal} hold
  \item
  $\EE[G_n(X_0)]\neq 0$ and $G_{\infty,n}'(0)\neq 0$,
  for sufficiently large $n$
  \item
  for some $\delta>0$ it holds that
  $n^{\kappa_0 + \delta - d - 1/\alpha}/G_{\infty,n}'(0) \to 0$ as $n\to\infty$, where
  $\kappa_0$ is defined in \eqref{eq:kappa_0}.
  \end{enumerate}
  \end{enumerate}
  Then, as $n\to\infty$,
  \[
  n^{1-(d+1/\alpha)}
    \cdot
    \frac
    {
    \EE[G_n(X_0)]
    }
    {
    G_{\infty,n}'(0)
    }
    \cdot
    \frac
    {1}
    {n}
    \sum_{t=1}^{n}
    \left(
    \frac
    {
    G_n(X_t)
    }
    {
    \EE[G_n(X_0)]
    }
    \ - \
    1
    \right)\  \stackrel{\mathrm{d}}{\longrightarrow} \ Z_\alpha,
  \]
  where $Z_{\alpha}$ is a symmetric $\alpha$-stable random variable defined in Theorem~\ref{thm:clt_partsum}.
\end{corollary}
\begin{remark}[On the rate of convergence in  Corollary~\ref{cor:sub_ord_univariateI_optim}]
\label{rmk:ratesCLT}
Since we consider $G_n$ rather than fixed $G$, the rate of convergence in Corollary~\ref{cor:sub_ord_univariateI_optim} is also governed by the ratio of $\EE[G_n(X_0)]$ and $G_{\infty,n}'(0)$.
This can lead to markedly different rates, which is best illustrated by considering again $G_n(x)=\ind\{x>u_n\}$, where
  \begin{align*}
    \frac
    {\EE[G_n(X_0)]}
    {G_{\infty,n}'(0)}
    \
    =
    \
    \frac
    {\PP[X_0>u_n]}
    {f_{X_0}(u_n)}
    \,.
  \end{align*}
  We compare the case of standard Gaussian innovations with that of $S\alpha S$ innovations with $\alpha \in(1,2)$, covered by Assumption~\ref{asu:f}(i) and (ii), respectively, so that the stationary distribution of $(X_t)$ is Gaussian and $S\alpha S$ with $\alpha \in(1,2)$, respectively. For simplicity, suppose that $\sum_{j=0}^{\infty} a_j^2 = 1$. Then by Mills' ratio it follows in the first case
  \begin{align*}
    \frac
    {\PP[X_0>u_n]}
    {f_{X_0}(u_n)}
    \
    \sim
    \
    \frac
    {1}
    {u_n}
    \,.
  \end{align*}
  This is in line with the intuition that the total rate should slow down when considering only extremes.
  In the second case, however, by Karamata's theorem (see \cite[Theorem~B.1.5]{dehaanExtremeValueTheory2006}) it holds that
  \begin{align*}
    \frac
    {\PP[X_0>u_n]}
    {f_{X_0}(u_n)}
    \
    \sim
    \
    \frac{
    u_n
    }{\nu}
    \,.
  \end{align*}
  This is surprising, because compared with the classical long memory rate $1/n^{1-(d+1/\alpha)}$ the total rate speeds up when looking only at extremes.
\end{remark}
The next section discusses this phenomenon and other examples in detail.

\subsection{Central limit theory for Peaks-over-Threshold estimators (heavy tails)}
\label{sec:main:hvy}

We first discuss the case when the stationary distribution of $(X_t)$ is heavy-tailed with index $\nu>1$. By~\cite[Equation~(15.3.5)]{kulikHeavyTailedTimeSeries2020}, this will be the case if the $\varepsilon_t$ are heavy-tailed with index $\nu>1$, and an application of~\cite[Theorem~1.7.2]{binghamRegularVariation1987} yields that, if $f_{X_0}$ is ultimately decreasing, then $f_{X_0} \in\mathrm{RV}_{-\nu-1}$. In this setting, we have the following consequence of Corollary~\ref{cor:sub_ord_univariateI_optim}; here and throughout $\log(x)_+=\log(x)\lor 0=\log(x) \ind\{ x>1 \}$.

\begin{corollary}[Applications to heavy tails -- deterministic thresholds]
  \label{cor:heavy_det}
  Let Assumptions~\ref{asu:coef} and~\ref{asu:f} hold, and assume that $f_{X_0}\in\mathrm{RV}_{-\nu-1}$ with $\nu>1$. If $u_n\to\infty$ is such that for some $\delta>0$, $u_n=\operatorname{o}(n^{(d+1/(\nu \land 2)-\kappa_0-\delta)/(\nu+1)})$, where $\kappa_0$ is defined in \eqref{eq:kappa_0} then, as $n\to\infty$,
    \[
  n^{1-(d+1/(\nu \land 2))}
    u_n
  \cdot
      \left(
      \frac{ \sum_{t=1}^{n}
      \ind\{X_t>u_n\} }
      {n \PP[X_0>u_n]}
      \ - \ 1
      \right)
      \ \stackrel{\mathrm{d}}{\longrightarrow} \
      \nu
      Z_{\nu\land 2}
     \]
  and
    \[
 n^{1-(d+1/(\nu \land 2))}
    u_n
  \cdot
      \left( \frac
      {
      \sum_{t=1}^{n}
        \log(
        X_t
        /
        u_n
        )_+
      }
      {n \PP[X_0>u_n]}
      \ - \
      \EE[\log(X_0/u_n) \mid X_0>u_n]
      \right)
      \ \stackrel{\mathrm{d}}{\longrightarrow} \
      \frac{\nu}{\nu+1}
      Z_{\nu\land 2}.
    \]
\end{corollary}

In a remark after \cite[Theorem~4.17]{beranLongMemoryProcessesProbabilistic2013} the authors point out that (for a fixed subordination function $G$) a stable limit may be obtained from summation of bounded random variables. This observation extends to our work, where the subordination function depends on the sample size.

\begin{remark}[Equivalent form of Corollary~\ref{cor:heavy_det}]
\label{rmk:equivheavy_det}
An equivalent form of the second convergence in Corollary~\ref{cor:heavy_det}, closer to Corollary~\ref{cor:sub_ord_univariateI_optim}, is
\[
 n^{1-(d+1/(\nu \land 2))}
    u_n
  \cdot
      \left( \frac
      {
      \sum_{t=1}^{n}
        \log(
        X_t
        /
        u_n
        )_+
      }
      {n \EE[\log(X_0/u_n)_+]}
      \ - \
      1
      \right)
      \ \stackrel{\mathrm{d}}{\longrightarrow} \
      \frac{\nu^2}{\nu + 1}
      Z_{\nu\land 2}.
\]
This is due to the well-known convergence $\EE[\log(X_0/u) \mid X_0>u] \to 1/\nu$ as $u\to\infty$.
\end{remark}

In statistical practice, the choice of the threshold $u_n$ is crucial. Typical solutions~\cite{scamcd2012} are data-driven, and it is therefore more common to have random instead of deterministic thresholds.
In theory, one way of shifting from deterministic to random thresholds is to prove results for deterministic thresholds first and, often by means of empirical process theory, extend them to random thresholds.
The Hill estimator (among others)
has enough structure to avoid reliance on empirical process theory, which can be technically delicate. Our approach is based on a derandomization device (Lemma~\ref{lem:derand} in the supplement), specifically avoiding empirical process theory, and proving directly the convergence of the random threshold version based on the joint convergence of the deterministic threshold version and the order statistic $X_{n-k:n}$. In this way, the next result, which is the counterpart of Corollary~\ref{cor:heavy_det} for random thresholds, is obtained by combining Lemma~\ref{lem:derand} and Lemma~\ref{lem:derand:poly} in the supplement.
\begin{corollary}[Application to heavy tails -- random thresholds]
  \label{cor:heavy_rand}
  Let the Assumptions of Corollary~\ref{cor:heavy_det} hold. Then, with $k=\lfloor n \PP[X_0>u_n]  \rfloor$, it holds that, as $n\to\infty$,
\begin{align*}
  &
    n^{1-(d+1/(\nu \land 2))
    }
    q_{X_0}
    \left(
    1
    -
    \frac{k}{n}
    \right)
    \\&
    \cdot
    \
      \left(
      \frac
      {
      \sum_{i=1}^{k}
        \log(
        X_{n-i+1:n}
        /
        X_{n-k:n}
        )
      }
      {n \PP[X_0>u_n]}
      \ - \
      \EE[\log(X_0/u_n) \mid X_0>u_n]
      \right)
      \ \stackrel{\mathrm{d}}{\longrightarrow}\
      \frac{1}{\nu+1} Z_{\nu\land 2}.
  \end{align*}
\end{corollary}

\begin{remark}[Comparison with known results in the literature I]
\label{rmk:rateI}
As in the case of the empirical tail function in the first display of Corollary~\ref{cor:heavy_det}, the Hill estimators with deterministic thresholds (second display of Corollary~\ref{cor:heavy_det}) and with random thresholds (Corollary~\ref{cor:heavy_rand})
show faster speeds compared with classical convergence rates in long memory settings.
In the deterministic threshold case (Corollary~\ref{cor:heavy_det}), the Hill estimator shows the same increased speed as the empirical tail function, namely, with the extra factor $u_n$.
In the random threshold case (Corollary~\ref{cor:heavy_rand}), the factor of increase $q_{X_0}(1-k/n)$ may be seen as the deterministic counterpart of the random threshold $X_{n-k:n}$, since $X_{n-k:n}/q_{X_0}(1-k/n) \to_{\PP} 1$. This follows from the convergence of the empirical tail function in the first display of Corollary \ref{cor:heavy_det} along with inverting this function using Lemma~\ref{lem:koen} in the supplement. It stands in sharp contrast with what is observed in i.i.d.~and short memory settings.
\end{remark}

Although the limiting structure of the Hill estimator is the same for deterministic and random thresholds, the asymptotic scale is lower in the random case. Specifically, it shifts from $\nu/(\nu + 1)$ to $1/(\nu+1)$, where, as $\nu\to\infty$, the former converges to 1, while the latter tends to 0.
This suggests a phase transition (in the sense of \cite{samorodnitskyStochasticProcessesLong2016}) in the case of random thresholds, occurring as we pass from heavy to light tails, that is, in the limit $\nu\to\infty$.
This motivates a second set of examples featuring light tails.

\subsection{Central limit theory for Peaks-over-Threshold estimators (light tails)}
\label{sec:main:light}

In the light-tailed case, the asymptotic behavior of PoT estimators should be expected to be different. For example, one has then $\EE[\log(X_0/u) \mid X_0>u] \to 0$ as $u\to\infty$. This is crucial when it comes to investigating the convergence of the moment estimator of the extreme value index~\cite{dekeinhaa1989}, which is partly built upon the Hill estimator; see~\cite[Section 3.5]{dehaanExtremeValueTheory2006} for a review.

\begin{corollary}[Applications to light tails -- deterministic thresholds]
  \label{cor:light_det}
  Let Assumptions~\ref{asu:coef} and~\ref{asu:f} hold, and suppose that $f_{X_0}(x)= \omega(x) \exp(-|x|^{\beta}/\beta)$ for some $\beta > 0$, where $\omega\in C^1(\RR)$ is
  a symmetric function satisfying $|x|^{\beta-1}\omega(x)\to 0$ as $x\to 0$, $\omega(x)>0$ for $x$ large enough, and with derivative $\omega'$ either regularly varying at infinity or constant equal to 0 for $|x|$ large enough.
  If $u_n\to\infty$ is such that for some $\delta>0$,
  $
  u_n = (-\beta\log(n)(\kappa_0 + \delta - 1/2 -d)+o(1))^{1/\beta}
  $,
  where $\kappa_0$ is defined as in \eqref{eq:kappa_0} with $\alpha=2$, then, as $n\to\infty$,
    \[
  n^{1/2-d}
    u_n^{1-\beta}
  \cdot
      \left(
      \frac{ \sum_{t=1}^{n}
      \ind\{X_t>u_n\} }
      {n \PP[X_0>u_n]}
      \ - \ 1
      \right)
      \
      \stackrel{\mathrm{d}}{\longrightarrow}
      \
      Z_{2}
     \]
  and
    \[
  n^{1/2-d}
    u_n^{1-\beta}
  \cdot
      \left( \frac
      {
      \sum_{t=1}^{n}
        \log(
        X_t
        /
        u_n
        )_+
      }
      {n \PP[X_0>u_n]}
      \ - \
      \EE[\log(X_0/u_n) \mid X_0>u_n]
      \right)
      \
      \stackrel{\mathrm{d}}{\longrightarrow}
      \
      Z_{2}.
    \]
\end{corollary}
 The factor $u_n^{1-\beta}$ allows for interpolation between slowly decaying (e.g., polynomial) and rapidly decaying (e.g., exponential) tails, depending on $\beta$.
    For instance, $\beta=1$ corresponds to Laplace-type tails, while $\beta=2$ gives Gaussian decay.
    The assumptions of Corollary~\ref{cor:light_det} can be verified easily
    for normally distributed innovations, that is, where $\beta=2$.
   We highlight this special case in the following corollary.
\begin{corollary}[Gaussian innovations -- deterministic thresholds]
  \label{cor:gaussian}
  Let the $\varepsilon_t$ be i.i.d. Gaussian centered, and suppose that Assumption~\ref{asu:coef} holds.
  If $u_n\to\infty$ is such that for some $\delta>0$, $
  u_n = (-2\log(n)(\kappa_0 + \delta - 1/2 -d)+o(1))^{1/2}
  $,
  where $\kappa_0$ is defined as in \eqref{eq:kappa_0} with $\alpha = 2$, then, as $n\to\infty$,
    \[
  \frac{n^{1/2-d}}{u_n}
  \cdot
      \left(
      \frac{ \sum_{t=1}^{n}
      \ind\{X_t>u_n\} }
      {n \PP[X_0>u_n]}
      \ - \ 1
      \right)
      \
      \stackrel{\mathrm{d}}{\longrightarrow}
      \
      Z_{2}
     \]
  and
    \[
  \frac{n^{1/2-d}}{u_n}
  \cdot
      \left( \frac
      {
      \sum_{t=1}^{n}
        \log(
        X_t
        /
        u_n
        )_+
      }
      {n \PP[X_0>u_n]}
      \ - \
      \EE[\log(X_0/u_n) \mid X_0>u_n]
      \right)
      \
      \stackrel{\mathrm{d}}{\longrightarrow}
      \
      Z_{2}.
    \]
\end{corollary}

\begin{remark}[Comparison with known results in the literature II]
\label{rmk:rateII}
Another comparison with the i.i.d.~setting is worthwhile here. On the one hand, when the $X_i$ are i.i.d.~with density satisfying $f_{X_0}(x)= \omega(x) \exp(-|x|^{\beta}/\beta)$ for some $\beta>0$ and a function $\omega$ converging to a positive constant $\omega_0$ at infinity, then a repeated use of Lemma~1 in~\cite{stupflerSimpleSufficientCriteria2025} and a straightforward application of the Lindeberg central limit theorem yield
   \[
   \sqrt{n \PP[X_0>u_n]} u_n^{\beta} \left( \frac
     {
     \sum_{t=1}^{n}
       \log(
       X_t
       /
       u_n
       )_+
     }
     {n \PP[X_0>u_n]}
     \ - \
     \EE[\log(X_0/u_n) \mid X_0>u_n]
     \right)
     \stackrel{\mathrm{d}}{\longrightarrow}
     \mathcal{N}(0,\sigma^2)
   \]
   when $n\PP[X_0>u_n]\to \infty$, where $\sigma^2=\sigma^2(\beta,\omega_0)$ can be explicitly calculated. On the other hand, when the $X_i$ are i.i.d.~heavy-tailed with $f_{X_0}\in\mathrm{RV}_{-\nu-1}$, then another straightforward application of the Lindeberg central limit theorem yields
   \[
   \sqrt{n \PP[X_0>u_n]} \left( \frac
     {
     \sum_{t=1}^{n}
       \log(
       X_t
       /
       u_n
       )_+
     }
     {n \PP[X_0>u_n]}
     \ - \
     \EE[\log(X_0/u_n) \mid X_0>u_n]
     \right)
     \stackrel{\mathrm{d}}{\longrightarrow}
     \mathcal{N}(0,2/\nu^2)
\]
when $n\PP[X_0>u_n]\to \infty$. In other words, when the $X_i$ are i.i.d., the speed of convergence goes from
\[
\sqrt{n \PP[X_0>u_n]} \sim \sqrt{n u_n^{-1/\nu}}
\]
in the heavy-tailed setting to
\[
\sqrt{n \PP[X_0>u_n]} u_n^{\beta} \sim \sqrt{n u_n^{1+\beta} \exp(-u_n^{\beta}/\beta)}
\]
in the light-tailed setting. The exponentially decaying term $\exp(-u_n^{\beta}/\beta)$ dramatically slows down the speed as one goes from heavy to light tails. By contrast, in the long memory case, the speed of convergence goes from $n^{1/2-d} u_n$ in the heavy-tailed setting (when $\nu>2$) to $n^{1/2-d} u_n^{1-\beta}$ in the light-tailed case. The slowdown is thus much more pronounced in the i.i.d.~setting than in the long memory case.
\end{remark}

As in Section~\ref{sec:main:hvy},
we finally consider
an analogue of Corollary~\ref{cor:heavy_rand} in the light-tailed context.

\begin{corollary}[Application to light tails -- random thresholds]
\label{cor:light_rand}
 Under the Assumptions of Corollary~\ref{cor:light_det} and with the notation of Corollary~\ref{cor:heavy_rand} it holds, as $n\to\infty$,
 \begin{align*}
  &
    n^{1/2-d}
    \left(
    q_{X_0}
    \left(
    1
    -
    \frac{k}{n}
    \right)
    \right)^{1-\beta}
    \\&
    \cdot
    \
      \left(
      \frac
      {
      \sum_{i=1}^{k}
        \log(
        X_{n-i+1:n}
        /
        X_{n-k:n}
        )
      }
      {n \PP[X_0>u_n]}
      \ - \
      \EE[\log(X_0/u_n) \mid X_0>u_n]
      \right)
      \
      \stackrel{\PP}{\longrightarrow}
      \
      0\,.
  \end{align*}
\end{corollary}
The analogous result for heavy-tails (Corollary~\ref{cor:heavy_rand}) suggests that there is a phase transition in the limit between heavy- and light tails.
Observe that the speed coming from deterministic thresholds and light tails (Corollary~\ref{cor:light_det}) is too slow to identify a non-trivial limiting structure for light tails and random thresholds (Corollary~\ref{cor:light_rand}).
This hints at a weakening or even absence of a long memory effect in this case.
Absence of long memory has indeed been observed in a similar setting by \cite{kulikTailEmpiricalProcess2011a},
where the authors recovered i.i.d.~behavior with the standard speed $\sqrt{k}$.

\section{Numerical illustrations}
\label{sec:simstudy}
\subsection{Main objective and setting}
\label{sec:simstudy:setting}

The purpose of this section is to illustrate the central limit theorem in Corollary~\ref{cor:sub_ord_univariateI_optim} when the innovations are heavy-tailed. Our motivation is to get a grasp of the extent to which the asymptotic distributions of the PoT estimators obtained in Section~\ref{sec:main:hvy} show up in finite samples. To the best of our knowledge, there is no exact method for simulating long memory linear time series with heavy tails. We therefore use a (long enough) truncated coefficient sequence,
\[
a_0 = 0 \qquad \mbox{and} \qquad a_j \ = \  j^{-(1-d)}
\,,
\qquad
 j = 1, \ldots, p=10^8
 \,,
\]
with $d = 0.1$, and symmetric $\alpha$-stable innovations $\varepsilon_t$ with $\alpha = 1.9$ and scale parameter $\eta = 1$. We simulate $N=10{,}000$ i.i.d.~copies of $(X_t,t=1, \ldots, n=p/10=10^7)$, using GPU acceleration; the sample size $n=p/10$ is such that the number of coefficients $p$ is large enough compared with the length of the time series, in order to simulate long memory behavior of the series $X_t=\sum_{j=1}^{\infty} j^{-(1-d)} \varepsilon_{t-j}$.

\subsection{Slow convergence of pre-asymptotic scaling factors}
\label{sec:simstudy:scale}

The main ingredient in our central limit theory is Theorem~\ref{thm:clt_partsum}, which provides the weak convergence behavior of partial sums of the $X_t$, namely,
\begin{align*}
     n^{-d-1/\alpha}
    \sum_{t=1}^n
    X_t
    \ \stackrel{\mathrm{d}}{\longrightarrow} \
    Z_\alpha
    \qquad\text{as}\ n\to\infty,
\end{align*}
where $Z_\alpha$ is symmetric $\alpha$-stable with scale
\[
\eta
  \ := \
\frac{c_a}{d}
    \cdot
    \left(
    A
    \cdot
      \frac
      {\Gamma(2-\alpha)\cos(\pi\alpha/2)}
      {1-\alpha}
      \right)^{1/\alpha}
    \left(
    \int_{-\infty}^{1}
    \left(
    (1-v)^{d}_+
    -
    (-v)^d_+
    \right)^{\alpha}
    \,\mathrm{d}v
    \right)^{1/\alpha}
      \,.
\]
Observe that in our case $c_a=1$ so that, by symmetry of the $\alpha$-stable innovations,
\[
\eta = \frac{1}{d}
    \left( \int_{-\infty}^{1}
    \left(
    (1-v)^{d}_+
    -
    (-v)^d_+
    \right)^{\alpha}
    \,\mathrm{d}v
    \right)^{1/\alpha}
      \,.
\]
For $d=0.1$, $\alpha=1.9$, by numerical integration we get
\begin{align*}
 \eta
 \ = \
 \eta_{d=0.1, \alpha = 1.9}
 \ = \
   10 \left(
   \int_{-\infty}^{1}
   \left(
   (1-v)^{0.1}_+
   -
   (-v)^{0.1}_+
   \right)^{1.9}
   \,\mathrm{d}v
   \right)^{1/1.9}
   \
   \approx
   \
   9.344
   \,.
\end{align*}
To assess how fast the asymptotic scale $\eta$ is approached in finite samples, one may write the rescaled partial sum as
\begin{align*}
 n^{-d-1/\alpha}
 \frac{1}{\eta}
 \sum_{t=1}^n
 X_t
 \ &= \
 n^{-d-1/\alpha}
 \frac{1}{\eta}
 \sum_{t=1}^n
 \sum_{j=1}^{\infty}
 j^{-(1-d)}
 \varepsilon_{t-j}
 \\
 \ &= \
 n^{-d-1/\alpha}
 \frac{1}{\eta}
 \left(
 \sum_{t=1}^{n-1}
 \varepsilon_{n-t}
 \sum_{j = 1}^t
 j^{-(1-d)}
 \ + \
 \sum_{t=0}^\infty
 \varepsilon_{-t}
 \sum_{j=1}^{n}
 (j+t)^{-(1-d)}
 \right)
    \,.
\end{align*}
Since the $\varepsilon_t$ are i.i.d.~symmetric $\alpha$-stable, so is this rescaled partial sum, with scale
\begin{align} \label{eq:scale-partialsums}
 s=s_{n,d,\alpha} = n^{-d-1/\alpha}
 \frac{1}{\eta}
 \left(
 \sum_{t=1}^{n-1}
 \left(
 \sum_{j = 1}^{ t}
 j^{-(1-d)}
 \right)^{\alpha}
 \ + \
 \sum_{t=0}^\infty
 \left(
 \sum_{j=1}^{n}
 (j+t)^{-(1-d)}
 \right)^{\alpha}
 \right)^{1/\alpha}
 \,.
\end{align}
If $\eta$ were to be a correct approximation of the scale of $n^{-d-1/\alpha} \sum_{t=1}^n X_t$, then $s_{n,d,\alpha}$ would be close to 1. Numerical approximations, however, show that the ratio $s_{n,d,\alpha}$ converges to 1 slowly, suggesting that $\eta$
does not
correctly approximate the scale of $n^{-d-1/\alpha} \sum_{t=1}^n X_t$. We illustrate this in the top panel of Figure~\ref{fig:scale}, where even for the large value $n=10^6$ one finds $s\approx 0.72$.

\begin{figure}[ht]
 \centering
 \includegraphics[scale=0.4]{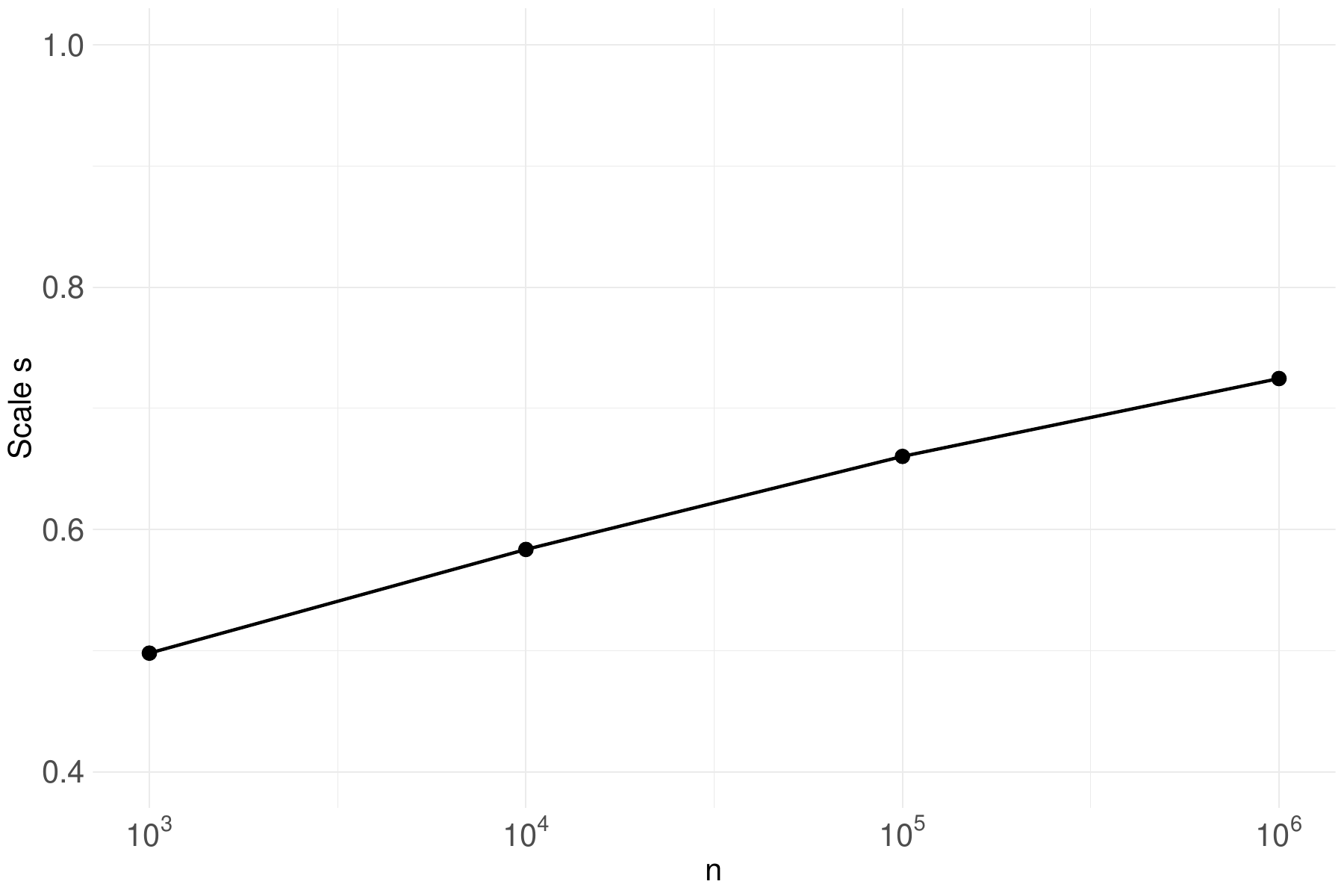} \\
 \includegraphics[scale=0.4]{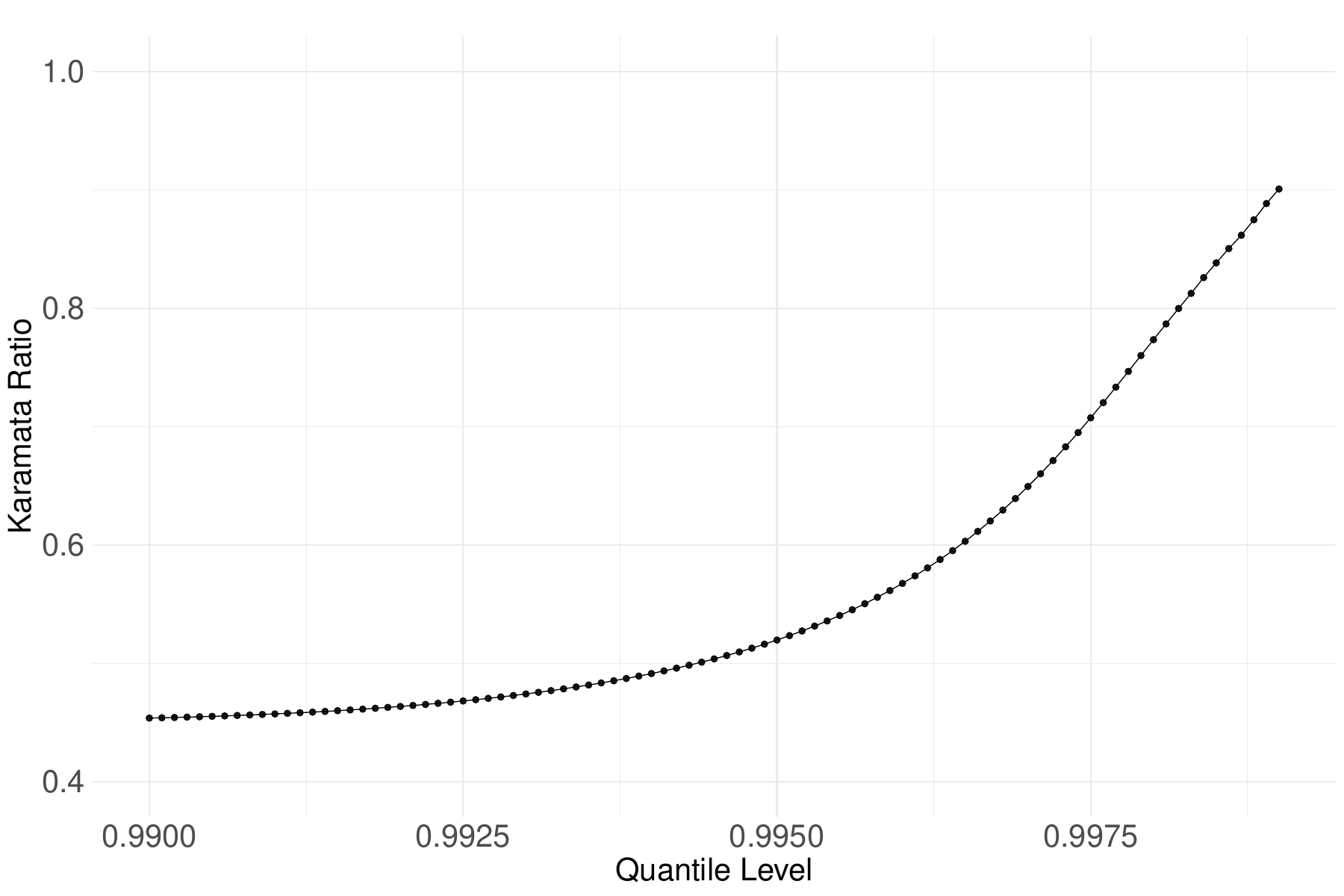}
 \caption{Illustration of the convergence of pre-asymptotic scaling factors. Top: Values of $s_{n,d,\alpha}$ in \eqref{eq:scale-partialsums} for $n\in \{10^3, 10^4, 10^5, 10^6\}$, $d=0.1$ and $\alpha=1.9$.
 Bottom: Values of $\alpha \PP[X>u]/ (u f_{X}(u))$ for the unit $\alpha$-stable distribution along the levels $u=q_X(\tau)$ for $\tau\in [0.99,0.999]$, with $\alpha=1.9$.}
 \label{fig:scale}
\end{figure}

Another scaling quantity, namely $\EE[G_n(X_0)]/G_{\infty,n}'(0)$, appears in Corollary~\ref{cor:sub_ord_univariateI_optim}. For exceedance counts, using the indicator function $G_n(x)=\ind\{ x>u_n \}$, this ratio is equal to $\PP[X_0>u_n]/f_{X_0}(u_n)$, which is known given the distribution of $X_0$. But even for relatively simple extensions such as $G_n(x)=(\log(x)-\log(u_n)) \ind\{ x>u_n \}$, the exact calculation of $\EE[G_n(X_0)]$ and $G_{\infty,n}'(0)$ becomes an issue. This is because the resulting integrals typically do not have a simple closed form. A solution is an approximation by Monte Carlo simulation, using for example
\[
  \EE[G_n(X_0)]
  \ = \
  \EE[\log(X_0/u_n)_+]
  \
  \approx
  \
  \widehat{\EE}[\log(X_0/u_n)_+]
  \ = \
  \frac{1}{m}
  \sum_{i=1}^m
  \log(X^{(i)}_0/u_n)\ind\{X^{(i)}_0>u_n\}
\]
and, by partial integration (see the proof of Corollary~\ref{cor:heavy_det}),
\begin{align*}
  G_{\infty,n}'(0)
  \ = \
  \int_{u_n}^\infty
    \frac{
    f_{X_0}(x)
    }{x}
    \,\mathrm{d}x
  \ = \
  \EE[X_0^{-1}\ind\{X_0>u_n\}]
  & \
  \approx
  \
  \widehat{\EE}[X_0^{-1}\ind\{X_0>u_n\}]
  \\
  & \ = \
  \frac{1}{m}
  \sum_{i=1}^m
  \frac{1}{X^{(i)}_0}
  \ind\{X^{(i)}_0>u_n\}
  \,,
\end{align*}
where $X_0',\ldots,X_0^{(m)}$, $m\in\NN$ are i.i.d.~samples of $X_0$; here we take $m=10^7$.

It should be highlighted that in statistical applications, the distribution of $X_0$ is not assumed to be known. A workaround is to combine Karamata's and Slutsky's theorem, as is done in the proof of Corollary~\ref{cor:heavy_det}, to replace $\EE[G_n(X_0)]/G_{\infty,n}'(0)$ by an asymptotic equivalent. For example, for $G_n(x)=\ind\{ x>u_n \}$,
\[
\frac{\EE[G_n(X_0)]}{G_{\infty,n}'(0)} = \frac{\PP[X_0>u_n]}{f_{X_0}(u_n)} \ \mbox{ and } \ \frac{u_n}{\nu}
\]
are asymptotically equivalent. As illustrated in the bottom panel of Figure~\ref{fig:scale}, when $X_0$ has a stable distribution these two quantities may, however, differ substantially, even at fairly high thresholds $u_n$ of order, say, $q_{X_0}(0.999)$. This is largely due to the speed of convergence of $\PP[X_0>u_n]$ to its power law equivalent, where the remainder term is of order $O(u_n^{-\alpha})$ and
vanishes slowly as $n\to\infty$;
see~\cite[Corollary 2 p.94]{zolotarev1986one} for details. It is thus not advisable to use the asymptotic equivalent $u_n/\nu$ in place of $\PP[X_0>u_n]/f_{X_0}(u_n)$ when $G_n(x)=\ind\{ x>u_n \}$. More accurate asymptotic equivalents could be obtained by estimating the remainder term in Karamata's theorem, for instance using second-order extended regular variation~\cite[Sections 2.3 and B.3]{dehaanExtremeValueTheory2006}. The construction of such refined asymptotic equivalents is outside of the scope of this work.

This discussion motivates, in our particular case of stable distributed $X_0$, comparing the finite-sample distributions of
\begin{equation}
\label{eqn:PoT_rescaled}
n^{1-(d+1/\alpha)}
\frac{1}{\eta}
    \frac{\PP[X_0>u_n]}{f_{X_0}(u_n)}
  \cdot
      \left(
      \frac{ \sum_{t=1}^{n}
      \ind\{X_t>u_n\} }
      {n \PP[X_0>u_n]}
      \ - \ 1
      \right)
\end{equation}
and
\begin{equation}
\label{eqn:Hill_rescaled}
n^{1-(d+1/\alpha)}
\frac{1}{\eta}
    \frac{
    \widehat{\EE}[\log(X_0/u_n)_{+}]
    }{
    \widehat{\EE}[X_0^{-1}\ind\{X_0>u_n\}]
    }
  \cdot
      \left( \frac
      {
      \sum_{t=1}^{n}
        \log(
        X_t
        /
        u_n
        )_+
      }
      {n \EE[ \log(
        X_0
        /
        u_n
        )_+ ]}
      \ - \
      1
      \right)
\end{equation}
with the symmetric unit stable distribution with index $\alpha$.

\subsection{Histograms of exceedance counts and Hill estimators}
\label{sec:simstudy:histograms}
The top rows of Figures~\ref{fig:PoT} and~\ref{fig:Hill} report histograms of $N=10{,}000$ copies from ~\eqref{eqn:PoT_rescaled} and~\eqref{eqn:Hill_rescaled}, respectively, when the length $n$ of the time series $(X_t)$ belongs to $\{10^3, 10^5, 10^7\}$. It can be seen there that the structure of the histograms roughly matches that of a symmetric unit stable distribution, only with incorrect scale. From the discussion in Section~\ref{sec:simstudy:scale}, it follows that a possible reason for that is incorrect scaling of the partial sums; the middle rows of Figures~\ref{fig:PoT} and~\ref{fig:Hill} show the histograms of these same realizations divided by the scale parameter $\widehat{s}$ of a stable distribution fitted to the realizations of the rescaled partial sums $n^{-d-1/\alpha} \eta^{-1} \sum_{t=1}^n X_t$. This fit and subsequent stable fits are obtained using the {\tt stableFit} function from the {\tt fBasics} package in {\tt R}, with argument {\tt type="q"}, based on the method proposed by~\cite{mcculloch1986}. For exceedance counts (Figure~\ref{fig:PoT}), this scaling results in a much better fit of the symmetric unit stable distribution, with only minor deviations for $n=10^3$. For the Hill estimator (Figure~\ref{fig:Hill}),  this still does not result in a good fit for all values of $n$, and
lack of symmetry due to skewness can be observed.
To separate scaling and skewness issues, we report in the bottom row of Figures~\ref{fig:PoT} and~\ref{fig:Hill} the histograms of the realizations of~\eqref{eqn:PoT_rescaled} and~\eqref{eqn:Hill_rescaled}, rescaled not by $\widehat{s}$, but by a scale parameter $\widetilde{s}$ obtained by fitting a stable distribution directly to the realizations. In the case of the Hill estimator
we confirm the presence of skewness, and while its degree seems to decrease from $n=10^3$ (bottom row of Figure~\ref{fig:Hill}, left panel) to $n=10^5$ (bottom row of Figure~\ref{fig:Hill}, middle panel), this is not the case from $n=10^5$ to $n=10^7$ (bottom row of Figure~\ref{fig:Hill}, right panel). As a conclusion, it seems that, in the simplest case of exceedance counts, the finite-sample distribution is dictated by the behavior of partial sums of $X_t$, while in more difficult cases such as the Hill estimator, a degree of skewness appears, which could be due to the particular structure of the subordinating function $G_n$.

\begin{figure}[t]
\centering
\vspace{-70pt}
\includegraphics[scale=0.36]{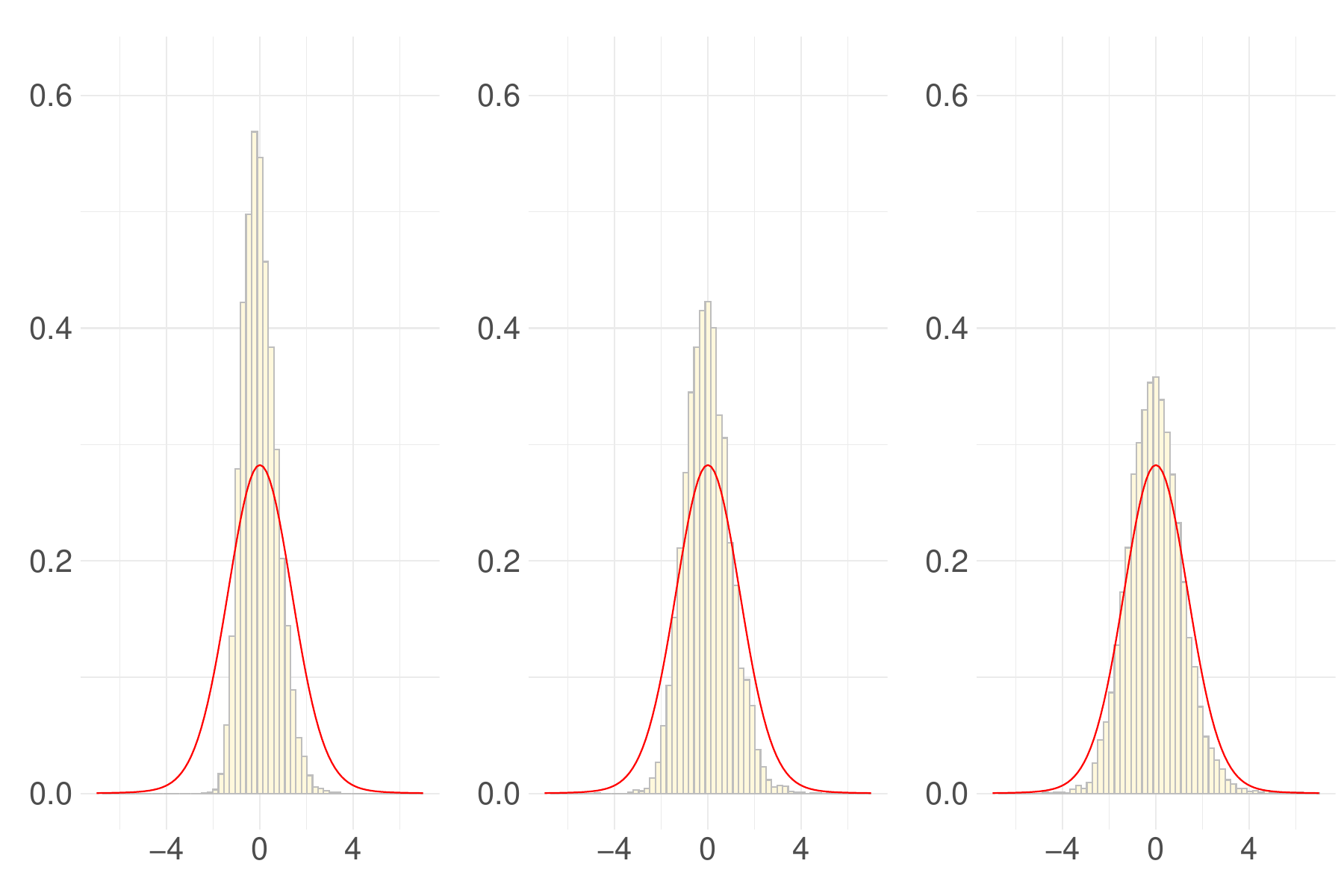} \\
\includegraphics[scale=0.36]{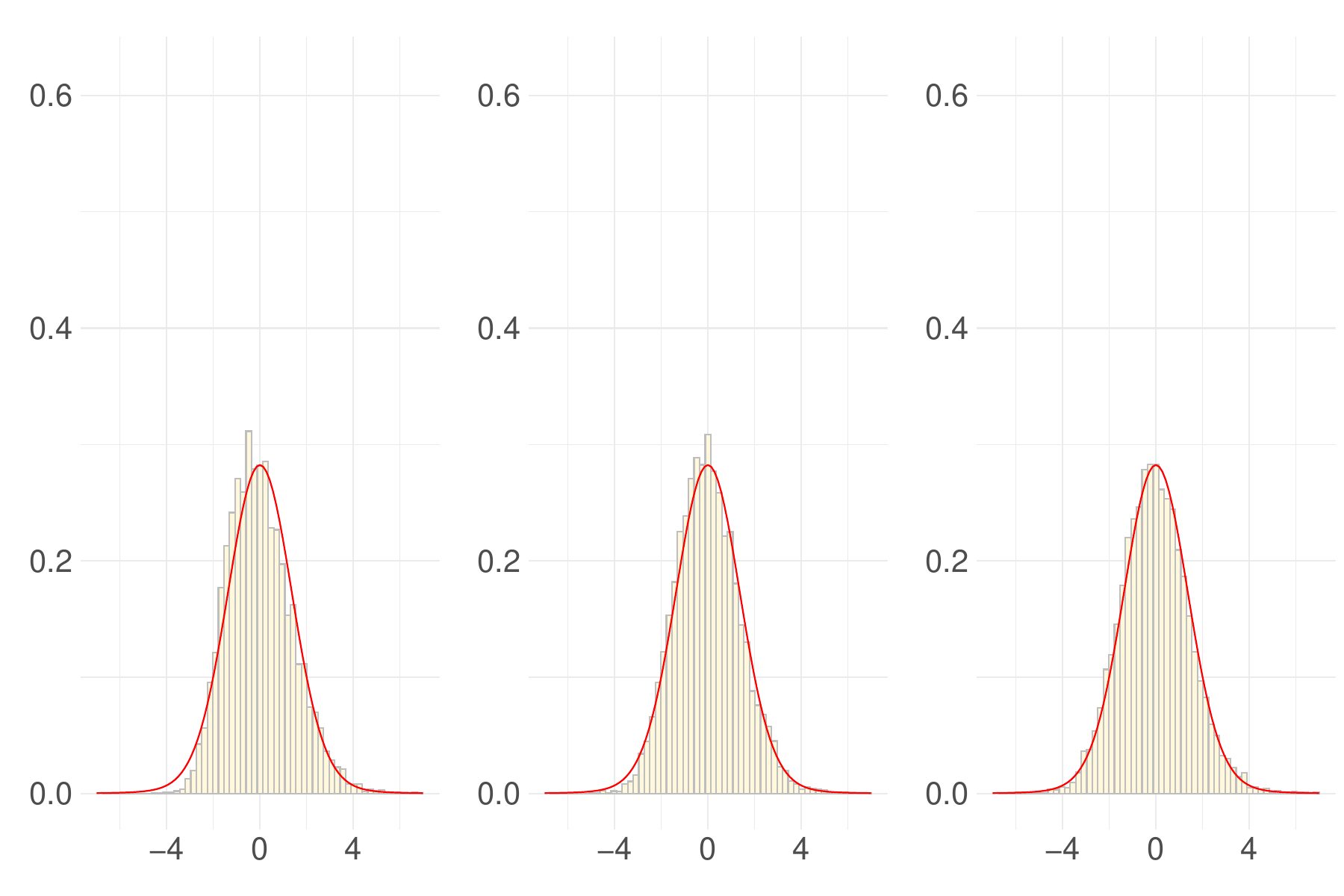} \\
\includegraphics[scale=0.36]{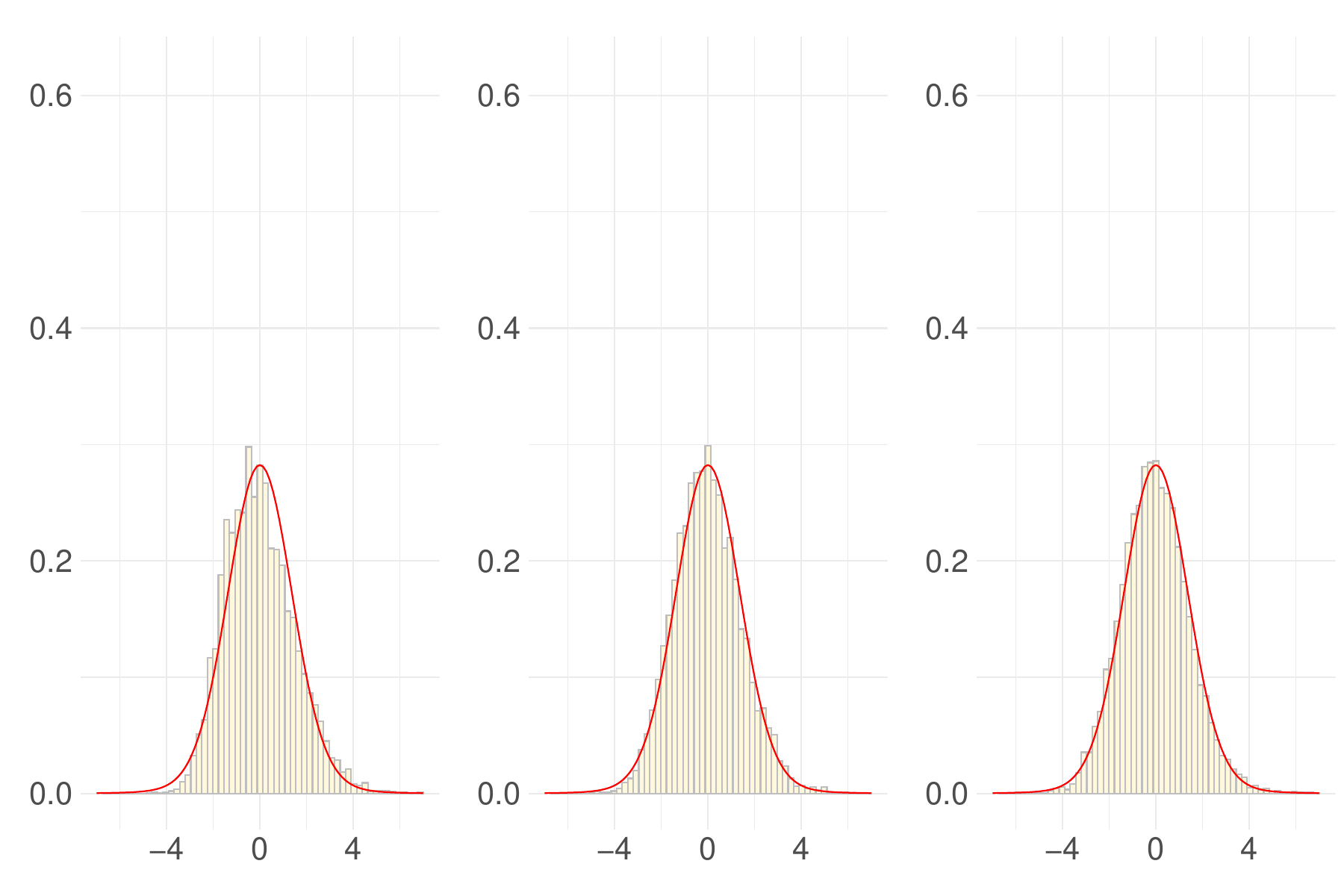}
\caption{Histograms of $N=10{,}000$ realizations of the centered and rescaled exceedance counts in~\eqref{eqn:PoT_rescaled}. Top panels: raw realizations, middle panels: realizations rescaled by the parameter $\widehat{s}$ (of a stable distribution fitted to the partial sums), bottom panels: realizations rescaled by the parameter $\widetilde{s}$ (of a stable distribution fitted directly to these raw realizations). The length of the time series $(X_t)$, generated with $a_j=j^{-(1-d)}$ and symmetric unit $\alpha$-stable innovations (for $d=0.1$ and $\alpha=1.9$), is $n=10^k$ with $k=3$ (left panels), $5$ (middle panels), $7$ (right panels), with $u_n=1-n^{-1/5} = 1.38$, $2.67$, $3.74$.}
\label{fig:PoT}
\end{figure}

\begin{figure}[ht]
\centering
\vspace{-70pt}
\includegraphics[scale=0.36]{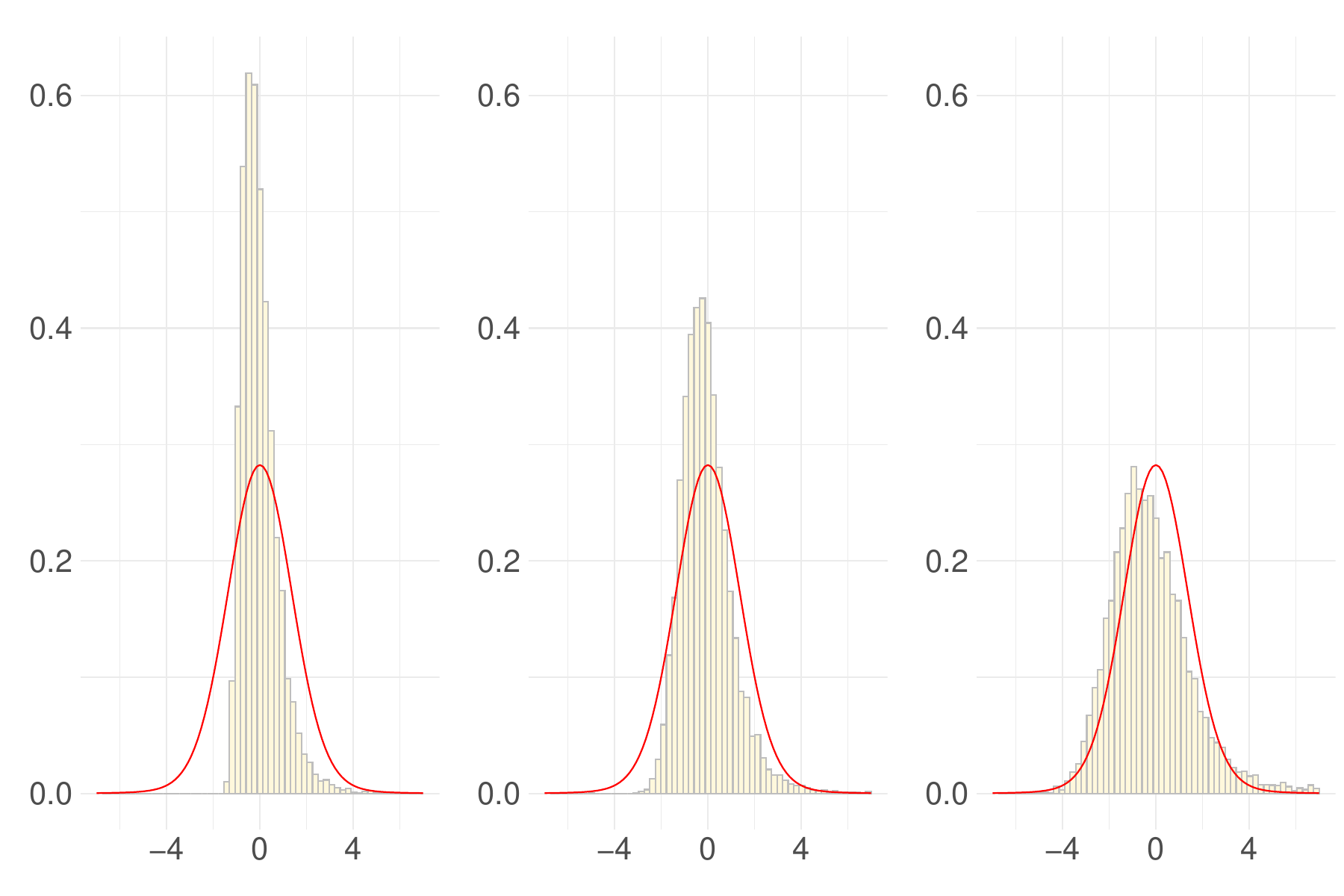} \\
\includegraphics[scale=0.36]{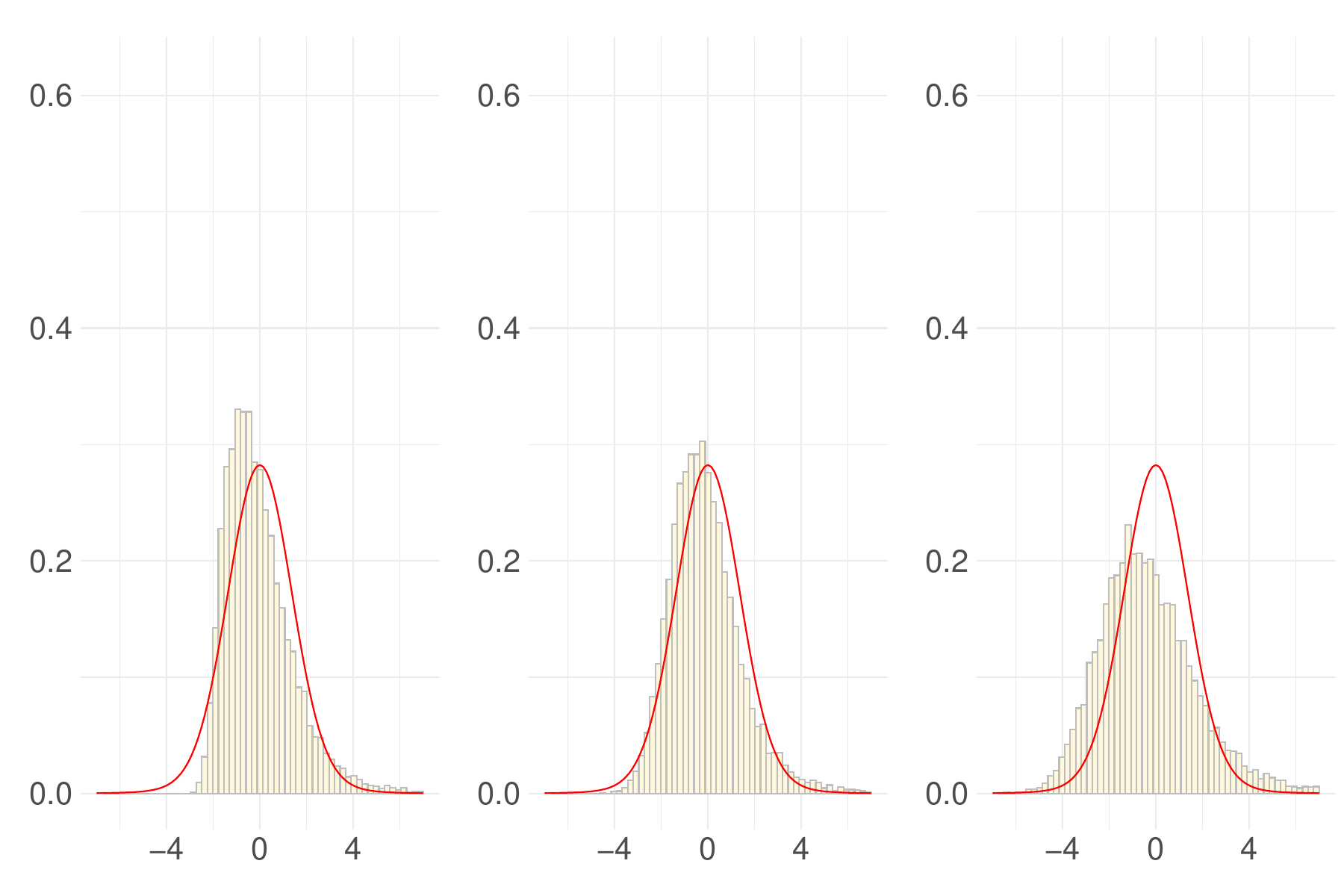} \\
\includegraphics[scale=0.36]{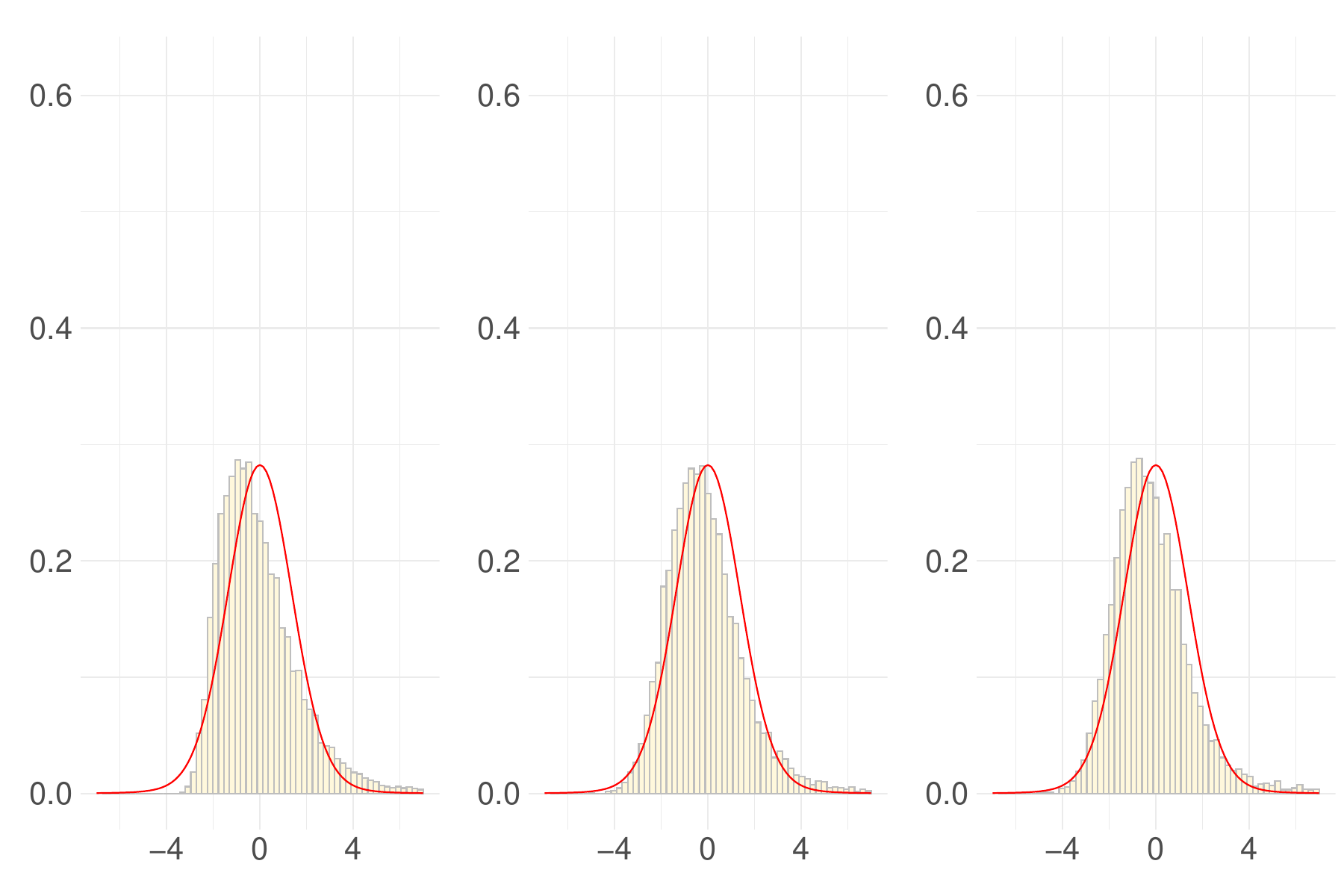}
\caption{Histograms of $N=10{,}000$ realizations of the centered and rescaled Hill estimator in~\eqref{eqn:Hill_rescaled}. Top panels: raw realizations, middle panels: realizations rescaled by the parameter $\widehat{s}$ (of a stable distribution fitted to the partial sums), bottom panels: realizations rescaled by the parameter $\widetilde{s}$ (of a stable distribution fitted directly to these raw realizations). The length of the time series $(X_t)$, generated with $a_j=j^{-(1-d)}$ and symmetric unit $\alpha$-stable innovations (for $d=0.1$ and $\alpha=1.9$), is $n=10^k$ with $k=3$ (left panels), $5$ (middle panels), $7$ (right panels), with $u_n=1-n^{-1/5} = 1.38$, $2.67$, $3.74$.}
\label{fig:Hill}
\end{figure}

\section{Discussion}
\label{sec:discussion}
Our work yields a variety of unexpected results, already highlighted in the previous sections. We now discuss these findings, offer interpretations, and place them in a broader context.
We highlighted that the ratio $\EE[G_n(X_0)]/G'_{\infty,n}(0)$, which appears in the rate of the central limit theorem Corollary~\ref{cor:sub_ord_univariateI_optim}, is widely
different in the heavy- and light-tailed cases when $G_n(x)=\ind\{x>u_n\}$. Subsequently, we showed that this applies also to the Hill estimator, where $G_n(x)=\log(x/u_n)\ind\{x>u_n\}$.
This raises the question of how the heavy- and light-tailed setting can differ so markedly in this regard, and how the increased rate in the heavy-tailed case arises.
We offer an interpretation based on the observation that heavy tails exhibit strong asymptotic dependence, leading to extremal clustering, while light tails, such as those of the normal distribution, exhibit asymptotic independence.
To be more precise, we define asymptotic (in)dependence by
      \begin{align*}
        \PP[X_t>u\mid X_0>u]
        \ \overset{ u \to\infty}{\to} \
        \chi(t)
        \begin{cases}
          \ = \ 0
          \qquad\text{for all $t>0$ when asymptotic independent,}\\
          \ > \ 0
          \qquad\text{for some $t>0$ when asymptotic dependent.}\\
        \end{cases}
      \end{align*}
      The quantity $\chi(t)$ is called tail-dependence coefficient. It is linked to the tail process $(Y_t)$ of $(X_t)$ (see \cite{basrakRegularlyVaryingMultivariate2009} for a reference to tail processes) via the identity $\chi(t)=\PP[Y_t>1]$.
Let us first discuss the asymptotically dependent, that is, heavy-tailed setting.
To control extremal clustering, all PoT results in the literature, to the best of our knowledge, assume some form of anti-clustering condition.
For example, one assumes that for an intermediate rate $(r_n)$, a threshold sequence $(u_n)$, and all $s,t>0$ it holds
  \begin{align*}
    \lim_{m\to\infty}
    \limsup_{n\to\infty}
    \frac
    {1}
    {\PP[|X_0|>u_n]}
    \sum_{j=m}^{r_n}
    \PP[|X_0|>u_ns, |X_j|>u_nt]
    \ = \
    0
    \,.
  \end{align*}
  This condition is called
  $\mathcal{S}(r_n,u_n)$
  in \cite[Section~9.2 p.~243]{kulikHeavyTailedTimeSeries2020}.
  According to (15.3.33) therein,
  in the case of heavy-tailed linear time series there exists a sufficient condition for $S(r_n,u_n)$ to hold, that is,
  \begin{align}
    \label{eq:disc:1}
    \sum_{j=1}^{\infty}
    j\cdot |a_j|^{\delta}
    \ < \
    \infty
    \qquad \text{for some}\
    \delta \in (0,\alpha)\cap (0,2]
    \,.
  \end{align}
  It, however, follows immediately that under Assumption~\ref{asu:coef} this simpler condition is never satisfied.
  We even have that
  a
  sufficient condition for $\mathcal{S}(r_n,u_n)$ to fail is that $\sum_{t=1}^{\infty}\PP[|Y_t|>1]=\infty$, where $(Y_t)$ is the tail process of the time series. This follows from applying Fatou's lemma, that is,
  \begin{align*}
    \sum_{t=1}^{\infty}
    \PP[|Y_t|>1]
    \ \le \
    \liminf_{n\to\infty}
    \frac
    {1}
    {\PP[|X_0|>u_n]}
    \sum_{j=1}^{r_n}
    \PP[|X_0|>u_n, |X_j|>u_n]
    \,.
  \end{align*}
  With \cite[(15.3.34)]{kulikHeavyTailedTimeSeries2020}, if moreover the $|a_j|$ are decreasing, then $\sum_{t=1}^{\infty} \PP[|Y_t|>1]<\infty$
  is equivalent to
  \begin{align*}
    \sum_{j=1}^{\infty}
    j\cdot
    |a_j|^{\nu}
    \ < \ \infty
    \quad\text{which under Assumption~\ref{asu:coef} becomes}\quad
    \sum_{j=1}^{\infty}
    j^{1-(1-d)\nu}
    \ < \ \infty
    \,.
  \end{align*}
  Since this holds if and only if $\nu\ge 2/(1-d)$,
we know that the anti-clustering condition fails for such combinations of $(\nu,d)$. In particular, it fails for all $\nu\in(1,2)$.
It therefore seems plausible that long clusters of extremes may accelerate the speed of convergence obtained in Corollary~\ref{cor:sub_ord_univariateI_optim}.

Having discussed the asymptotically dependent setting and its connection to anti-clustering conditions, we now turn to the asymptotically independent case featured in our work (Corollary~\ref{cor:gaussian}).
In the case of Gaussian innovations, and thus Gaussian time series, we observe behavior more consistent with existing literature: for deterministic thresholds, focusing on extremes, that is, reducing the effective number of summands, slows down the speed of convergence.
Since the time series is asymptotically independent, anti-clustering conditions hold, and we return to the standard setting.

Another point we want to emphasize is the sensitivity of PoT estimators to threshold choice. The assumption $u_n=\operatorname{o}(n^{(d+1/(\nu \land 2)-\kappa_0-\delta)/(\nu+1)})$ in Corollary~\ref{cor:heavy_det}, where $\kappa_0$ is defined in \eqref{eq:kappa_0}, reflects that the threshold $u_n$ should be chosen carefully. In simulations not reported here, we tried using higher quantile levels such as $q_{X_0}(1-n^{-1/3.5})$, and we observed that this could make the finite-sample distribution even less close to the asymptotic limit. As a consequence, one should not blindly base inference procedures upon the asymptotic distribution obtained in our results. Similar problems have been reported already in i.i.d.~extreme value settings irrespective of whether the second moment is finite or not~\cite{daostuuss2024,daostuuss2025,padstu2022}. A solution may be to develop dedicated self-normalization or subsampling procedures, as in~\cite{baitaq2017,jacmcelroypol2012,romwol1999}, or to rely on a smaller, more tractable class of models such as ARFIMA models~\cite{verstoche2025}.
On a different note, but still about threshold choice, observe that, unlike in the i.i.d.~or weakly dependent case, results for deterministic and random thresholds differ in our long memory setting.
While \cite{kulikTailEmpiricalProcess2011a} show that, in an asymptotically independent setting, long memory effects for the Hill estimator in stochastic volatility models arise only with deterministic thresholds, our work allows for a comparison of both asymptotically dependent and independent settings (Corollaries~\ref{cor:heavy_rand} and \ref{cor:light_rand}).
 We note that stochastic volatility models with asymptotic dependence exist~\cite{janssenStochasticVolatilityModel2016}; see also \cite[p.487]{kulikHeavyTailedTimeSeries2020} for a discussion of the literature. Future work could take these models into consideration.

\section*{Acknowledgments}

Financial support from the French CNRS within the project ``Extreme value analysis of time series through de-randomization techniques'', funded through the IEA program, is gratefully acknowledged. The authors also acknowledge computing support by the state of Baden-W\"urttemberg through bwHPC. G.~Stupfler acknowledges further financial support from the French {\it Agence Nationale de la Recherche} under the grants ANR-23-CE40-0009 (EXSTA project) and ANR-11-LABX-0020-01 (Centre Henri Lebesgue), as well as from the TSE-HEC ACPR Chair ``Regulation and systemic risks'' and the Chair Stress Test, RISK Management and Financial Steering of the Foundation Ecole Polytechnique. I.~Scheffel thanks Julian Weitz who pointed him towards the use of GPU acceleration for simulating linear time series.

\clearpage

\appendix

\renewcommand{\thesection}{\Alph{section}}

\begin{center}
{\Large Supplementary Material to the article \\[1ex] Central limit theory for Peaks-over-Threshold partial sums of long memory linear time series} \\[1ex]
Ioan Scheffel, Marco Oesting, Gilles Stupfler
\end{center}

\section{Auxiliary results}
In this section we collect a host of results that are subsequently used in the proofs, but may also be of independent interest.
\subsection{Martingale difference inequality}
We start this section by recalling---and refining---a fundamental inequality for martingale difference arrays (cf.~\cite{bahrInequalities$r$thAbsolute1965})
which is satisfied in particular by independent, centered random variables.
\begin{lemma}[Martingale Difference Inequality]
  \label{lem:bahr}
  \
  Let $p\in[1,2]$, and
  let $(Y_{m,i}\,, i=1, \ldots,m\,, m\in\NN )\subset L^p(\PP)$ be an
  array of random variables satisfying
\begin{align}
  \label{eq:bahr:cond:1}
  \EE
  \Bigg[
  Y_{m,\ell + 1} \, \Bigg{|} \, \sum_{i=1}^{\ell}Y_{m,i}
  \Bigg]
  \ = \
  0
  \qquad\text{for all}\ 1\le \ell \le m-1
  \ \text{and all}\  m\in\NN\,.
\end{align}
\begin{enumerate}
  \item
It holds that
\begin{align}
  \label{eq:bahr:fi}
  \EE
  \left|
  \sum_{i=1}^{m}
  Y_{m,i}
  \right|^p
  \ \le \
  2
  \sum_{i=1}^{m}
  \EE
  \left|
  Y_{m,i}
  \right|^p
  \qquad\text{for all}\ m\in\NN\,.
\end{align}
\item
If there exists a sequence $(Y_i)_{i\in\NN}\subset L^p(\PP)$ such that
  \begin{align}
    \label{eq:bahr:cond:2}
    \lim_{m\to\infty}
    \sum_{i=1}^m
    Y_{m,i}
     =
    \sum_{i=1}^{\infty}
    Y_{i}
    \quad\text{$\PP$-almost surely, and }\quad
    \liminf_{m\to\infty}
    \sum_{i=1}^m
    \EE|
    Y_{m,i}
    |^p
    \le
    \sum_{i=1}^\infty
    \EE|
    Y_{i}
    |^p
    \,,
  \end{align}
  then
\begin{align}
  \label{eq:bahr:infty}
  \EE
  \left|
  \sum_{i=1}^{\infty}
  Y_i
  \right|^p
  \ \le \
  2 \sum_{i=1}^{\infty}
  \EE
  \left|
  Y_i
  \right|^p
  \,.
\end{align}
\end{enumerate}
  \end{lemma}
\begin{proof}
  Part (i) is exactly~\cite[Theorem~2]{bahrInequalities$r$thAbsolute1965}.
  To prove Part (ii),
  we apply Fatou's lemma together with \eqref{eq:bahr:fi} and \eqref{eq:bahr:cond:2} to obtain
  \begin{align*}
  \EE
  \left|
  \sum_{i=1}^{\infty}
  Y_i
  \right|^p
    \ = \
  \EE \left(
  \liminf_{m\to\infty} \left|
  \sum_{i=1}^{m}
    Y_{m,i}
  \right|^p \right)
  \ \le \
    \liminf_{m\to\infty}
  \EE
  \left|
  \sum_{i=1}^{m}
    Y_{m,i}
  \right|^p
  & \ \le \
    \liminf_{m\to\infty}
    2
  \sum_{i=1}^{m}
  \EE
  \left|
    Y_{m,i}
  \right|^p \\
  &\le
    2
  \sum_{i=1}^{\infty}
  \EE
  \left|
  Y_i
  \right|^p
  \,,
  \end{align*}
  as announced.
\end{proof}

\subsection{Regularity of distribution functions and their derivatives}
For $\gamma\in (0,1)$, we define
\begin{align*}
  g_\gamma \ :\  \mathbb{R} \to (0,1], \qquad x\ \mapsto \ g_\gamma(x) \ =\  \frac{1}{(1 + |x|)^{\gamma + 1}}.
\end{align*}
For $k\in\NN$, we define the truncated time series as
\begin{align*}
  X_{t,k}
  \ := \ \sum_{j=0}^k
  a_j \varepsilon_{t-j}
  \qquad\text{and}\qquad
  \widetilde{X}_{t,k}
  \ := \
  X_t
  \ - \ X_{t,k}
  \ = \ \sum_{j=k+1}^\infty
  a_j \varepsilon_{t-j}
  \,.
\end{align*}
By convention, we write $X_{t,\infty}:=X_t$. In the following, let $f_k$ denote the density of $X_{0,k}$, and $f_{\infty}$ denote the density of $X_{0,\infty}=X_0$.
Note that
$X_{0,k}=_{d} X_{t,k}$ for all $t\in\mathbb{Z}$, due to stationarity of the sequence $(\varepsilon_t)$.
The next lemma is needed when dealing with heavy tails ($\alpha\in(1,2)$). It is a substitute for the mean value theorem that may be applied in settings with finite fourth moments (compare \cite{koulAsymptoticsEmpiricalProcesses2001, hoAsymptoticExpansionEmpirical1996}). Also note that, as explained in the introduction, we assimilate the case $\alpha=2$ to fit in the framework of $\alpha\in(1,2)$, which is essentially done here.
\begin{lemma}
  \label{lem:f}
  Let Assumptions~\ref{asu:coef} and~\ref{asu:f} hold.
  Then there exists $k_0 \in \NN$ such that
  for all $k>k_0$ the distribution functions of $X_0$ and $X_{0,k}$
  belong to $C^2(\RR)$.
  Furthermore, for all $\gamma\in(0,\alpha-1)$ and $r\in (1,2)$ with $1+\gamma<r<\alpha$ and $1/(1-d)<r$,
  it holds for $x,y\in \RR$ with $|x-y|\le 1$ and all $k>k_0$
  \begin{align}
    \label{lem:f:1}
    |f_{\infty}'(x)|
    \ + \
    |f_k'(x)|
    &
    \ \lesssim \
    g_\gamma(x)
    \,,
    \\
    \label{lem:f:2}
    |
    f_{\infty}'(x)
    -
    f_{\infty}'(y)
    |
    \ + \
    |
    f_k'(x)
    -
    f_k'(y)
    |
    &
    \ \lesssim \
    |x-y|
    \cdot
    g_\gamma(x)
    \,,
    \\
    \label{lem:f:3}
    |
    f_{\infty}'(x)
    -
    f_k'(x)
    |&
    \ \lesssim \
    k^{-(1-d) + 1/r}\cdot
    g_\gamma(x)
    \,,
  \end{align}
  where $\lesssim$ is $\le$ up to a constant $C_{\gamma,r}>0$ only depending on $\gamma$ and $r$.
\end{lemma}
\begin{proof}
  Under Assumption~\ref{asu:f}(ii),
  the statement is proved in
  \cite[Lemma~4.2]{koulAsymptoticsEmpiricalProcesses2001}.
  Under Assumption~\ref{asu:f}(i),
   a proof of a slightly different statement is found in
  \cite[Lemma~2]{giraitisCentralLimitTheorem1999}.
  Indeed, the only difference is in \eqref{lem:f:1}-\eqref{lem:f:3}, that is, \cite[Lemma~2]{giraitisCentralLimitTheorem1999} prove that
\begin{align*}
    |f_{\infty}'(x)|
    \ + \
    |f_k'(x)|
    &
    \ \lesssim \
    g_1(x)
    \,,
    \\
    |
    f_{\infty}'(x)
    -
    f_{\infty}'(y)
    |
    \ + \
    |
    f_k'(x)
    -
    f_k'(y)
    |
    &
    \ \lesssim \
    |x-y|
    \cdot
    g_1(x)
    \,,
    \\
    |
    f_{\infty}'(x)
    -
    f_k'(x)
    |&
    \ \lesssim \
    k^{2d-1}\cdot
    g_1(x)
    \,.
  \end{align*}
  Since $g_1\le g_\gamma$ for $\gamma\in(0,1)$, we immediately recover \eqref{lem:f:1} and \eqref{lem:f:2}.
  The only remaining difference is in the third inequality.
  For this, note that
  $
    k^{2d-1}
    \ \le \
    k^{-(1-d) + 1/r}
  $
  due to
  \begin{align*}
    (2d-1)
    \ + \
    (1-d) - 1/r
    \ = \
    d
    \ - \
    1/r
    \ \le \
    1
    \ - \
    \frac{2}{\alpha}
       \
       \le
    \
    0
  \end{align*}
  since
  $1/\alpha<1/r$, $d<1-1/\alpha$, and $\alpha\le 2$. Inequality \eqref{lem:f:3} follows.
\end{proof}
The next lemma is an analytic tool to deal with $g_\gamma$, which may arise after applying Lemma~\ref{lem:f}.
\begin{lemma}
  \label{lem:5.1}
  Let $\gamma\in (0,1)$. Then the following hold:
  \begin{enumerate}[label=(\roman*)]
    \item
    For all $y,z\in\mathbb{R}$,
\begin{align}
    \label{lem:5.1:1}
    g_{\gamma}
    (z+y)
    \ \lesssim \
    g_\gamma(z)
    \cdot
    (1\lor |y|)^{1+\gamma}\,.
  \end{align}
  \item
    Let $a,b\in\mathbb{R}$ with $a<b$. Then, for all $z\in\mathbb{R}$,
\begin{align}
    \label{lem:5.1:2}
    \int_a^b g_\gamma(z-w)\,\mathrm{d}w
    \ \lesssim \
    g_\gamma(z)
    \cdot
    ((b-a) \lor 1)^{1+\gamma}
    \,.
  \end{align}
  \end{enumerate}
  Here $\lesssim$ denotes inequality up to a constant $C_{\gamma}>0$ depending only on $\gamma$.
\end{lemma}
\begin{proof}
  The Inequality \eqref{lem:5.1:1} follows from \cite[Lemma~5.1]{koulAsymptoticsEmpiricalProcesses2001}.
  We now prove \eqref{lem:5.1:2}.
Let $a,b\in\RR$ with $a<b$, and define
$
   u
    :=
   b-a + 1
   \ \ge \
   1
$
and
$
   \tilde{z}
    :=
   z
   -
   a+1
$.
  Then
\begin{align*}
  \int_a^b g_\gamma(z-w) \,\mathrm{d}w
  \ = \
  \int_1^{u}
  g_\gamma(\tilde{z}-w)
  \,\mathrm{d}w
  \,.
\end{align*}
Since
$g_\gamma(\widetilde z) \lesssim g_\gamma(z)$, it suffices to consider the case
  $a = 1 <b$.
Note that for all
$b>1$, we have
$
  b\ \le \ 2((b-1)\lor 1)
$, and thus it suffices to show
for all $z\in\RR$
\begin{align}
  \label{eq:result}
  \frac{1}{b^{1+\gamma}}
  \int_1^b
  g_{\gamma}(z-w)
  \,\mathrm{d}w
  \ \lesssim \
  g_\gamma(z)
  \,.
\end{align}
We now consider two cases.
\\
\textbf{Case $|z|\le 1$:}
It holds
\begin{align*}
  \frac
  {
  |z|+1
  }
  {2}
  \ \le \
  1
  \ < \
  b
  \qquad
  \text{and hence}
  \qquad
  \frac
  {1}
  {b^{1+\gamma}}
  \ \lesssim \
  \frac
  {1}
  {(1+|z|)^{1+\gamma}}
    \ = \
    g_\gamma(z)
  \,.
\end{align*}
Since
\begin{align}
  \label{eq:result_bound}
  \int_1^b
  g_{\gamma}(z-w)
  \,\mathrm{d}w
  \ \le \
  \int_{\RR}
  g_{\gamma}(w)
  \,\mathrm{d}w
  \ < \
  \infty
  \qquad\text{for all}\  z\in\RR
  \,,
\end{align}
we get \eqref{eq:result}.
\\
\textbf{Case $|z|>1$:}
We will use, without further mention,
the inequalities
$1/|z|\le 2/(1+|z|)\le 2/|z|$,
and consider three subcases.
\\
\textbf{Subcase
$b\ge |z|/2$:
}
It holds
\begin{align*}
  \frac
  {1}
  {b^{1+\gamma}}
  \ \lesssim \
  \frac
  {1}
  {(1+|z|)^{1+\gamma}}
  \,,
\end{align*}
and by \eqref{eq:result_bound}, this implies \eqref{eq:result}.
\\
\textbf{Subcase
 $z/2>b\ge 1$:
}
In this case it holds
$z-1>z-b>0$, and
\begin{align*}
  \int_1^b
  g_\gamma(z-w)
  \,\mathrm{d}w
  \ \le \
  \int_{z-b}^{z-1}
  \frac
  {1}
  {w^{1+\gamma}}
  \,\mathrm{d}w
  &
  \ = \
  \frac
  {1}
  {\gamma}
  \left(
  \frac
  {1}
  {(z-b)^{\gamma}}
  \ - \
  \frac
  {1}
  {(z-1)^{\gamma}}
  \right)
  \\&
  \ = \
  \frac
  {1}
  {\gamma(z-b)^{\gamma}}
  \left(
  1
  \ - \
  \left(
  \frac
  {z-b}
  {z-1}
  \right)^{\gamma}
  \right)
  \,.
\end{align*}
By the fundamental theorem of calculus, we obtain
\begin{align*}
  1
  \ - \
  \left(
  \frac
  {z-b}
  {z-1}
  \right)^{\gamma}
  \ = \
  -
  \int_1^b
  \frac{\mathrm{d}}{\mathrm{d} v}
  \left(
  \left(
  \frac
  {z-v}
  {z-1}
  \right)^{\gamma}
  \right)
  \,\mathrm{d} v
  &
  \ = \
  \frac
  {\gamma}
  {(z-1)^\gamma}
  \cdot
  \int_1^b
  (z-v)^{\gamma -1}
  \,\mathrm{d}v
  \\&
  \ \le \
  \frac
  {\gamma}
  {(z-1)^{\gamma}}
  \cdot
  b
  (z-b)^{\gamma -1}
\end{align*}
since $\gamma\in(0,1)$.
Thus
\begin{align*}
\int_1^b
  g_\gamma(z-w)
  \,\mathrm{d}w
  \ \le \
  \frac
  {b}
  {(z-1)^{\gamma}(z-b)}
  \ \le \
  \frac
  {b}
  {(z-z/2)^{\gamma}(z-z/2)}
  \ \lesssim \
  \frac
  {b}
  {z^{\gamma+1}}
  \quad
  \text{by}
   \
   z/2 > b \ge 1\,.
\end{align*}
It follows \eqref{eq:result}, that is,
\begin{align}
  \frac
  {1}
  {b^{1+\gamma}}
\int_1^b
  g_\gamma(z-w)
  \,\mathrm{d}w
  \ \lesssim \
  \frac
  {1}
  {b^{\gamma}}
  \frac
  {1}
  {(1+|z|)^{1+\gamma}}
  \ \lesssim \
  \frac
  {1}
  {(1+|z|)^{1+\gamma}}
  \ = \
  g_\gamma(z)
  \,.
  \label{eq:result_bound:1}
\end{align}
\textbf{Subcase
$-z/2>b\ge 1$:
}
Using the symmetry of $g_\gamma$, we write
\begin{align*}
  \int_1^b
  g_\gamma(z-w)
  \,\mathrm{d}w
  \ =\
  \int_1^b
  g_\gamma(w-z)
  \,\mathrm{d}w
  &
  \ \le \
  \int^{b-z}_{1-z}
  \frac
  {1}
  {w^{1+\gamma}}
  \,\mathrm{d}w
  \\
  &
  \ = \
  \frac
  {1}
  {\gamma}
  \left(
  \frac
  {1}
  {(1-z)^{\gamma}}
  \ - \
  \frac
  {1}
  {(b-z)^{\gamma}}
  \right)
  \\&
  \ = \
  \frac
  {1}
  {\gamma(b-z)^{\gamma}}
  \left(
  \left(
  \frac
  {b-z}
  {1-z}
  \right)^{\gamma}
  \ - \
  1
  \right)
  \,.
\end{align*}
Applying the fundamental theorem of calculus again and using $\gamma \in (0,1)$, we obtain
\begin{align*}
  \left(
  \frac
  {b-z}
  {1-z}
  \right)^{\gamma}
   \ - \
  1
  \ = \
  \int_1^b
  \frac{\mathrm{d}}{\mathrm{d} v}
  \left(
  \left(
  \frac
  {v-z}
  {1-z}
  \right)^{\gamma}
  \right)
  \,\mathrm{d} v
  &
  \ = \
  \frac
  {\gamma}
  {(1-z)^{\gamma}}
  \cdot
  \int_1^b
  (v-z)^{\gamma - 1}
  \,\mathrm{d}v
  \\&
  \ \le \
  \frac
  {\gamma}
  {1-z}
  \cdot
  b
  \,.
\end{align*}
Thus,
\begin{align*}
  &
\int_1^b
  g_\gamma(z-w)
  \,\mathrm{d}w
  \\&
  \ \le \
  \frac
  {b}
  {(b-z)^{\gamma}(1-z)}
  \ \le \
  \frac
  {b}
  {(1-z)^{\gamma}(1-z)}
  \ \lesssim \
  \frac
  {b}
  {|z|^{\gamma+1}}
  \quad
  \text{by}
   \
  - z/2 > b \ge 1\,.
\end{align*}
As in \eqref{eq:result_bound:1},
this yields \eqref{eq:result}.
This completes the proof.
\end{proof}
\subsection{Swap integration and differentiation}
The next lemma asserts that all the notions constructed from $G_n$ and used in this work are well-defined. It justifies interchanging differentiation and integration, which simplifies the terms we are going to work with.
\begin{lemma}[Swap Integration and Differentiation]
  \label{lem:swap}
  Let Assumptions~\ref{asu:coef},~\ref{asu:G_n} and~\ref{asu:f} hold
  with $\gamma_G < d/(1-d) $, and define
$
  G_{k,n}(y)
  \ := \
  \EE[G_n(X_{0,k}+y)]
  \ = \
  \int_{\mathbb{R}} G_n(x) f_k(x-y)\mathrm{d}x
  \,.
$
  Then the following statements hold:
  \begin{enumerate}[label=(\roman*)]
    \item
    $\EE|G_n(X_{0,k}+y)|^r<\infty$ for all $k\in \NN \cup \{\infty\}$,  $n\in\NN$, $y\in\RR$, and for all $r\in (0,\alpha)$.
  \item
With the notation of Lemma~\ref{lem:f},
    for all $k>k_0$ or $k=\infty$, it holds that the function $G_{k,n}$ is continuously differentiable with
  \begin{align*}
    G_{k,n}'(y)
    \ := \
     \frac{\mathrm{d}}{\mathrm{d}y}
     G_{k,n}(y)
     \ = \
    -
    \int_{\RR}
    G_n(x)
    f'_{k}
    (x-y)
    \,\mathrm{d}x
    \ < \ \infty
    \qquad\text{for all}\ y\in\RR\,.
  \end{align*}
  \end{enumerate}
   \end{lemma}
\begin{proof}
\textbf{
Proof of Part (i):
}
  By Assumption~\ref{asu:coef} there exists $\delta>0$ such that $(\alpha-\delta)(1-d)>1$.
  Let $\gamma_G$ as in Assumption~\ref{asu:G_n}.
  Since $r<\alpha$ and
  \begin{align}
    \label{nrtd}
  \gamma_G
  <
\frac{d}{1-d}
  \ = \
  \frac{1}{1-d}
  \ - \
  1
  \ < \
  \alpha - 1 \le 1
  \end{align}
  by assumption, we can choose $\delta>0$ small enough such that $\alpha-\delta\geq 1$ and $r\gamma_G<\alpha-\delta$. For $y\in \RR$ it follows
 \begin{align*}
   \EE|G_n(X_{0,k}+y)|^r
    &
    \ \lesssim \
    \EE[(1+|y+X_{0,k}|)^{r\gamma_G}]
    \\&
    \ \lesssim \
    \EE[(1+|y+X_{0,k}|)^{\alpha-\delta}]
    \\&
    \ \lesssim \
    1 \ + \
    |y|^{\alpha-\delta}
    \ + \
    \EE|X_{0,k}|^{\alpha-\delta}
    \\&
    \ \lesssim \
    1 \ + \
   |y|^{\alpha-\delta}
    \ + \
    \EE|\varepsilon|^{\alpha-\delta}
    \sum_{j=0}^\infty
   |a_j|^{\alpha-\delta}
    \\&
    \ \lesssim \
    1 \ + \
    |y|^{\alpha-\delta}
    \ + \
    \sum_{j=1}^\infty
    j^{-(\alpha-\delta)(1-d)}
    \ < \ \infty
    \,.
  \end{align*}
  The third inequality is due to the convexity of $x\mapsto |x|^{\alpha-\delta}$,
  and
  the fourth inequality is due to
  the last part of Lemma~\ref{lem:bahr}(i).
  This concludes the proof of Part (i).
  \\
\textbf{
Proof of Part (ii):
} Fix $y\in\RR$ and assume $y\neq y'$.
  By Part (i), it follows $\EE|G_n(X_{0,k}+y)|<\infty$ for all $y\in\RR$, so that we can use linearity of the expectation to get
 \begin{align}
    \label{lem:swap:1}
    \begin{split}
    &
    \frac{
    \EE[G_n(X_{0,k}+y)]
    -
    \EE[G_n(X_{0,k}+y')]
    }{y-y'}
    \ = \
    \int_\RR
    G_n(x)
    \frac{
    f_k(x-y)
    -
    f_k(x-y')
    }{y-y'}
    \,\mathrm{d}x
    \,.
    \end{split}
  \end{align}
  If we may interchange the limit $y'\to y$ with the integral, then
 \begin{align*}
    &
   \frac{\mathrm{d} }{\mathrm{d}y}
    \EE[G_n(X_{0,k}+y)]
    \ = \
    \lim_{y'\to y}
    \frac{
    \EE[G_n(X_{0,k}+y)]
    -
    \EE[G_n(X_{0,k}+y')]
    }{y-y'}
    \\&
    \ = \
    \int_\RR
    G_n(x)
    \lim_{y'\to y}
    \frac{
    f_k(x-y)
    -
    f_k(x-y')
    }{y-y'}
    \,\mathrm{d}x
    \ = \
    -
    \int_\RR
    G_n(x)
    f'_k(x-y)
    \,\mathrm{d}x
    \,,
  \end{align*}
  as required.
  We seek to use dominated convergence to interchange these limits, and to this end we show
  that there is a bivariate function $g$ such that $g(\cdot,y)\in L^1(\RR)$ and
  \begin{align}
    \label{lem:swap:eq:4}
\left|
    G_n(x)
    \frac{
    f_k(x-y)
    -
    f_k(x-y')
    }{y-y'}
\right|
\ \lesssim \
    g(x,y)
    \qquad\text{uniformly in $y'$ with $|y'-y|\le 1$.}
  \end{align}
  First we apply the fundamental theorem of calculus and the triangle inequality to get
    \begin{align*}
    \left|
    G_n(x)
    \frac{
    f_k(x-y)
    -
    f_k(x-y')
    }{y-y'}
    \right|
    \ \le \
    \frac{
    G_n(x)
    }
      {|y-y'|}
      \int_{(y\land y')}^{(y\lor y')}
    \left|
    f'_k(x-s)
    \right|
    \,\mathrm{d}s
    \,.
  \end{align*}
Since $\gamma_G< \alpha-1$
    by \eqref{nrtd},
    there exists $\gamma$ with $\gamma_G<\gamma<\alpha-1\le 1$ such that from Lemma~\ref{lem:f} with $1+\gamma < (\alpha \land 2) $, and the first part of  Lemma~\ref{lem:5.1} we get
  \begin{align*}
      \int_{(y\land y')}^{(y\lor y')}
    \left|
    f'_k(x-s)
    \right|
    \,\mathrm{d}s
    &
    \ \lesssim \
    g_{\gamma}(x)
      \int_{(y\land y')}^{(y\lor y')}
      (1\lor|s|)^{1+\gamma}
      \,\mathrm{d}s
      \\&
      \ \le \
    g_{\gamma}(x)
    |y-y'|
      (1\lor|y|\lor |y'|)^{1+\gamma}
      \\
      &
      \ \lesssim \
    g_{\gamma}(x)
    |y-y'|
    (1+|y|)^2
  \end{align*}
  since $|y-y'|\le 1$.
  Then, with
   Assumption~\ref{asu:G_n}
  \begin{align*}
    &
    \frac{
    G_n(x)
    }
      {|y-y'|}
      \int_{(y\land y')}^{(y\lor y')}
    \left|
    f'_k(x-s)
    \right|
    \,\mathrm{d}s
    \\&
    \ \lesssim \
    G_n(x)
    g_{\gamma}(x)
    (1+|y|)^2
    \ \lesssim \
    (1+|x|)^{\gamma_G}
    g_{\gamma}(x)
    (1+|y|)^2
    \ \lesssim \
    g_{\gamma-\gamma_G}(x)
    (1+|y|)^2
    \,.
  \end{align*}
  With $\gamma-\gamma_G>0$ and therefore
  \begin{align}
    \label{eq:lem:swap:g_bdd}
    \int_\RR
    g_{\gamma-\gamma_G}(x)
    \,\mathrm{d}x
    \ < \ \infty
    \,,
  \end{align}
    it follows \eqref{lem:swap:eq:4} with $g(x,y)=  g_{\gamma-\gamma_G}(x)
    (1+|y|)^2$.
    To complete the proof, we show that $G_{k,n}'$ is continuous.
    To this end, let $y\in\RR$ and $(y_m)_{m\in\mathbb{N}}$ a bounded sequence with $y_m\to y$ for $m\to\infty$.
    We want to show
    \begin{align*}
      \lim_{m\to\infty}
      G_{k,n}'(y_m)
      &
      \ = \
      \lim_{m\to\infty}
      -
      \int_{\RR}
      G_n(x)
      f_k'(x-y_m)
      \,\mathrm{d}x
      \\&
      \ = \
      -
      \int_{\RR}
      G_n(x)
      \lim_{m\to\infty}
      f_k'(x-y_m)
      \,\mathrm{d}x
      \\&
      \ = \
      -
      \int_{\RR}
      G_n(x)
      f_k'(x-y)
      \,\mathrm{d}x
      \\&
      \ = \
      G_{k,n}'(y)
      \,.
    \end{align*}
    The third equality is due to the continuity of $f_k'$ (see Lemma~\ref{lem:f}). To show the second equality we want to apply dominated convergence. To this end, note that by~\eqref{lem:f:1} in Lemma~\ref{lem:f} and then Lemma~\ref{lem:5.1}(i), with the choice of $\gamma$ as above,
    \begin{align*}
      \left|
      G_n(x)
      f_{k}'(x-y_m)
      \right|
       \ \lesssim \
      (1+|x|)^{\gamma_G}
      g_{\gamma}(x)
      (1\lor |y_m|)^{1+\gamma}
      \ \lesssim \
      g_{\gamma-\gamma_G}(x)
      \,.
    \end{align*}
    Together with \eqref{eq:lem:swap:g_bdd} we may apply dominated convergence to finish the proof.
\end{proof}

Lemma~\ref{lem:swap} tells us that
\begin{align*}
  G_{k,n}(y)
  \ := \
  \EE[G_n(X_{0,k}+y)]
  \qquad\text{and}\qquad
  G_{k,n}'(y)
  \ := \
  \frac{\mathrm{d}}{\mathrm{d}y}
  \EE[G_n(X_{0,k}+y)]
\end{align*}
are well-defined
for $k>k_0$.
The following lemma establishes a connection between consecutive values of $G_{k,n}$ and $G_{k,n}'$.
 \begin{lemma}
   \label{lem:cond}
  Let Assumptions~\ref{asu:coef},~\ref{asu:G_n} and~\ref{asu:f} hold true.
  It holds for $j\in\{0,1\}$ and $k>k_0$
   \[
     G_{k+1,n}^{(j)}(y)
     \ = \
     \EE[G_{k,n}^{(j)}(a_{k+1}\varepsilon_{-k-1}+y)]
     \qquad
     \text{and}
      \qquad
     G_{\infty,n}^{(j)}(y)
     \ = \
     \EE[G_{k,n}^{(j)}(\widetilde{X}_{0,k}+y)]
     \,.
   \]
 \end{lemma}
 \begin{proof}
   Let $(f\ast g)(y)$ denote the convolution defined by $\int_\RR f(y-x)g(x)\,\mathrm{d}x$.
   Since $\varepsilon$ is symmetric, the density $f_\varepsilon$ is even, and its derivative $f'_\varepsilon$ is odd.
   Consequently, $f_k$ is even and $f_k'$ is odd
  for all $k\in\NN\cup\{\infty\}$.
   Thus, for $k>k_0$,
   \begin{align}
     \label{eq:lab1}
 G_{k+1,n}(y)
     \ = \
     \int_\RR
     G_n(x)
     f_{k+1}(x-y)
     \,\mathrm{d}x
     \ = \
     (G_n\ast f_{k+1})(y)
     \,,
   \end{align}
   and as shown in Lemma~\ref{lem:swap}
   \begin{align}
     \label{eq:lab2}
 G_{k+1,n}'(y)
     \ = \
     \ - \
     \int_\RR
     G_n(x)
     f_{k+1}'(x-y)
     \,\mathrm{d}x
     \ = \
    (G_n\ast f_{k+1}')(y)
    \,.
   \end{align}
Observe that $f_{k+1} = f_k \ast f_{a_{k+1}\varepsilon_{-k-1}}$, by the independence of the innovations, and then, by dominated convergence,
  \begin{align}
    \label{eq:lab3}
    f_{k+1}' = (f_{k}\ast f_{a_{k+1}\varepsilon_{-k-1}})'
    \ = \
    (f'_{k}\ast f_{a_{k+1}\varepsilon_{-k-1}})
  \end{align}
  since
   $f_{a_{k+1}\varepsilon_{-k-1}}\in L^{1}(\RR)$,
   and
   $f_{k}\in C^1(\mathbb{R})$ with bounded derivative for all $k>k_0$ by~\eqref{lem:f:1} in Lemma~\ref{lem:f}.
      Thus for $j\in\{0,1\}$,
    \begin{align*}
    G_{k+1,n}^{(j)}(y)
    &
    \ = \
    (G_n\ast f_{k+1}^{(j)})(y)
    \\&
    \ = \
      ((G_n\ast f_{k}^{(j)})\ast f_{a_{k+1}\varepsilon_{-k-1}})(y)
    \\&
    \ = \
    (G_{k,n}^{(j)}\ast f_{a_{k+1}\varepsilon_{-k-1}})(y)
    \\&
    \ = \
    \int_\RR
      G_{k,n}^{(j)}(x)f_{a_{k+1}\varepsilon_{-k-1}}(x-y)
      \,\mathrm{d}x
      \\&
    \ = \
    \EE[G_{k,n}^{(j)}(a_k\varepsilon_{-k-1} + y)]
    \,,
  \end{align*}
  where the first and third equalities are due to \eqref{eq:lab1} and \eqref{eq:lab2}, and the second equality is due to
  \eqref{eq:lab3} (and associativity of the convolution operator).
   Observe, finally, that $f_{\infty} = f_k \ast f_{\widetilde{X}_{0,k}}$, and that a similar argument can be written for $f_{a_{k+1}\varepsilon_{-k-1}}$ replaced by $f_{\widetilde{X}_{0,k}}$.
   Therefore the statement holds also for $k=\infty$. This completes the proof.
   \end{proof}

\section{Proof of Theorem~\ref{thm:approx}}

The proof of Theorem~\ref{thm:approx} is long. It draws on ideas from \cite{koulAsymptoticsEmpiricalProcesses2001}, where the authors conducted their analysis without the benefit of finite second moments---a difficulty we must also address.
Throughout this section, let $n\in\NN$ unless stated otherwise.
To track the deviation of the PoT statistic from its approximation, we define a centered process that isolates the remainder term after linear approximation. This term will be central in applying martingale arguments.
In the sequel, for $n\in\NN$, define
  \begin{align*}
    \mathcal{U}_n(X_t)
    \ := \
    G_n(X_t) \ -\  \EE[G_n(X_t)] \ -\  G_{\infty,n}'(0) X_t
    \qquad\text{for}\ t\in\ZZ\,.
  \end{align*}
Note that under Assumptions~\ref{asu:coef},~\ref{asu:G_n} and~\ref{asu:f}, the term $G_{\infty,n}'(0)$ is indeed well-defined by Lemma~\ref{lem:swap}. Hence, $\mathcal{U}_n(X_t)$ itself will be well-defined and integrable once we verify that $G_n(X_0)$ and $X_0$
are integrable. For $k\in\ZZ$, let $\mathcal{F}_k=\sigma(\varepsilon_k,\varepsilon_{k-1},\ldots)$ denote the past $\sigma$-algebra generated by the sequence $(\varepsilon_t)$ up to index $k$.
For $Y\in L^1(\PP)$, we define the
projection
  \begin{align*}
    P_k\,Y
    \ := \
    \mathbf{E}\left[Y\, | \, \mathcal{F}_k\right]
    \ - \
    \mathbf{E}\left[Y\, | \, \mathcal{F}_{k-1}\right]
    \qquad \text{for all}\ k\in \ZZ
    \,.
  \end{align*}
We proceed with a sequence of lemmas that prepare the proof of the final result.
The next lemma establishes integrability properties of the centered process $\mathcal{U}_n(X_t)$---first in its unprojected form, then after projection by $P_k$.
This integrability is needed to apply the lemma on martingale difference arrays (Lemma~\ref{lem:bahr}) later in the proof. Note that Part~(ii) also provides an upper bound that contributes directly to the final result.
  \begin{lemma}
    \label{lem:2.2}
    Under Assumptions~\ref{asu:coef},~\ref{asu:G_n} and~\ref{asu:f}, for $r\in[1,\alpha)$, the following statements hold:
    \begin{enumerate}[label=(\roman*)]
      \item $\mathcal{U}_n(X_0)\in L^r(\PP)$.
      \item
      $
      \norm{P_k\,\mathcal{U}_n(X_0)}_{L^r(\PP)}
       \lesssim
      \norm{G_n(X_0)}_{L^r(\PP)}
      \lor
      |G_{\infty,n}'(0)|
      <\infty
      $
      for all $k\in\mathbb{Z}$.
    \end{enumerate}
  \end{lemma}
  \begin{proof}
  \textbf{Proof of Part (i):}
  Let $r\in [1,\alpha)$.
    By Assumption~\ref{asu:G_n}, the inequalities
    \begin{align*}
      \gamma_G
      \ < \
      \frac{d}{1-d}
      \ < \
      1
      \ < \
      \frac{1}{1-d}
      \,,
    \end{align*}
    which hold because of $d<1/2$ (see Assumption~\ref{asu:coef}), and the convexity of $x \mapsto |x|^r$, we obtain
    \begin{align}
      \label{eq:2.2:1}
      \begin{split}
      \EE|G_n(X_0)|^r
      &
      \ \lesssim \
      \EE(1+|X_0|)^{r\gamma_G}
      \ \le \
      \EE(1+|X_0|)^{r}
      \ \le \
      2^{r-1}
      \left(
1
 +
      \EE|X_0|^{r}
      \right)
      \\
      &
      \ \le \
      2
      \left(
1
 +
 \EE|X_0|^r
      \right)
      \,.
      \end{split}
    \end{align}
    To further bound $\EE|X_0|^r$,
     we choose $\widetilde r \in (r\lor (1/(1-d)), \alpha)$, such that, by Hölder's inequality,
    $
    \EE|X_0|^r
    \ \le \
      (\EE|X_0|^{\widetilde r})^{r/\widetilde{r}}
    $, and continue to bound $\EE|X_0|^{\widetilde{r}}$. To this end,
    we apply Lemma~\ref{lem:bahr} to the array $(Y_{m,i})=(a_i \varepsilon_{-i}\,\colon \, i=0,\ldots,m\,,m\in\NN)$ and the sequence $(Y_i)=(a_i \varepsilon_{-i})_{i\ge 0}$.
       By Assumption~\ref{asu:f}, both $(Y_{m,i})$ and $(Y_i)\subset L^{\widetilde{r}}(\PP)$.
       Since the innovations are independent and centered,
  Condition~\eqref{eq:bahr:cond:1} is satisfied.
    Moreover, Condition \eqref{eq:bahr:cond:2} is clear, so by \eqref{eq:bahr:infty} in Lemma~\ref{lem:bahr} and Assumption~\ref{asu:coef}, we get
    \begin{align}
      \EE|X_0|^{\widetilde r}
      \ = \
      \EE
      \left|\sum_{i=0}^\infty Y_i
      \right|^{\widetilde r}
      \ \lesssim \
      \sum_{i=0}^\infty
      \EE\left|
      Y_i
      \right|^{\widetilde r}
      \ \lesssim \
      \sum_{i=1}^\infty
      i^{-(1-d)\widetilde r}
      \ < \ \infty\,,
      \label{eq:2.2:2}
    \end{align}
    since we chose $\widetilde r>1/(1-d)$.
    Using \eqref{eq:2.2:1}, this shows that $G_n(X_0)\in L^r(\PP)$.
    Since $X_0\in L^r(\PP)$ by \eqref{eq:2.2:2},
    the proof of Part (i) is complete.
    \\
  \textbf{Proof of Part (ii):}
    From Part (i), we know that $P_k \, \mathcal{U}_n (X_0)$ is well-defined.
  For fixed $j,k\in\mathbb{Z}$, observe that
  \begin{align*}
    P_{-k}
    \varepsilon_{-j}
    &
    \ = \
    \begin{cases}
      0\,\quad &\text{if} \ j\neq k\,,\\
      \varepsilon_{-k}\,\quad &\text{if} \ j= k\,,\\
    \end{cases}
  \end{align*}
  so, by dominated convergence, it follows that
 \begin{align*}
   P_{k}\, X_0
   \ = \
   \sum_{i=0}^\infty
   a_i
   P_{k}\,\varepsilon_{-i}
   \ = \
    \begin{cases}
      0\,\quad &\text{if} \ k>0\,,\\
      a_{-k}\varepsilon_{k}\,\quad &\text{if} \ k\le 0\,.\\
    \end{cases}
 \end{align*}
 Thus,
\begin{align}
\nonumber
  P_{k}\,\mathcal{U}_n(X_0)
  &
  \ = \
  P_{k}\,
  (G_n(X_0) \ -\  \EE[G_n(X_0)] \ -\  G_{\infty,n}'(0)X_0)
  \\
\label{eq:PkU}
  &
  \ = \
    \EE[G_n(X_0)\, | \, \mathcal{F}_{k}]
    \ - \
    \EE[G_n(X_0)\, | \, \mathcal{F}_{k-1}]
    \ - \
    G_{\infty,n}'(0)
    \cdot
    a_{-k}\cdot\varepsilon_{k}
    \cdot
    \ind\{k\le 0\}
    \,.
\end{align}
  Applying Jensen's inequality with $r\ge 1$, and using the properties of conditional expectation,
  we obtain
\begin{align*}
  \EE
    |\EE[G_n(X_0)\, | \, \mathcal{F}_{k}]|^r
    &
    \ \le \
    \EE|G_n(X_0)|^r
       \ < \ \infty\
       \qquad\text{as shown in Part (i)}
    \,.
    \end{align*}
Combining Equation~\eqref{eq:PkU} with the fact that $\varepsilon\in L^r(\PP)$ for $r<\alpha$ completes the proof of Lemma~\ref{lem:2.2}.
  \end{proof}
    The next lemma provides an intermediate upper bound on the error term
    \begin{align*}
    \sum_{t=1}^n
         \left(G_n(X_t)
         \ - \
         \EE[G_n(X_0)]
         \ - \
         G_{\infty,n}'(0)
         X_t\right) = \sum_{t=1}^{n} \mathcal{U}_n(X_t)
    \end{align*}
    from Theorem~\ref{thm:approx}(i). The bound is obtained by first decomposing the error term via projections, and then applying the martingale difference inequality (Lemma~\ref{lem:bahr}).
\begin{lemma}
  \label{lem:cond_decomp}
  Under the assumptions of Lemma~\ref{lem:2.2}, it holds that
\begin{align*}
      \mathbf{E}\left|
\sum_{t=1}^{n}
  \mathcal{U}_n(X_t)
      \right|^r
      \ \lesssim \
      \sum_{k=-\infty}^{n}
      \left(
      \sum_{t=k\lor 1}^{n}
      \norm{
      P_{k-t}\, \mathcal{U}_n(X_0)
      }_{L^r(\PP)}
      \right)^r
      \,.
  \end{align*}
\end{lemma}
At this stage it is not obvious that the right-hand side is finite, and thus that this upper bound is informative; the next step of the proof, in Lemma~\ref{lem:core_r}, turns this into an informative upper bound.
\begin{proof}
  Since the process $(X_t)$ is (strictly) stationary, we may define
  \begin{align*}
    T_n(G_n)
    \ := \ \sum_{t=1}^{n}
    \mathcal{U}_n(X_t)
    \ = \
    \sum_{t=1}^{n}
    G_n(X_t) \ -\  \EE[G_n(X_0)] \ -\  G_{\infty,n}'(0) X_t
    \,.
  \end{align*}
Because $X_t$ is $\mathcal{F}_t$-measurable and the sequence $(\mathcal{F}_k)$ is increasing, $T_n(G_n)$ is $\mathcal{F}_n$-measurable. Then
  \begin{align*}
    P_{k}\, T_n(G_n)
    \ = \
    0
    \qquad \text{for}\ k>n
    \qquad\text{and}\qquad
    P_{k}\, \mathcal{U}_n(X_t)
    \ = \ 0
    \qquad\text{for}\ k>t
    \,.
  \end{align*}
 We can therefore write, for any $\ell<n$,
  \[
    T_n(G_n)
    \ - \
    \EE[T_n(G_n)\mid \mathcal{F}_{\ell-1}]
    \ = \
    \sum_{k=\ell}^n
    P_k\, T_n(G_n).
  \]
  Since $\cap_{k\in \mathbb{Z}} \mathcal{F}_k$ is the trivial $\sigma-$algebra, $\EE[T_n(G_n)\mid \mathcal{F}_{\ell-1}]\to \EE[T_n(G_n)]=0$ as $\ell\to-\infty$. Hence, we obtain
  \begin{align}
    \label{eq:above:md}
    T_n(G_n)
    &
    \ = \
    \sum_{k=-\infty}^n
    P_{k}\, T_n(G_n)
    \,,
  \end{align}
  and then
  \begin{align}
  \label{eq:above:pk}
  P_k\, T_n(G_n) &= \sum_{t=1}^{n} P_k \, \mathcal{U}_n(X_t) = \sum_{t=k\lor 1}^{n} P_k\, \mathcal{U}_n(X_t).
    \end{align}
We aim to apply Lemma~\ref{lem:bahr} to Equation~\eqref{eq:above:md}.
To this end, define the array
\begin{align*}
  (Y_{m,i})
  \ = \
  (P_{i+n-m}\, T_n(G_n)\,\colon\, i=1,\ldots,m\,,m\in\NN)
  \,,
\end{align*}
so that
\begin{align*}
  \sum_{i=1}^m
  Y_{m,i}
  \ = \
  \sum_{k=n-m+1}^{n}
  P_{k}\,T_n(G_n)
  \,,
\end{align*}
and define the associated sequence $(Y_i)_{i\geq 1}$ by
$
  Y_i = P_{n-i+1}\, T_n(G_n)
$.
By Lemma~\ref{lem:2.2}(ii) and Identity \eqref{eq:above:pk} it holds that $(Y_{m,i}),(Y_i)\subset L^r(\PP)$.
Note that
\begin{align*}
  \sum_{i=1}^{\infty}
  Y_i
  \ = \
  \sum_{k=-\infty}^{n}
  P_k\, T_n(G_n)
  \ = \
  \lim_{m\to\infty}
  \sum_{k=n-m+1}^{n}
  P_{k}\,T_n(G_n)
  \ = \
  \lim_{m\to\infty}
  \sum_{i=1}^m
  Y_{m,i}
  \,,
\end{align*}
and similarly
$
\sum_{i=1}^{\infty}
\EE|Y_i|^r
= \lim_{m\to\infty}
\sum_{i=1}^m
\EE|Y_{m,i}|^r
$.
To apply Lemma~\ref{lem:bahr}, it remains to verify Condition \eqref{eq:bahr:cond:1}.
To this end, observe
  that Equation \eqref{eq:above:md} represents
  an $\mathcal{F}_\ell$-martingale difference decomposition
  in the sense that, for $\ell\in\ZZ$,
  \begin{align*}
    &
    \EE[P_{\ell+1}\,T_n(G_n)\mid \mathcal{F}_\ell]
    \ = \
    \EE[
    \EE[T_n(G_n)\mid \mathcal{F}_{\ell+1}]
    \mid
    \mathcal{F}_{\ell}
    ]
    \ - \
    \EE[T_n(G_n)\mid \mathcal{F}_{\ell}]
    \ = \ 0
    \,.
  \end{align*}
  Consequently
  \begin{align}
    \label{eq:above:1}
    \EE[Y_{m,\ell+ 1}\mid \mathcal{F}_{n - m + \ell}]
    \ = \
    0
    \qquad\text{for all}\ 1\le \ell \le m-1\,, m\in\NN\,.
  \end{align}
  Let
  $\mathcal{V}^{m}_\ell$ denote the $\sigma$-algebra
  generated by
  \begin{align*}
    \sum_{i=1}^{\ell}
    Y_{m,i}
    \ = \
    \sum_{i=1}^{\ell}
    P_{i+n-m}\,T_n(G_n)
    \ = \
    \sum_{k=n-m+1}^{n-m+\ell}
    P_k\,T_n(G_n)
    \,.
  \end{align*}
  Since the filtration $(\mathcal{F}_k)$ is increasing, all summands are $\mathcal{F}_{n - m + \ell}$-measurable.
  Therefore, $\mathcal{V}^m_\ell\subset \mathcal{F}_{n-m+\ell}$.
  Combined with
    \eqref{eq:above:1} this verifies Condition \eqref{eq:bahr:cond:1}.
    We can now apply Inequality \eqref{eq:bahr:infty} to obtain
    \begin{align*}
\EE
\left|
      \sum_{k=-\infty}^n P_{k} T_n(G_n)
\right|^r
      \ = \
\EE
\left|
      \sum_{i=1}^\infty Y_i
\right|^r
\ \lesssim \
      \sum_{i=1}^\infty
\EE
\left|
Y_i
\right|^r
      \ = \
      \sum_{k=-\infty}^n
\EE
\left|
      P_{k} T_n(G_n)
\right|^r
\,.
    \end{align*}
With Equation \eqref{eq:above:md} and Equation \eqref{eq:above:pk}, it follows that
  \begin{align}
    \label{eq:2}
    \begin{split}
      \EE
      \left[
      \left|T_n(G_n)\right|^r
      \right]
      &
      \ \lesssim \
      \sum_{k\le n}^{}
      \mathbf{E}\left|
      \sum_{t=k\lor 1}^{n}
      P_k\, \mathcal{U}_n(X_t)
      \,
      \right|^r
      \\
      &
      \ \le \
      \sum_{k\le n}^{}
      \left(
      \sum_{t=k\lor 1}^{n}
      \norm{
      P_{k-t}\, \mathcal{U}_n(X_0)
      }_{L^r(\PP)}
      \right)^r
      \,.
    \end{split}
  \end{align}
  The final inequality follows from the  triangle inequality for the $L^r(\PP)$ norm, using that $\norm{
      P_k\, \mathcal{U}_n(X_t)
      }_{L^r(\PP)} = \norm{
      P_{k-t}\, \mathcal{U}_n(X_0)
      }_{L^r(\PP)}$ due to stationarity.
      This completes the proof.
\end{proof}
The most delicate step now lies ahead: bounding each projection term individually. The first part of the next lemma sets the stage by identifying a suitable decomposition of the error. The core of the argument is in the second part, where we introduce auxiliary parameters $\gamma$ and $r$ to mitigate the effect of infinite second moments.
\begin{lemma}[Error Decomposition and Control]
  \label{lem:core_r}
  Let Assumptions~\ref{asu:coef},~\ref{asu:G_n} and~\ref{asu:f} hold, let $k_0$ be as defined by Lemma~\ref{lem:f}, and let $k>k_0$.
  \begin{enumerate}[label=(\roman*)]
    \item
    We can write
    $
    P_{-k}\,\mathcal{U}_n(X_0)
     =
    R_1
     +
    R_2
     +
    R_3
    $,
    where
  \begin{align*}
    R_1
    &\ := \
    G_{k-1,n}(X_0-X_{0,k-1})
    \ - \
    \int
    G_{k-1,n}(X_0-X_{0,k}+a_kz)
    \,\mathrm{d}\PP_\varepsilon(z)
    \\&
    \qquad\ - \
    a_k\cdot\varepsilon_{-k}
    \cdot
    G'_{k-1,n}(X_0-X_{0,k})
    \,,
    \\
    R_2
    &\ := \
    a_k\cdot\varepsilon_{-k}
    \left(
    G_{\infty,n}'(X_0-X_{0,k})
    \ - \
    G_{\infty,n}'(0)
    \right)
    \,,
    \\
    R_3&\ := \
    a_k\cdot\varepsilon_{-k}
    \left(
    G_{k-1,n}'(X_0-X_{0,k})
    \ - \
    G_{\infty,n}'
(X_0-X_{0,k})
    \right)
    \,.
  \end{align*}
  \item
    Let $\gamma\in(0,1)$ and $r\in[1,2]$ such that
  \begin{align*}
    \gamma_G
    \ < \
    \gamma
    \ < \
\min
\left\{
    \frac{d}{1-d}
    \,,
    1
    -
    \frac{1}{r(1-d)}
    \,,
    \frac{\alpha}{r}
    -
    1
\right\}
    \qquad\text{and}\qquad
    \frac{1}{1-d}
    \ < \
    r
    \ < \
    \alpha
    \,.
  \end{align*}
    For each $j\in\{1,2,3\}$, we have
    $
    \norm{R_j}_{L^r(\PP)}
    \
 \lesssim
    \
      k^{-(1-d)(1+\gamma)}
      \cdot
      u_n^{-\gamma+\gamma_G}
    $.
    Therefore
    \begin{align}
      \label{lem:core_r:iii:final}
      \norm{
      P_{-k}\, \mathcal{U}_n(X_0)
      }_{L^r(\PP)}
      \
      \lesssim
      \
      k^{-(1-d)(1+\gamma)}
      \cdot
      u_n^{-\gamma+\gamma_G}
      \,.
    \end{align}
   \end{enumerate}
\end{lemma}
Below and throughout, the symbol $\asymp$ means asymptotic equivalence of sequences and functions up to a constant $C\neq0$ that may differ from 1.
 \subsection{Proof of Lemma~\ref{lem:core_r}(i)}
 By definition
\begin{align*}
  &
  R_1
  \ + \
  R_2
  \ + \
  R_3
  \\&
  \ = \
    G_{k-1,n}(X_0-X_{0,k-1})
    \ - \
    \int
    G_{k-1,n}(X_0-X_{0,k}+a_kz)
    \,\mathrm{d}\PP_\varepsilon(z)
    \ - \
    a_k\cdot\varepsilon_{-k}
    \cdot
    G_{\infty,n}'
    (0)
    \,.
\end{align*}
On the other hand, using~\eqref{eq:PkU}, we have
\begin{align*}
  P_{-k}\,\mathcal{U}_n(X_0)
  &
  \ = \
    \EE[G_n(X_0)\, | \, \mathcal{F}_{-k}]
    \ - \
    \EE[G_n(X_0)\, | \, \mathcal{F}_{-(k+1)}]
    \ - \
    a_k\cdot\varepsilon_{-k}
    \cdot
    G_{\infty,n}'(0)
    \,.
\end{align*}
Note that $X_{0,k-1}=\sum_{j=0}^{k-1}a_j\varepsilon_{-j}$ is independent of $\mathcal{F}_{-k}$, and $X_0-X_{0,k-1}=\sum_{j=k}^{\infty}a_j\varepsilon_{-j}$ is $\mathcal{F}_{-k}$-measurable.
Therefore,
\begin{align*}
    \EE[G_n(X_0)\, | \, \mathcal{F}_{-k}] & \ = \ \EE[G_n(X_{0,k-1}+(X_0-X_{0,k-1}))\, | \, \mathcal{F}_{-k}]
  \\&
    \ = \
    \int
    G_n(z+(X_0-X_{0,k-1}))
    \,\mathrm{d}
  \PP_{X_{0,k-1}}(z)
  \\&
  \ = \
  G_{k-1,n}(X_0-X_{0,k-1})
  \,.
\end{align*}
Similarly, by Lemma~\ref{lem:cond},
\begin{align*}
    \EE[G_n(X_0)\, | \, \mathcal{F}_{-(k+1)}]
    \ = \
  G_{k,n}(X_0-X_{0,k})
  \ = \
    \int
    G_{k-1,n}(X_0-X_{0,k}+a_kz)
    \,\mathrm{d}\PP_\varepsilon(z)
    \,.
\end{align*}
Consequently, combining the expressions gives
\begin{align*}
  & P_{-k}\,\mathcal{U}_n(X_0) \\
    & \ = \ G_{k-1,n}(X_0-X_{0,k-1})
    \ - \
    \int
    G_{k-1,n}(X_0-X_{0,k}+a_kz)
    \,\mathrm{d}\PP_\varepsilon(z)
    \ - \
    a_k\cdot\varepsilon_{-k}
    \cdot
    G_{\infty,n}'
    (0) \\
    & \ = \  R_1
  \ + \
  R_2
  \ + \
  R_3\,,
\end{align*}
as claimed.
\subsection{Proof of Lemma~\ref{lem:core_r}(ii)}
Given the bounds on the $R_j$, the Bound~\eqref{lem:core_r:iii:final} follows directly from Part (i). We now turn to the proof of the announced bounds on the $R_j$. Fix $k>k_0$ for the remainder of the proof.
Note that
by Assumption~\ref{asu:G_n}, and for any $\gamma$ with $\gamma_G<\gamma<d/(1-d)$, we have
\begin{align}
  \label{eq:iii:1}
  \int_\RR
  G_n(x)
  g_\gamma(x)
  \,\mathrm{d}x
  \ \lesssim \
  \int_{u_n}^{\infty}
  (1+|x|)^{-(1+\gamma-\gamma_G)}
  \,\mathrm{d}x
  \
  \asymp
  \
  u_n^{-\gamma+\gamma_G}\,.
\end{align}
Moreover, the next inequalities ensure that the conditions of Lemma~\ref{lem:f} are met:
\begin{align*}
  1
  \ + \
  \gamma
  \ < \
  1
  \ + \
  \frac{d}{1-d}
  \ = \
  \frac{1}{1-d}
  \ < \
  r
  \ < \
  \alpha
  \,.
\end{align*}
\subsubsection{Analysis of $R_1$}
Write $R_1$ as
  \begin{align}
    \label{eq:r1:0}
    \begin{split}
    R_1
    &
    \ = \
    G_{k-1,n}(X_0-X_{0,k-1})
    \ - \
    \int_\RR
    G_{k-1,n}(X_0-X_{0,k}+a_kz)
    \,\mathrm{d}\PP_\varepsilon(z)
    \\&
    \qquad
    \ - \
    a_k\cdot\varepsilon_{-k}
    \cdot
    G'_{k-1,n}(X_0-X_{0,k})
    \\&
    \ = \
    \int_\RR
    [
    G_{k-1,n}(X_0-X_{0,k}+a_k\varepsilon_{-k})
    \ - \
    G_{k-1,n}(X_0-X_{0,k}+a_kz)
    \\&
    \qquad
    \ - \
    a_k\cdot\varepsilon_{-k}
    \cdot
    G'_{k-1,n}(X_0-X_{0,k})]
    \,\mathrm{d}\PP_{\varepsilon}(z)
\\&
    \ = \
    \int_\RR
    \bigg[ \left(
    \int_{a_k z}^{a_k\varepsilon_{-k}}
    G_{k-1,n}'(X_0-X_{0,k}+s)
    \,\mathrm{d}s
    \right)
    \\&
    \qquad\ - \
    a_k\cdot\varepsilon_{-k}
    \cdot
    G'_{k-1,n}(X_0-X_{0,k})
    \ + \
    a_k\cdot z
    \cdot
    G'_{k-1,n}(X_0-X_{0,k}) \bigg]
    \,\mathrm{d}\PP_{\varepsilon}(z)
\\&
    \ = \
    \int_\RR
    \left(
    \int_{a_k z}^{a_k\varepsilon_{-k}}
    [ G_{k-1,n}'(X_0-X_{0,k}+s)
    \ - \
    G_{k-1,n}'(X_0-X_{0,k}) ]
    \,\mathrm{d}s
    \right)
        \,\mathrm{d}\PP_{\varepsilon}(z)
        \,.
    \end{split}
    \end{align}
    Note that the third equality follows by applying the fundamental theorem of calculus to $G_{k-1,n}$, which is justified by
    Lemma~\ref{lem:swap}.(ii).
    Note also that adding the term
\begin{align*}
    a_k\cdot z
    \cdot
    G'_{k-1,n}(X_0-X_{0,k})
\end{align*}
inside the integral is valid because $\varepsilon$ is centered.
To proceed, we consider three mutually exclusive cases:
\begin{itemize}
  \item
  $
      1
      \
      \ge
      \
      (
      |a_k\varepsilon_{-k}|
      \lor
      |a_kz|
      )
  $
  \item
  $
|a_k\varepsilon_{-k}|
      \
      >
      \
\left(
      |a_kz|
      \lor
      1
\right)
  $
  \item
  $
      |a_kz|
      \
      >
      \
\left(
      |a_k\varepsilon_{-k}|
      \lor
      1
\right)
  $
\end{itemize}
These three cases exhaust all possibilities and will be treated separately in the following sections.
  \paragraph{Case $1\ge (|a_k\varepsilon_{-k}|\lor |a_k z|)$}

  We bound $R_1$ as follows:
    \begin{equation}
    \label{eq:r1:0_bound}
  |R_1|
  \ \le \
\int_\RR
\left|
    \int_{a_kz}^{a_k\varepsilon_{-k}}
\left|
    G_{k-1,n}'(X_0-X_{0,k}+s)
    \ - \
    G_{k-1,n}'(X_0-X_{0,k})
\right|
    \,\mathrm{d}s
\right|
        \,
    \mathrm{d}\PP_\varepsilon(z)\,.
    \end{equation}
   Let $s$
   lie in the interval linking $a_kz$ to $a_k\varepsilon_{-k}$, so that $|s|\le 1$.
   Let $\gamma\in(\gamma_G,1)$ be as in Lemma~\ref{lem:f}.
    We can use \eqref{lem:f:2} there
    and~\eqref{lem:5.1:1} in Lemma~\ref{lem:5.1} to get for all $x\in\RR$
\begin{align}
  \label{eq:r1:0b}
  \begin{split}
  &
\left|
    f_{k-1}'(x-(X_0-X_{0,k}+s))
    \ - \
    f_{k-1}'(x-(X_0-X_{0,k}))
\right|
    \\
    &
    \ \lesssim \
     |s|
    \cdot
    g_\gamma(x-[X_0-X_{0,k}])
    \ \lesssim \
    |s|^{\gamma}
    \cdot
    g_\gamma(x)
    \cdot
    (1  +   |X_0-X_{0,k}|)^{1+\gamma}
    \,.
  \end{split}
\end{align}
Combining \eqref{eq:r1:0_bound} with
the representation of $G_{k-1,n}'$ from
Lemma~\ref{lem:swap}.(ii) and then the bound from \eqref{eq:r1:0b}, we bound $R_1$ as
    \[
    |R_1|
    \leq
    (1  +   |X_0-X_{0,k}|)^{1+\gamma}
    \cdot
    \int_\RR
    G_n(x)
    g_\gamma(x)
    \,\mathrm{d}x
    \cdot
    \int_\RR
    \left|
    \int_{a_kz}^{a_k\varepsilon_{-k}}
    |s|^{\gamma}
    \,\mathrm{d}s
    \right|
        \,
    \mathrm{d}\PP_\varepsilon(z)
    \,.
    \]
    Finally, by~\eqref{eq:iii:1} and using $\varepsilon\in L^{1+\gamma}(\PP)$ as $1+\gamma < \alpha$, we conclude
    \begin{align}
  \label{eq:r1:2}
      \begin{split}
        |R_1|
    &
    \ \lesssim \
    (1  +   |X_0-X_{0,k}|)^{1+\gamma}
    \cdot
    u_n^{-\gamma+\gamma_G}
    \cdot
    |a_k|^{1+\gamma}
    \cdot
    \left(
    |\varepsilon_{-k}|^{1+\gamma}
    \ + \
    \EE
    |\varepsilon|^{1+\gamma}
    \right)
    \\&
    \ \lesssim \
    (1  +   |X_0-X_{0,k}|^{1+\gamma})
    \cdot
    u_n^{-\gamma+\gamma_G}
    \cdot
    |a_k|^{1+\gamma}
    \cdot
    [|\varepsilon_{-k}|^{1+\gamma} \vee 1]
    \,.
  \end{split}
\end{align}
We now show that
\begin{align}
  \label{eq:r2:x0bdd}
  \EE|X_0-X_{0,k}|^{p}
  \ \lesssim \
  k^{1-(1-d)p}
  \ \lesssim \
  1
  \qquad\text{for}\
p\in
\left(
\frac{1}{1-d}
\,,
\alpha
\right)
  \,.
\end{align}
To this end, define the array and limiting sequence
\begin{align*}
  (Y_{m,i})
  \ = \
  (a_{k+i}\varepsilon_{-(k+i)}\,,i=1,\ldots,m\,,m\in\NN)
  \qquad\text{and}\qquad
  (Y_i)
  \ = \
  (a_{k+i}\varepsilon_{-(k+i)})_{i\in\NN}
  \,.
\end{align*}
Note that \eqref{eq:bahr:cond:1} and \eqref{eq:bahr:cond:2} are satisfied for $(Y_{m,i})$ and $(Y_i)$.
Then we may write
\begin{align*}
  X_0
   -
  X_{0,k}
  \ = \
\sum_{i=k+1}^\infty
  a_i\varepsilon_{-i}
  \ = \
\sum_{i=1}^\infty
  a_{k+i}\varepsilon_{-(k+i)}
  \ = \
\sum_{i=1}^\infty
Y_i
  \,.
\end{align*}
The $\varepsilon_j$ are independent, centered and belong to $L^p(\PP)$, so applying Inequality~\eqref{eq:bahr:infty}
in Lemma~\ref{lem:bahr}.(ii), we obtain
\begin{align*}
  \EE|X_0-X_{0,k}|^p
  \ \lesssim \
  \sum_{i=k+1}^\infty
  |a_i|^{p}
  \EE|\varepsilon_{-i}|^p
  \ \lesssim \
  \sum_{i=k+1}^\infty
  i^{-(1-d)p}
  \ \lesssim \
  k^{1-(1-d)p}
  \ \le \
  1
  \,,
\end{align*}
where we used $|a_i|\asymp i^{-(1-d)}$ and $(1-d)p>1$. This proves \eqref{eq:r2:x0bdd}. To conclude, note that
\begin{align*}
\frac{1}{1-d}
\ < \
  (1+\gamma)r
\ < \
\left(
1
\ + \
\frac{\alpha}{r}
-
1
\right)
r
  \ = \
  \alpha
  \,.
\end{align*}
Since $X_0-X_{0,k}$ and $\varepsilon_{-k}$ are independent, and using the moment bound from~\eqref{eq:r2:x0bdd} with $p=(1+\gamma)r<\alpha$ along with $|a_i|\asymp i^{-(1-d)}$, we obtain
\begin{align*}
  \norm{
    (1  +   |X_0-X_{0,k}|)^{1+\gamma}
    \cdot
    u_n^{-\gamma+\gamma_G}
    \cdot
    |a_k|^{1+\gamma}
    \cdot
    (1\lor|\varepsilon_{-k}|)^{1+\gamma}
  }_{L^r(\PP)}
  \ \lesssim \
    u_n^{-\gamma+\gamma_G}
    \cdot
    k^{-(1-d)(1+\gamma)}
    \,.
\end{align*}
Combining this with the bound from
\eqref{eq:r1:2}, we conclude
\begin{align*}
  \norm{
  R_1
  \ind\{1\ge (|a_kz|\lor|a_k\varepsilon_{-k})\}
  }_{L^r(\PP)}
  \ \lesssim \
    k^{-(1-d)(1+\gamma)}
    u_n^{-\gamma+\gamma_G}
    \qquad
    \text{for $k>k_0+1$}\,.
\end{align*}
\paragraph{Case $|a_k\varepsilon_{-k}|>(1\lor |a_k z|)$}
We now estimate $R_1$ restricted to the event $
|a_k\varepsilon_{-k}|>(1\lor |a_k z|)
$.
To do so, we note that on this event $|\varepsilon_{-k}| > |z|$ and we split the integrand in \eqref{eq:r1:0}, obtaining
\begin{align}
  \label{eq:r1:3}
  \begin{split}
  &
  |R_1\ind\{
|a_k\varepsilon_{-k}|>(1\lor |a_k z|)
  \}|
  \\&
  \ \le \
  \ind\{|a_k\varepsilon_{-k}|>1\}
\left|
  \int_\RR
  \left(
  \int_{a_k z}^{a_k \varepsilon_{-k}}
  G_{k-1,n}'(X_0-X_{0,k}+s)
    \ind\{|\varepsilon_{-k}|>|z|\}
  \,\mathrm{d}s
  \right)
  \mathrm{d}\PP_\varepsilon(z)
  \right|
  \\&
  \ + \
  \ind\{|a_k\varepsilon_{-k}|>1\}
\left|
  \int_\RR
  \left(
  \int_{a_k z}^{a_k \varepsilon_{-k}}
  G_{k-1,n}'(X_0-X_{0,k})
    \ind\{|\varepsilon_{-k}|>|z|\}
  \,\mathrm{d}s
  \right)
  \mathrm{d}\PP_\varepsilon(z)
  \right|
  \,,
  \end{split}
\end{align}
and bound each term separately.
To bound the first term in \eqref{eq:r1:3}, we note that necessarily $|s|<|a_k \varepsilon_{-k}|$ and we use successively Lemma~\ref{lem:swap}.(ii), the bound $|f_{k-1}'(x)|\lesssim g_{\gamma}(x)$ (from Lemma~\ref{lem:f}), Inequality~\eqref{lem:5.1:2} from Lemma~\ref{lem:5.1}, Inequality~\eqref{lem:5.1:1} from Lemma~\ref{lem:5.1}, and~\eqref{eq:iii:1}, resulting in
\begin{align*}
  &
  \ind\{|a_k\varepsilon_{-k}|>1\}
\left|
  \int_\RR
  \left(
  \int_{a_k z}^{a_k \varepsilon_{-k}}
  G_{k-1,n}'(X_0-X_{0,k}+s)
    \ind\{|\varepsilon_{-k}|>|z|\}
  \,\mathrm{d}s
  \right)
  \mathrm{d}\PP_\varepsilon(z)
  \right| \\
  &
\ \le \
  \ind\{|a_k\varepsilon_{-k}|>1\}
  \int_{|s|\le |a_k \varepsilon_{-k}|}
  \left(
  \int_\RR
  G_n(x)
  |
  f_{k-1}'
  (x-(X_0-X_{0,k}+s))
  |
  \,\mathrm{d}x
  \right)
  \,\mathrm{d}s \\
  &
  \ \lesssim \
  \ind\{|a_k\varepsilon_{-k}|>1\}
  \int_\RR
  G_n(x)
  \left(
  \int_{|s|\le |a_k \varepsilon_{-k}|}
  g_{\gamma}((x-(X_0-X_{0,k}))- s)
  \,\mathrm{d}s
  \right)
  \,\mathrm{d}x \\
  &
  \ \lesssim \
  |a_k\varepsilon_{-k}|^{1+\gamma}
  \int_{\RR}
  G_n(x)
  g_{\gamma}(x-(X_0-X_{0,k}))
  \,\mathrm{d}x \\
  &
  \ \lesssim \
  |a_k\varepsilon_{-k}|^{1+\gamma}
  (1+ |X_0-X_{0,k}|^{1+\gamma})
  \int_{\RR}
  G_n(x) g_{\gamma}(x)
  \,\mathrm{d}x \\
  &
  \ \lesssim \
  |a_k \varepsilon_{-k}|^{1+\gamma}
  (1+ |X_0-X_{0,k}|^{1+\gamma})
  \cdot
  u_n^{-\gamma+\gamma_G}
  \,.
\end{align*}
We now similarly estimate the second term in \eqref{eq:r1:3} using that $G_{k-1,n}'(X_0-X_{0,k})$ is constant in $s$
and we obtain
\begin{align*}
  &
  \ind\{|a_k\varepsilon_{-k}|>1\}
\left|
  \int_\RR
  \left(
  \int_{a_k z}^{a_k \varepsilon_{-k}}
  G_{k-1,n}'(X_0-X_{0,k})
    \ind\{|\varepsilon_{-k}|>|z|\}
  \,\mathrm{d}s
  \right)
  \mathrm{d}\PP_\varepsilon(z)
  \right|
\\
    &
\ \le \
  \ind\{|a_k\varepsilon_{-k}|>1\}
  \int_{|s|\le |a_k \varepsilon_{-k}|}
  \left(
  \int_\RR
  G_n(x)
  |
  f_{k-1}'
  (x-(X_0-X_{0,k}))
  |
  \,\mathrm{d}x
  \right)
  \,\mathrm{d}s
    \\
  &
  \ \lesssim \
  \ind\{|a_k\varepsilon_{-k}|>1\}
  |a_k \varepsilon_{-k}|
  \int_\RR
  G_n(x)
  g_{\gamma}(x-(X_0-X_{0,k}))
  \,\mathrm{d}x
  \\
  &
  \ \lesssim \
  |a_k \varepsilon_{-k}|^{1+\gamma}
  (1+ |X_0-X_{0,k}|^{1+\gamma})
  \cdot
    u_n^{-\gamma+\gamma_G}
  \,,
\end{align*}
where, in the last step, we also used that $
|a_k\varepsilon_{-k}|
\le
|a_k\varepsilon_{-k}|^{1+\gamma}
$ due to $|a_k\varepsilon_{-k}|>1$.
Combining the bounds for both terms and applying the moment bound from \eqref{eq:r2:x0bdd} and $|a_k| \sim k^{-(1-d)}$, we obtain
\begin{align*}
  \norm{R_1\ind\{
  |a_k\varepsilon_{-k}|> (1\lor|a_k z| )
  \}}_{L^r(\PP)}
  \ \lesssim \
  k^{-(1+\gamma)(1-d)}u_n^{-\gamma+\gamma_G}
\,.
\end{align*}
\paragraph{Case $|a_kz|>(1\lor |a_k\varepsilon_{-k}|)$}
As in the previous case we split the integrand in \eqref{eq:r1:0}:
\begin{align}
  \label{eq:r1:6}
  \begin{split}
  &
  |R_1\ind\{
|a_kz|>(1\lor |a_k \varepsilon_{-k}|)
  \}|
  \\&
  \ \le \
\left|
  \int_\RR
  \ind\{|a_kz|>1\}
  \left(
  \int_{a_k z}^{a_k \varepsilon_{-k}}
  G_{k-1,n}'(X_0-X_{0,k}+s)
    \ind\{|z|>|\varepsilon_{-k}|\}
  \,\mathrm{d}s
  \right)
  \mathrm{d}\PP_\varepsilon(z)
  \right|
  \\&
  \ + \
\left|
  \int_\RR
  \ind\{|a_kz|>1\}
  \left(
  \int_{a_k z}^{a_k \varepsilon_{-k}}
  G_{k-1,n}'(X_0-X_{0,k})
    \ind\{|z|>|\varepsilon_{-k}|\}
  \,\mathrm{d}s
  \right)
  \mathrm{d}\PP_\varepsilon(z)
  \right|
  \,.
  \end{split}
\end{align}
We now bound the first term in \eqref{eq:r1:6}. The argument proceeds as in the previous case: we apply Lemma~\ref{lem:swap}.(ii) and estimate the integrals using the smoothness and decay properties of the derivative of the density $f'_{k-1}$. We find:
\begin{align*}
  &
\left|
  \int_\RR
  \ind\{|a_kz|>1\}
  \left(
  \int_{a_k z}^{a_k \varepsilon_{-k}}
  G_{k-1,n}'(X_0-X_{0,k}+s)
    \ind\{|z|>|\varepsilon_{-k}|\}
  \,\mathrm{d}s
  \right)
  \mathrm{d}\PP_\varepsilon(z)
  \right|
\\
  &
\ \le \
  \int_{\RR}
  \ind\{|a_kz|>1\}
  \left(
  \int_\RR
  \left(
  \int_{|s|\le |a_k z|}
  G_n(x)
  |
  f_{k-1}'
  (x-(X_0-X_{0,k}+s))
  |
  \,\mathrm{d}s
  \right)
  \,\mathrm{d}x
  \right)
  \,\mathrm{d}\PP_{\varepsilon}(z)
    \\
  &
  \ \lesssim \
  \int_{\RR}
  G_n(x)
  \left(
  \int_\RR
  \ind\{|a_kz|>1\}
  \left(
  \int_{|s|\le |a_k z|}
  g_{\gamma}((x-(X_0-X_{0,k}))-s)
  \,\mathrm{d}s
  \right)
  \,\mathrm{d}\PP_{\varepsilon}(z)
  \right)
  \,\mathrm{d}x
  \\&
  \ \lesssim \
  \left(
  \int_{\RR}
  G_n(x)
  g_{\gamma}(x-(X_0-X_{0,k}))
  \,\mathrm{d}x
  \right)
  \left(
  \int_{\RR}
  |a_kz|^{1+\gamma}
  \,\mathrm{d}\PP_{\varepsilon}(z)
  \right)
  \\&
  \ \lesssim \
  |a_k|^{1+\gamma}
  \EE|\varepsilon|^{1+\gamma}
  (1+ |X_0-X_{0,k}|)^{1+\gamma}
  \int_{\RR}
  G_n(x)
  g_{\gamma}(x)
  \,\mathrm{d}x
  \\&
  \ \lesssim \
  |a_k|^{1+\gamma}
  (1+ |X_0-X_{0,k}|)^{1+\gamma}
  u_n^{-\gamma+\gamma_G}
  \,.
\end{align*}
We now bound the second term in \eqref{eq:r1:6}, where $G_{k-1,n}'(X_0-X_{0,k})$ is constant in $s$:
\begin{align*}
  &
\left|
  \int_\RR
  \ind\{|a_kz|>1\}
  \left(
  \int_{a_k z}^{a_k \varepsilon_{-k}}
  G_{k-1,n}'(X_0-X_{0,k})
    \ind\{|z|>|\varepsilon_{-k}|\}
  \,\mathrm{d}s
  \right)
  \mathrm{d}\PP_\varepsilon(z)
  \right|
\\
    &
\ \le \
  \int_{\RR}
  \ind\{|a_kz|>1\}
  \left(
  \int_{|s|\le |a_k z|}
  \left(
  \int_\RR
  G_n(x)
  |
  f_{k-1}'
  (x-(X_0-X_{0,k}))
  |
  \,\mathrm{d}x
  \right)
  \,\mathrm{d}s
  \right)
  \mathrm{d}\PP_\varepsilon(z)
    \\
  &
  \ \lesssim \
  \left(
  \int_{|a_kz|>1}
  |a_k z|
  \mathrm{d}\PP_\varepsilon(z)
  \right)
  \left(
  \int_\RR
  G_n(x)
  g_{\gamma}(x-(X_0-X_{0,k}))
  \,\mathrm{d}x
  \right)
  \\
  &
  \ \lesssim \
  \left(
  \int_{|a_kz|>1}
  |a_k z|^{1+\gamma}
  \mathrm{d}\PP_\varepsilon(z)
  \right)
  (1+ |X_0-X_{0,k}|)^{1+\gamma}
  \left(
  \int_\RR
  G_n(x)
  g_{\gamma}(x)
  \,\mathrm{d}x
  \right)
  \\&
  \ \lesssim \
  |a_k|^{1+\gamma}
  (1+ |X_0-X_{0,k}|)^{1+\gamma}
  u_n^{-\gamma+\gamma_G}
      \,.
\end{align*}
Combining both bounds with the moment bound in \eqref{eq:r2:x0bdd}, we conclude analogously to the previous case that
\begin{align*}
  \norm{R_1\ind\{
  |a_kz|> (1\lor|a_k \varepsilon_{-k}| )
  \}}_{L^r(\PP)}
  \ \lesssim \
  k^{-(1+\gamma)(1-d)}u_n^{-\gamma+\gamma_G}
\,.
\end{align*}
This completes the proof of the control of $R_1$.
\subsubsection{Analysis of $R_2$}
We begin with the estimate
\begin{align}
  \label{eq:r2:11}
  \begin{split}
  &
      |
      f_{\infty}'(x-(X_0-X_{0,k}))
      \ - \
      f_{\infty}'(x)
      |
      \\&
      \ \le \
      |
      f_{\infty}'(x-(X_0-X_{0,k}))
      \ - \
      f_{\infty}'(x)
      |
      \cdot
    \ind\{|X_0 - X_{0,k}|\le 1\}
      \\&
      + \
      \left(
      |
      f_{\infty}'(x-(X_0-X_{0,k}))
      |
      \ + \
      |
      f_{\infty}'(x)
      |
      \right)
    \ind\{|X_0 - X_{0,k}|> 1\}
    \,.
  \end{split}
\end{align}
Note that by Inequality \eqref{lem:f:2} from Lemma~\ref{lem:f}, we have
\begin{align}
  \label{eq:r2:12}
  \begin{split}
  &
      |
      f_{\infty}'(x-(X_0-X_{0,k}))
      \ - \
      f_{\infty}'(x)
      |
      \cdot
    \ind\{|X_0 - X_{0,k}|\le 1\}
    \\&
    \ \lesssim \
      |X_0-X_{0,k}|
      \cdot
      g_\gamma(x)
      \,.
  \end{split}
      \intertext{
  Moreover, using~\eqref{lem:f:1} from Lemma~\ref{lem:f} and~\eqref{lem:5.1:1} from Lemma~\ref{lem:5.1}, it follows that
      }
      \label{eq:r2:13}
      \begin{split}
      &
      \left(
      |
      f_{\infty}'(x-(X_0-X_{0,k}))
      |
      \ + \
      |
      f_{\infty}'(x)
      |
      \right)
    \ind\{|X_0 - X_{0,k}|> 1\}
    \\&
            \ \lesssim \
            \left(
      g_\gamma
      (x-(X_0-X_{0,k})) + g_\gamma
      (x)
            \right)
    \ind\{|X_0 - X_{0,k}|> 1\}
    \\&
    \ \lesssim \
      g_\gamma(x)
      \cdot
      |X_0-X_{0,k}|^{1+\gamma}
      \,.
      \end{split}
\end{align}
Therefore, combining Inequalities \eqref{eq:r2:11}-\eqref{eq:r2:13}, we get
\begin{align}
  \label{eq:r2:14}
      |
      f_{\infty}'(x-(X_0-X_{0,k}))
      \ - \
      f_{\infty}'(x)
      |
      \ \lesssim \
      g_{\gamma}(x)
      \left(
      |X_0-X_{0,k}|
      \lor
      |X_0-X_{0,k}|^{1+\gamma}
      \right).
\end{align}
Applying successively Lemma~\ref{lem:swap}.(ii), Inequality \eqref{eq:r2:14}, and~\eqref{eq:iii:1}, we obtain
    \begin{align}
      \label{eq:r2:15}
      \begin{split}
      &
      |
      G_{\infty,n}'(X_0-X_{0,k})
      \ - \
      G_{\infty,n}'(0)
      |
      \\
      &
      \ \le \
      \int_{\RR}
      G_n(x)
      \cdot
       |
      f_{\infty}'(x-(X_0-X_{0,k}))
      \ - \
      f_{\infty}'(x)
      |
      \,\mathrm{d}x
      \\&
      \ \lesssim \
      \left(
      |X_0-X_{0,k}|
      \lor
      |X_0-X_{0,k}|^{1+\gamma}
      \right)
      \int_{\RR}
      G_n(x)
      g_\gamma(x)
      \,\mathrm{d}x
      \\&
      \ \lesssim \
      \left(
      |X_0-X_{0,k}|
      \lor
      |X_0-X_{0,k}|^{1+\gamma}
      \right)
      u_n^{-\gamma + \gamma_G}
      \,.
      \end{split}
    \end{align}
    From the definition of $R_2$, together with the independence of $X_0-X_{0,k}$ and $\varepsilon_{-k}$, and Inequality \eqref{eq:r2:15}, we find
\begin{align}
  \label{eq:r2:155}
  \norm{R_2}_{L^r(\PP)}
  &
  \ \lesssim \
  k^{-(1-d)}
  \left(
  \norm{
      X_0-X_{0,k}
  }_{L^r(\PP)}
  \ + \
  \norm{
      |X_0-X_{0,k}|^{1+\gamma}
  }_{L^r(\PP)}
  \right)
  u_n^{-\gamma+\gamma_G}
  \,,
\end{align}
since $|a_k|\asymp k^{-(1-d)}$, and $\varepsilon\in L^r(\PP)$ for $r<\alpha$.
We now apply Estimate~\eqref{eq:r2:x0bdd} for both $p=r$ and $p=r(1+\gamma)$.
Since
$1/(1-d)<r<r(1+\gamma)<\alpha$ by the condition in Lemma \eqref{lem:core_r}(ii),
it follows
\begin{align}
  \label{eq:r2:16}
  \begin{split}
  \norm{
      X_0-X_{0,k}
  }_{L^r(\PP)}
  \ + \
  \norm{
      |X_0-X_{0,k}|^{1+\gamma}
  }_{L^r(\PP)}
  &
  \ \lesssim \
  k^{1/r-(1-d)}
  \ + \
  k^{1/r-(1-d)(1+\gamma)}
  \\&
  \ \lesssim \
  k^{1/r-(1-d)}
  \,.
  \end{split}
\end{align}
Note finally that
\begin{align*}
  \frac{1}{r}
  -
  2(1-d)
  &
  \ = \
  -
  (1-d)
  \left(
  1
  -
  \frac{1}{r(1-d)}
  \right)
  \ - \
  (1-d)
  \\&
  \ \le \
  -
  (1-d)
  \gamma
    \ - \
  (1-d)
  \\&
  \ = \
  -(1-d)(1+\gamma)
  \,.
\end{align*}
Combining this inequality with Inequalities \eqref{eq:r2:155} and \eqref{eq:r2:16}, we obtain
\begin{equation}
  \label{eq:r2:main_result}
  \norm{R_2}_{L^r(\PP)}
  \ \lesssim \
  k^{1/r-2(1-d)}
  u_n^{-\gamma+\gamma_G}
  \ \lesssim \
  k^{
  -(1+\gamma)(1-d)
  }
  u_n^{-\gamma+\gamma_G}
  \,.
\end{equation}
This completes the proof of the control of $R_2$.
\subsubsection{Analysis of $R_3$}
    Using Inequality \eqref{lem:f:3} from Lemma~\ref{lem:f} and Inequality~\eqref{lem:5.1:1} from Lemma~\ref{lem:5.1}, we obtain
\begin{align}
  \label{eq:r3:3}
  \begin{split}
  &
      |f_{\infty}'(x-(X_0-X_{0,k}))
      \ - \
      f_{k-1}'(x-(X_0-X_{0,k}))|
      \\
      &\lesssim \
      k^{-(1-d)+1/r}
      g_\gamma(x-(X_0-X_{0,k})) \\
      & \lesssim \ k^{-(1-d)+1/r} g_\gamma(x) (1\lor|X_0-X_{0,k}|)^{1+\gamma}
      \,.
  \end{split}
\end{align}
    Applying Lemma~\ref{lem:swap}.(ii) together with Inequality \eqref{eq:r3:3} and then~\eqref{eq:iii:1}, we obtain
    \begin{align}
      \label{eq:r3:main}
      \begin{split}
      &
      \left|
    G_{k-1,n}'(X_0-X_{0,k})
    \ - \
    G_{\infty,n}'
(X_0-X_{0,k})
      \right|
      \\&
      \ \lesssim \
      \int_\RR
      G_n(x)
      |
      f_{\infty}'(x-(X_0-X_{0,k}))
      \ - \
      f_{k-1}'(x-(X_0-X_{0,k}))
      |
      \,\mathrm{d}x
      \\&
      \ \lesssim \
      k^{-(1-d)+1/r}
      \left(
      1
      \lor
      |X_0-X_{0,k}|^{1+\gamma}
      \right)
      u_n^{-\gamma+\gamma_G}
      \,.
      \end{split}
    \end{align}
    Using the independence of $X_0-X_{0,k}$ and $\varepsilon_{-k}\stackrel{d}{=} \varepsilon \in L^r(\PP)$, together with Inequality \eqref{eq:r3:main} and the fact that $|X_0-X_{0,k}|^{1+\gamma} \in L^r(\PP)$ by \eqref{eq:r2:x0bdd}, we obtain
    \begin{align*}
    \norm{
    R_3
    }_{L^r(\PP)}
    &
    \ \lesssim \
    |a_k|
      u_n^{-\gamma+\gamma_G}
      k^{-(1-d)+1/r}
      \left(
      1
      +
      \norm{|X_0-X_{0,k}|^{1+\gamma}}_{L^r(\PP)}
      \right)
    \,.
    \end{align*}
    Since $|a_k|\asymp k^{-(1-d)}$,
    we conclude, using the same argument as in \eqref{eq:r2:main_result}, that
    \[
    \norm{
    R_3
    }_{L^r(\PP)}
    \ \lesssim \
    k^{-2(1-d)+1/r}
      u_n^{-\gamma+\gamma_G}
    \ \lesssim \
    k^{-(1-d)(1+\gamma)}
      u_n^{-\gamma+\gamma_G}
      \,.
    \]
  This completes the proof of the control of $R_3$, and therefore the proof of Lemma~\ref{lem:core_r}.(ii).

  \qed

  We have bounded the individual summands in Lemma~\ref{lem:cond_decomp}. To advance the proof
  and finally show Equation \eqref{lem:core_r:iii:final}, we now require the following lemma, which controls the rate of divergence of a convolution-type sum that arises from applying the individual bounds.
  \begin{lemma}
  \label{lem:analytic}
  Let $\gamma,r \in\RR$ satisfy
    the assumptions of Lemma~\ref{lem:core_r}(ii).
  Then for all $n\in \NN$, we have
  \begin{align*}
    \sum_{k = -\infty}^{n}
      \left(
      \sum_{t=k\lor 1}^{n}
      \left(
      1
      \land
      (t-k)
      ^{-(1-d)(1+\gamma)}
      \right)
      \right)^r
      \ \lesssim\
      n^{r+1-(1-d)(1+\gamma)r}
      \,,
    \end{align*}
    where $\lesssim$ denotes inequality up to a constant $C_{\gamma,\alpha,d}>0$, depending only on $\gamma$, $\alpha$, and $d$. We adopt the convention $1/0=\infty$.
\end{lemma}
\begin{remark}
  The above lemma actually holds with the relaxed conditions
  \begin{align*}
    0
    \ < \
    \gamma
    \ < \
    \frac{d}{1-d}
    \qquad\text{and}\qquad
    \frac{1}{1-d}
    \ < \
    r
    \,.
  \end{align*}
\end{remark}
\begin{proof}
    The sum in question can be decomposed and reindexed as
    \begin{align*}
    & \sum_{k = -\infty}^{0}
      \left(
      \sum_{t=1}^{n}
      (t-k)
      ^{-(1-d)(1+\gamma)}
      \right)^r
      \ + \
    \sum_{k = 1}^{n}
      \left(
      1
      \ + \
      \sum_{t= k+1}^{n}
      (t-k)
      ^{-(1-d)(1+\gamma)}
      \right)^r
      \\&
      \ = \
    \sum_{k = 0}^{\infty}
      \left(
      \sum_{t=k+1}^{n+k}
      t
      ^{-(1-d)(1+\gamma)}
      \right)^r
      \ + \
    \sum_{k = 1}^{n}
      \left(
      1
      \ + \
      \sum_{t= 1}^{n-k}
      t
      ^{-(1-d)(1+\gamma)}
      \right)^r
\\&
\ = \
\left(
      \sum_{t=1}^{n}
      t
      ^{-(1-d)(1+\gamma)}
      \right)^r
\ +\
    \sum_{k = n+1}^{\infty}
      \left(
      \sum_{t=k+1}^{n+k}
      t
      ^{-(1-d)(1+\gamma)}
      \right)^r
      \\&
      \ + \
    \sum_{k = 1}^{n}
    \left(
      \left(
      1
      \ + \
      \sum_{t= 1}^{n-k}
      t
      ^{-(1-d)(1+\gamma)}
      \right)^r
      \ + \
\left(
      \sum_{t=k+1}^{n+k}
      t
      ^{-(1-d)(1+\gamma)}
      \right)^r
    \right)
      \,.
    \end{align*}
    The first and second terms correspond to $k=0$ and $k>n$, respectively, in the first double sum of the previous equation.
    The third term arises from a combination of the remaining terms
    $k\in\{1,\ldots,n\}$ in the first double sum and all terms in the second double sum
    of the previous equation.
    From the assumptions on $\gamma$ and $r$, we observe that
  \begin{align*}
    (1-d)(1+\gamma) < 1
    \qquad\text{and}\qquad
    r(1-d)(1+\gamma) >1
    \,.
  \end{align*}
    We begin by estimating the first term in the decomposition above.
    Since $(1-d)(1+\gamma)<1$, we have
    \begin{align}
    \label{eq:analytic:10}
\left(
      \sum_{t=1}^{n}
      t
      ^{-(1-d)(1+\gamma)}
      \right)^r
      \ \lesssim \
      n^{r-(1-d)(1+\gamma)r}
      \ \le \
      n^{r+1-(1-d)(1+\gamma)r}
      \,.
    \end{align}
    For $k\ge 1$, again since $(1-d)(1+\gamma)<1$, we also have
    \begin{align}
      \begin{split}
      \label{eq:analytic:11}
      \sum_{t=k+1}^{n+k}
      t
      ^{-(1-d)(1+\gamma)}
      &
      \ \le \
      \min\left( n
      \cdot
      k^{-(1-d)(1+\gamma)},
      \int_k^{n+k}
      t
      ^{-(1-d)(1+\gamma)}
      \,\mathrm{d}t \right) \\
      & \ \lesssim \
      \min\left( n
      \cdot
      k^{-(1-d)(1+\gamma)},
      (n+k)^{1-(1-d)(1+\gamma)}
      \right) \,.
      \end{split}
    \end{align}
    Using the first bound in the Estimate~\eqref{eq:analytic:11} and the fact that $(1-d)(1+\gamma)r>1$, we obtain
    \begin{align*}
      &
    \sum_{k = n+1}^{\infty}
      \left(
      \sum_{t=k+1}^{n+k}
      t
      ^{-(1-d)(1+\gamma)}
      \right)^r
      \\&
    \ \lesssim \
    n^r
    \sum_{k = n+1}^{\infty}
      k^{-(1-d)(1+\gamma)r}
      \ \lesssim \
    n^r
    \int_n^\infty
      t^{-(1-d)(1+\gamma)r}
      \,\mathrm{d}t
      \ \lesssim \
      n^{r+1-(1-d)(1+\gamma)r}
      \,.
     \end{align*}
    Applying the second bound in Estimate~\eqref{eq:analytic:11},
    we obtain
    \[
    \sum_{k = 1}^{n}
\left(
      \sum_{t=k+1}^{n+k}
      t
      ^{-(1-d)(1+\gamma)}
      \right)^r
      \ \lesssim \
    \sum_{k = 1}^{n}
      (n+k)^{r-(1-d)(1+\gamma)r}
    \ \lesssim \
      n^{r+1-(1-d)(1+\gamma)r}
      \,.
    \]
    Finally,
    recalling the first bound in~\eqref{eq:analytic:10}, we obtain
    \begin{align*}
 \sum_{k = 1}^{n}
      \left(
      1
      \ + \
      \sum_{t= 1}^{n-k}
      t
      ^{-(1-d)(1+\gamma)}
      \right)^r
      \ \lesssim \
      n
      \left(
      \sum_{t=1}^n
      t^{-(1-d)(1+\gamma)}
      \right)^r
      \ \lesssim \
      n^{r+1-(1-d)(1+\gamma)r}
      \,.
    \end{align*}
    This completes the proof of Lemma~\ref{lem:analytic}.
    \end{proof}
    With all necessary tools in hand, we now turn to the final part of the proof.
\subsection{
Proof of Theorem~\ref{thm:approx}
}
    Let
    $n>k_0$.
    We are going to apply Lemmas~\ref{lem:cond_decomp}, \ref{lem:core_r}, and \ref{lem:analytic}.
    We begin the proof.
    By Lemma~\ref{lem:cond_decomp}, it follows
    \[
      \EE
      \left|
      \sum_{t=1}^n
      G_n(X_t)
      \ - \
      \EE[G_n(X_t)]
      \ - \
      G_{\infty,n}'(0)X_t
      \right|^r
      \ \lesssim \
      \sum_{k=-\infty}^{n}
      \left(
      \sum_{t=k\lor 1}^{n}
      \norm{
      P_{k-t}\, \mathcal{U}_n(X_0)
      }_{L^r(\PP)}
      \right)^r
      \,,
    \]
    where $\mathcal{U}_n(X_t):=G_n(X_t)-\EE[G_n(X_t)]-G_{\infty,n}'(0)X_t$. We split the inner sum based on whether $t-k>k_0+1$ and use the convexity of $x\mapsto x^r$ (due to $r>1$) to get
    \begin{align*}
      &\EE
      \left|
      \sum_{t=1}^n
      G_n(X_t)
      \ - \
      \EE[G_n(X_t)]
      \ - \
      G_{\infty,n}'(0)X_t
      \right|^r \\
      &
      \ \lesssim \
      \sum_{k=-\infty}^{n}
      \left(
      \sum_{
      \begin{smallmatrix}
        (1\lor k)\le t \le n\\
        t-k > k_0+1
      \end{smallmatrix}
      }
      \norm{
      P_{k-t}\, \mathcal{U}_n(X_0)
      }_{L^r(\PP)}
      \right)^{r}
      \\&
      \qquad + \
      \sum_{k=-\infty}^{n}
      \left(
      \sum_{
      \begin{smallmatrix}
        (1\lor k)\le t \le n\\
        t-k \le k_0+1
      \end{smallmatrix}
      }
      \norm{
      P_{k-t}\, \mathcal{U}_n(X_0)
      }_{L^r(\PP)}
      \right)^r
           \,.
    \end{align*}
    For the first term, we apply Lemma~\ref{lem:core_r}.(ii) and Lemma~\ref{lem:analytic}.
    This yields
    \begin{align*}
      &
      \sum_{k=-\infty}^{n}
      \left(
      \sum_{
      \begin{smallmatrix}
        (1\lor k)\le t \le n\\
        t-k > k_0+1
      \end{smallmatrix}
      }
      \norm{
      P_{k-t}\, \mathcal{U}_n(X_0)
      }_{L^r(\PP)}
      \right)^{r}
      \\&
      \ \lesssim \
      u_n^{(\gamma_G-\gamma)r}
      \sum_{k=-\infty}^{n}
      \left(
      \sum_{
      t=1\lor k
      }^n
      \left(
      1
      \land
      (t-k)^{-(1-d)(1+\gamma)}
      \right)
      \right)^{r}
      \\&
      \ \lesssim \
      u_n^{(\gamma_G-\gamma)r}
      n^{r+1-(1-d)(1+\gamma)r}
      \,.
    \end{align*}
    For the second term, note that the constraints $t\le k_0+k+1$ and $t \ge (1\lor k)$ ensure that the inner sum is empty when $k<-k_0$. Therefore, applying Lemma~\ref{lem:2.2}.(ii), we obtain for sufficiently large $n\in\NN$ that
    \begin{align*}
      &
      \sum_{k=-\infty}^{n}
      \left(
      \sum_{
      \begin{smallmatrix}
        (1\lor k)\le t \le n\\
        t-k \le k_0 +1
      \end{smallmatrix}
      }
      \norm{
      P_{k-t}\, \mathcal{U}_n(X_0)
      }_{L^r(\PP)}
      \right)^r
      \\&
 \ \lesssim \
      \left(
      \norm{G_n(X_0)}_{L^r(\PP)}
      \lor
      |G_{\infty,n}'(0)|
      \right)^{r}
      \sum_{k=-k_0}^{n}
      \left(
      \#
      \{
      t\in\NN
      \,\colon\,
        (1\lor k)\le t \le (n\land (k+k_0+1))
      \}
      \right)^r
      \\&
 \ \leq \
      \left(
      \norm{G_n(X_0)}_{L^r(\PP)}
      \lor
      |G_{\infty,n}'(0)|
      \right)^{r}
      \sum_{k=-k_0}^{n}
      \left(
      \#
      \{
      t\in\NN
      \,\colon\,
        k\le t \le k+k_0+1
      \}
      \right)^r
      \\&
      \ = \
      \left(
      \norm{G_n(X_0)}_{L^r(\PP)}
      \lor
      |G_{\infty,n}'(0)|
      \right)^{r}
      (n+k_0+1)(k_0+2)^r
      \\&
      \ \lesssim \
      n \left(
      \EE|G_n(X_0)|^r
      \lor
      |G_{\infty,n}'(0)|^r
      \right)
      \,.
    \end{align*}
    We conclude that
    \begin{align*}
      \EE
      \left|\sum_{t=1}^n \mathcal{U}_n(X_t)\right|^r
      \ \lesssim \
      u_n^{(\gamma_G-\gamma)r}n^{r+1 - (1-d)(1+\gamma)r}
      \ + \ n
      \left(
      \EE|G_n(X_0)|^r
      \lor
      |G_{\infty,n}'(0)|^r
      \right)
    \,.
    \end{align*}
    This is the first inequality we wanted to prove. For the second inequality, observe that $1-(1-d)(1+\gamma) \ge 0$  due to $\gamma \le d/(1-d)$. Therefore
    the $n$ rate of the first term dominates.
Since $\gamma_G<\gamma$, it holds that $u_n^{\gamma_G-\gamma}<\infty$ for all $n\in\mathbb{N}$. Applying Lemma~\ref{lem:swap} gives that $\EE|G_n(X_0)|^{r}<\infty$ and $|G_{\infty,n}'(0)|<\infty$ for all $n\in\mathbb{N}$, such that we may conclude
\begin{align*}
   u_n^{(\gamma_G-\gamma)r}n^{r+1 - (1-d)(1+\gamma)r}
      \ + \ n
      \left(
      \EE|G_n(X_0)|^r
      \lor
      |G_{\infty,n}'(0)|^r
      \right)
      \ \lesssim \
   n^{r+1 - (1-d)(1+\gamma)r}
   \,.
\end{align*}
Taking power $1/r$ on both sides of the inequality finishes the proof of Theorem~\ref{thm:approx}.
    \qed

Having established a $L^r(\PP)$ moment bound in Theorem~\ref{thm:approx}, we now turn to the proof of Theorem~\ref{thm:approx_optimal} which sharpens the result by identifying the optimal rate of convergence. This requires a careful balance of the parameters $\gamma$ and $r$.
     \subsection{Proof of Theorem~\ref{thm:approx_optimal}}
     Combining Theorem~\ref{thm:approx} with the following proposition immediately gives the desired result.
     \begin{proposition}
\label{prop:optim}
 Denote
      $\kappa(\gamma,r)=1+1/r-(1-d)(1+\gamma)$.
       \begin{enumerate}[label=(\roman*)]
         \item
      There exists a pair of unique minimizers $(\gamma_0,r_0)$ of $\kappa(\gamma,r)$ over the compact domain
      \begin{align*}
        0
         \ \le \
         \gamma
         \ \le \
  \min \left\{
  \frac{d}{1-d}
  \,,
         1
          -
         \frac{1}{r(1-d)}
         \,,
         \frac{\alpha}{r}
          -
         1
  \right\}
           \qquad\text{and}\qquad
         \frac{1}{1-d}
         \ \le \
         r
         \ \le \
         \alpha
         \,,
      \end{align*}
       namely
       \begin{align*}
 \gamma_0
   :=
  \begin{cases}
    \frac{d}{1-d}
    &
    \ \text{if}\
    \frac{1}{(1-d)(1-2d)}
    < \alpha\,,\\
    \frac{\alpha(1-d)-1}{\alpha(1-d)+1}
    &
    \
    \text{else,}
  \end{cases}
  \
  r_0
   :=
  \begin{cases}
    \alpha(1-d)
    &
    \ \text{if}\
    \frac{1}{(1-d)(1-2d)}
    < \alpha\,,\\
    \frac{1}{2}
    \left(
    \frac{1}{1-d}
    \ + \ \alpha
    \right)
    &
    \
    \text{else,}
  \end{cases}
  \end{align*}
  yielding
  \begin{align*}
  \kappa_0
    \ = \
    \kappa(\gamma_0,r_0)
  \ = \
  \left\{ \begin{array}{l}
    \frac{1}{\alpha(1-d)}
    \ \text{ if } \
    \frac{1}{(1-d)(1-2d)}
    < \alpha\,,\\[10pt]
    \frac{2(1-d) + (1-\alpha(1-d)(1-2d))}{\alpha(1-d)+1} = d
    \ + \
    (1-d)
    \frac{3-\alpha(1-d)}{\alpha(1-d)+1}
    \ \text{ else.}
  \end{array} \right.
\end{align*}
\item
It holds
\begin{align*}
  \gamma_0
  \ = \
  \min \left\{
  \frac{d}{1-d}
  \,,
         1
          -
         \frac{1}{r_0(1-d)}
         \,,
         \frac{\alpha}{r_0}
          -
         1
  \right\}
  \,,\quad
  r_0
  \in
  \left(
  \frac{1}{1-d}
  \,,
  \alpha
  \right)
  \,,
  \end{align*}
  and
  \begin{align*}
  \kappa_0
  \ - \ (d+1/\alpha)
  \ < \
  0
  \,.
\end{align*}
\end{enumerate}
\end{proposition}
\begin{proof}
  \textbf{Proof of Part (i):}
     The goal is to
     minimize
     \begin{align*}
  \kappa(\gamma,r)
  \ = \
  1
  \ + \
  \frac{1}{r}
  \ - \
  (1-d)(1+\gamma)
\end{align*}
over $\gamma$ and $r$ from the
closure of the range in Theorem~\ref{thm:approx}(i), namely such that
\begin{align*}
  0
         \ \le \
         \gamma
         \ \le \
         \gamma(r)
         \qquad\text{and}\qquad
         \frac{1}{1-d}
         \ \le \
         r
         \ \le \
         \alpha
         \,,
\end{align*}
where the function $\gamma(r)$ is defined as
\begin{align*}
   \gamma(r)
   \ := \
         \gamma_{1}(r)
         \
         \land
         \
         \gamma_{2}(r)
         \
         \land
         \
         \gamma_{3}(r)
         \,,
 \end{align*}
 with
 \begin{align*}
   \gamma_1(r)
   \ := \
   \frac{d}{1-d}
   \,,
   \qquad
   \gamma_2(r)
   \ := \
   1
   -
   \frac{1}{r(1-d)}
   \,,
   \qquad
   \gamma_3(r)
   \ := \
   \frac{\alpha}{r}
   -
   1
   \,.
 \end{align*}
The function $\kappa$ is continuous and strictly decreasing in both $\gamma$ and $r$, which implies that
\[
  (\gamma_0,r_0)
  \
  =
  \
  \mathrm{argmin} \left\{ \kappa(\gamma,r)
  \, \Bigg| \,
  (\gamma,r)\in K
  \right\}
  \quad
  \text{with}
  \quad
   K:=  \left\{
  (\gamma(r),r)
  \, \Bigg| \,
  \frac{1}{1-d}
  \ \le \
  r
  \ \le \
  \alpha
  \right\}
  \,.
\]
Moreover, for any admissible value of $r$, the value of $\kappa(\gamma(r),r)$ is determined by the smallest among $\gamma_1(r),\gamma_2(r),\gamma_3(r)$. Specifically,
\begin{align*}
  \kappa(\gamma(r),r)
  \ = \
  \begin{cases}
    \kappa(\gamma_{1}(r),r)\,,&\qquad \gamma_{1}(r)\le (\gamma_2(r)\land \gamma_{3}(r))\,,\\
    \kappa(\gamma_{2}(r),r)\,,&\qquad \gamma_{2}(r)\le (\gamma_1(r)\land \gamma_{3}(r))\,,\\
    \kappa(\gamma_{3}(r),r)\,,&\qquad \gamma_{3}(r)\le (\gamma_1(r)\land \gamma_{2}(r))\,.\\
  \end{cases}
\end{align*}
Note that $r\mapsto \kappa(\gamma(r),r)$ is not differentiable everywhere. We shall therefore argue that on the set where it is differentiable the derivative does not vanish and therefore the optimal values have to lie on the boundary.
Note that this set is essentially
\begin{align*}
\left\{ r>0 \, \mid \, (\gamma_1(r),\gamma_2(r),\gamma_3(r))\  \mbox{are pairwise distinct} \right\},
\end{align*}
and it is nonempty because $\alpha>1/(1-d)$, and open since each $\gamma_j(r)$ is continuous. Next, we compute the derivative of $r\mapsto \kappa(\gamma(r),r)$ on this set.
 Clearly, for all $r>0$ it holds that
\begin{align*}
   \gamma_{1}'(r)
   \ = \
   0
   \,,
   \qquad
   \gamma_{2}'(r)
   \ = \
   \frac{1}{r^2(1-d)}
   \,,
   \qquad
   \gamma_{3}'(r)
   \ = \
   -
   \frac{\alpha}{r^2}
   \,.
 \end{align*}
 Applying the chain rule yields
 \begin{align*}
   &
   \frac{\mathrm{d}}{\mathrm{d}r}
   \kappa(\gamma_{j}(r),r)
   \ = \
   \gamma_j'(r)
  \frac{\partial}{\partial \gamma}
   \kappa(\gamma_j(r),r)
   \ + \
  \frac{\partial}{\partial r}
   \kappa(\gamma_j(r),r)
   \\&
   \ = \
   -
   \gamma_j'(r)
   (1-d)
   \ - \
\frac{1}{r^2}
   \,.
 \end{align*}
 Setting the derivative equal to zero gives
 \begin{align*}
   \gamma_j'(r)
   \ = \
   -
   \frac{1}{r^2(1-d)}
   \,.
 \end{align*}
 This equation clearly has no solution for $j=1,2$.
 For $j=3$, solving yields $\alpha=1/(1-d)$, which contradicts our assumption that $\alpha>1/(1-d)$.
 Therefore, we conclude that
 for all $r>0$ it holds that
 \begin{align*}
   \frac{\mathrm{d}}{\mathrm{d}r}
   \kappa(\gamma_{j}(r),r)
   \
   \neq
   \
   0
   \qquad\text{for}\ j=1,2,3
   \,.
 \end{align*}
To complete the analysis, we examine the behavior of $\kappa(\gamma(r),r)$ at values of $r$ where two or more of the functions $\gamma_j(r)$ coincide.

We begin by identifying the intersection points of the functions $\gamma_1(r),\gamma_2(r)$, and $\gamma_3(r)$.
Note that $\gamma_1$ is constant, $\gamma_2$ is strictly increasing, and $\gamma_3$ is strictly decreasing. Hence, each pair of functions can intersect at most once. These intersections occur at
\begin{align*}
\gamma_1
\
  =
\
  \gamma_2(r)
  \ = \
  \frac{d}{1-d}
  \qquad\text{if and only if}\qquad
  r
  \ = \
  r_{1,2}
  \ := \
  \frac{1}{1-2d}
  \,,
\end{align*}
\begin{align*}
\gamma_1
\
  =
\
  \gamma_3(r)
  \ = \
  \frac{d}{1-d}
  \qquad\text{if and only if}\qquad
  r
  \ = \
  r_{1,3}
  \ := \
  \alpha(1-d)
  \,,
\end{align*}
\begin{align*}
\gamma_2(r)
\
  =
\
  \gamma_3(r)
  \ = \
  \frac{\alpha(1-d)-1}{\alpha(1-d)+1}
  \qquad\text{if and only if}\qquad
  r
  \ = \
  r_{2,3}
  \ := \
  \frac{1}{2}
  \left(
  \frac{1}{1-d}
  \ + \
  \alpha
  \right)
  \,.
\end{align*}
We next need to calculate $\kappa(\gamma(r),r)$ at $r \in \{r_{1,2},r_{1,3},r_{2,3}\}$, which requires computing $\gamma(r)=\gamma_1\wedge\gamma_2(r)\wedge\gamma_3(r)$ at these points. To determine whether $\gamma_2(r_{2,3})$ lies above or below $\gamma_1$, we examine the sign of the difference:
\begin{align*}
  \mathrm{sgn}
  \left(
  \gamma_2(r_{2,3})
  \ - \
  \frac{d}{1-d}
  \right)
  \ = \
  \mathrm{sgn}
  \left(
  \alpha(1-d)(1-2d)
  \ - \ 1
  \right)
  \,.
\end{align*}
Thus,
\begin{align*}
  \gamma_2(r_{2,3})
  \ > \
  \frac{d}{1-d}
  \ = \
  \gamma_1
  \qquad\text{if and only if}\qquad
  \alpha
  \ > \
  \frac{1}{(1-d)(1-2d)}
  \,.
\end{align*}
We now distinguish between the two cases:
\begin{itemize}
  \item
$\alpha(1-d)(1-2d)> 1$
\item
$\alpha(1-d)(1-2d)\leq 1$
\end{itemize}
In the first case, $\alpha(1-d)(1-2d)>1$, the ordering of the intersection points is
\begin{align*}
   r_{1,2}
  \ < \
  r_{2,3}
  \ < \
  r_{1,3}
  \,,
\end{align*}
and we therefore have $\gamma(r)=d/(1-d)$ for all $r\in\{r_{1,2},r_{2,3},r_{1,3}\}$ due to the monotonicity properties of $\gamma_2$ and $\gamma_3$, see the left panel of Figure~\ref{fig:optimal}.
\begin{figure}[h]
    \centering
    \includegraphics[width=0.95\textwidth]{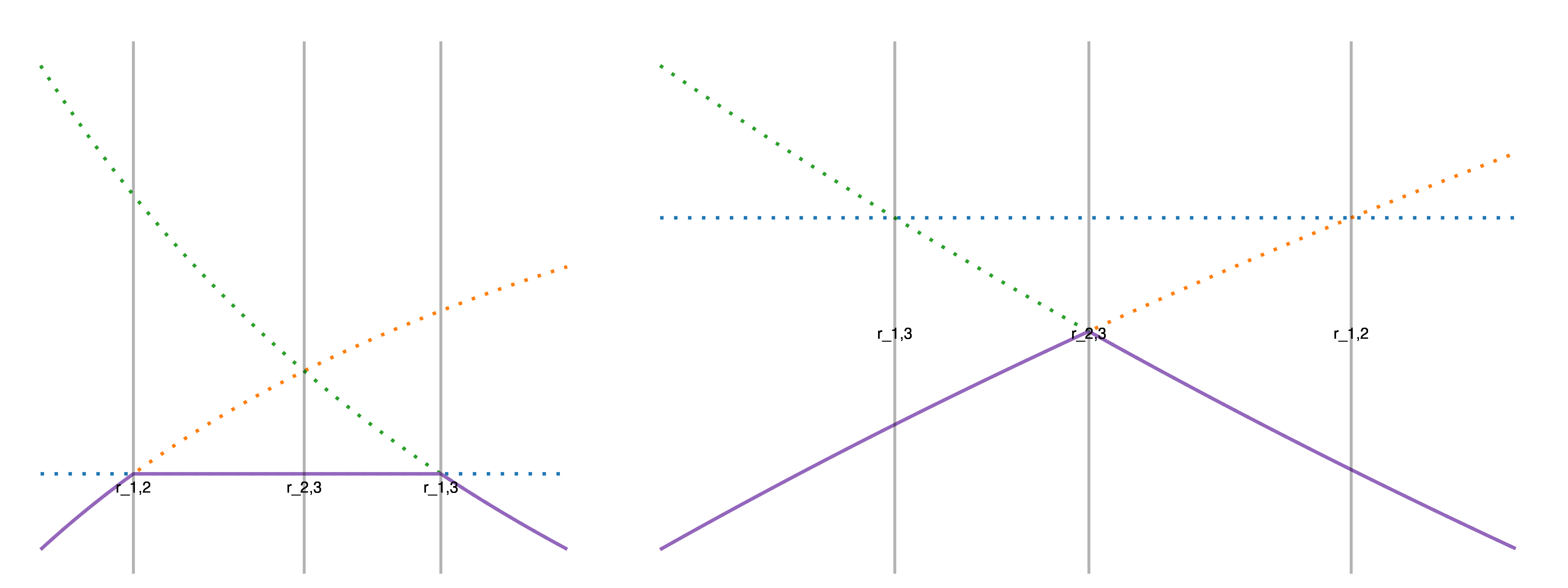}
    \caption{General position of the curves representing the functions $\gamma_1$ (blue dotted line), $\gamma_2$ (orange dotted line) and $\gamma_3$ (green dotted line) in the cases $\alpha(1-d)(1-2d)> 1$ (left panel) and $\alpha(1-d)(1-2d)\leq 1$ (right panel). The purple curve represents the function $\gamma_1\wedge \gamma_2 \wedge \gamma_3$.
    }
  \label{fig:optimal}
\end{figure}
Since $\kappa$ is decreasing in both $\gamma$~and~$r$,
the optimal value is attained at the largest feasible $r$, which in this case is $r_{1,3}$, provided it lies within the admissible interval.
Note that the inequality chain
\begin{align*}
 \frac{1}{1-d}
 \ < \
  \alpha(1-2d)
  \ < \
  \alpha(1-d)
  \ = \
  r_{1,3}
  \ < \
  \alpha
  \,,
\end{align*}
confirms that $r_{1,3}$ lies within the range $(1/(1-d),\alpha)$, and is therefore an admissible choice for minimizing $\kappa$ over $K$.
We conclude that, when $\alpha(1-d)(1-2d)>1$, the optimal choice is $(\gamma_0,r_0)=(d/(1-d),\alpha(1-d))$.
If on the contrary $\alpha(1-d)(1-2d)\le 1$, we have the ordering
\begin{align*}
   r_{1,3}
  \ \le \
  r_{2,3}
  \ \le \
  r_{1,2}
  \qquad\text{and}\qquad
  \gamma(r_{1,2})
  ,
  \gamma(r_{1,3})
  \ \le \
  \gamma(r_{2,3})
\end{align*}
as the right panel of Figure~\ref{fig:optimal} indicates. For $r\in[r_{2,3}, r_{1,2}]$, we have $\gamma(r)=\gamma_3(r)=\alpha/r-1$, and thus
\begin{align*}
  \kappa(\gamma(r),r)
  \ = \
  \kappa
  \left(
  \frac{\alpha}{r}
   -  1,r
  \right)
  \ = \
  1
  \ + \
  \frac{1}{r}
  (1-\alpha(1-d))
  \,.
\end{align*}
Since $\alpha(1-d)>1$ by assumption, this function is increasing in $r$ on the interval $[r_{2,3},r_{1,2}]$, so
$\kappa(\gamma(r),r)$ is minimal at $r=r_{2,3}$ on this interval.
As both $r_{1,3}\le r_{2,3}$ and $\gamma(r_{1,3})\le \gamma(r_{2,3})$ (see again the right panel of Figure~\ref{fig:optimal}), and since $\kappa$ is decreasing in both $\gamma$ and $r$, the minimum of $\kappa(\gamma(r),r)$ over the full interval $[r_{1,3},r_{2,3}]$ is attained at $r=r_{2,3}$. Finally, given that $r_{2,3}$ lies in the center of the interval $[1/(1-d),\alpha]$, it is therefore an admissible choice for minimizing $\kappa$ over $K$. Therefore, the optimal choice in the case $\alpha(1-d)(1-2d)\le 1$ is
\[
(\gamma_0, r_0) = \left( \frac{\alpha(1-d)-1}{\alpha(1-d)+1}
  \,,
  \frac{1}{2}
  \left(
  \frac{1}{1-d}
  \ + \
  \alpha
  \right)
  \right)
  \,.
\]
Plugging the values of $\gamma_0$ and $r_0$ into $\kappa(\gamma,r)$ finishes the proof of (i).
\\
  \textbf{Proof of Part (ii):}
Note that we already showed that $r_0\in(1/(1-d),\alpha)$, and the statement about $\gamma_0$ follows from $\gamma_0=\gamma_1(r_0)\wedge \gamma_2(r_0)\wedge \gamma_3(r_0)$.
From what we showed it also follows that $\gamma_0$ lies on the upper bound of the admissible interval for $\gamma$.
It remains to show $\kappa_0 - (d+1/\alpha)<0$. In the case $\alpha(1-d)(1-2d)>1$ this follows from
\begin{align*}
  \frac{1}{\alpha(1-d)}
  \ - \
  d
  \ - \
  \frac{1}{\alpha}
  \ = \
  \frac{1 - (1-d)(\alpha d + 1 )}{\alpha(1-d)}
  \ = \
  \frac{d(1-\alpha(1-d))}{\alpha(1-d)}
  \ < \
  0
  \,,
\end{align*}
since $\alpha(1-d)>1$.
In the opposite case $\alpha(1-d)(1-2d)\le 1$, it holds that
\begin{align*}
  &
    d
    \ + \
    (1-d)
    \frac{3-\alpha(1-d)}{\alpha(1-d)+1}
\ - \
d
\ - \
\frac{1}{\alpha}
  \\&
  \ = \
  \frac{1}{\alpha(\alpha(1-d)+1)}
  \left(
  \alpha(1-d)(3 - \alpha(1-d))
  \ - \
  \alpha(1-d) - 1
  \right)
  \\&
  \ = \
  \frac{1}{\alpha(\alpha(1-d)+1)}
  \left(
  \alpha(1-d)(2 - \alpha(1-d))
  \ - \
  1
  \right)
  \,.
  \intertext{By $\alpha(1-d)(1-2d)\le 1$ and $\alpha(1-d)\le \alpha \le 2$ this is bounded above by}
  &
  \frac{1}{\alpha(\alpha(1-d)+1)}
  \left(
  \frac{
  2 - \alpha(1-d)
  }{1-2d}
  \ - \
  1
  \right)
  \,,
  \intertext{which in turn by $\alpha(1-d)>1$ is bounded above by }
  &
  \frac{1}{\alpha(\alpha(1-d)+1)}
  \left(
  \frac{
  1
  }{1-2d}
  \ - \
  1
  \right)
  \ < \ 0\,.
\end{align*}
It follows (ii).
\end{proof}
\subsection{Proof of Corollaries~\ref{cor:heavy_det} and~\ref{cor:light_det}}
\begin{proof}[Proof of Corollary~\ref{cor:heavy_det}]
  We want to apply Corollary~\ref{cor:sub_ord_univariateI_optim} for $G_n(x)=\ind\{x>u_n\}$ and $G_n(x)=\log(x/u_n)\ind\{x>u_n\}$, respectively.
  Note that both choices of $G_n$ satisfy Assumption~\ref{asu:G_n} for all $\gamma_G>0$ arbitrarily small.
  By Lemma~\ref{lem:swap} it holds $\EE[G_n(X_0)] = \PP[X_0>u_n]$ and $G_{\infty,n}'(0)
    =f_{X_0}(u_n)$ for $G_n(x)=\ind\{x>u_n\}$. Besides, for $G_n(x)=\log(x/u_n)\ind\{x>u_n\}$, one has
  \begin{align*}
    \EE[G_n(X_0)]
    \ &= \
    \int_{u_n}^\infty \log(x/u_n)f_{X_0}(x)
    \,\mathrm{d}x
    \ \sim\
    \frac{
    \PP[X_0>u_n]
    }{\nu}
  \end{align*}
  by the regular variation property of $f_{X_0}$ (see {\it e.g.} \cite[Equation~(3.2.1)]{dehaanExtremeValueTheory2006}, and
  \begin{align*}
    G_{\infty,n}'(0)
    \ &= \
    -\int_{u_n}^\infty \log(x/u_n)f'_{X_0}(x)
    \,\mathrm{d}x
    \ = \
    \int_{u_n}^\infty
    \frac{
    f_{X_0}(x)
    }{x}
    \,\mathrm{d}x
    \ \sim\
    \frac{
    f_{X_0}(u_n)
    }{\nu+1}
    \,,
  \end{align*}
  by partial integration
  and Karamata's theorem \cite[Theorem~B.1.5]{dehaanExtremeValueTheory2006}.
  Therefore it holds for $G_n(x)=\ind\{x>u_n\}$ that
  \begin{align*}
    \frac{\EE[G_n(X_0)]}{G_{\infty,n}'(0)}
    \ = \
    \frac{
    \PP[X_0>u_n]
    }{f_{X_0}(u_n)}
    \ \sim
    \
    \frac{u_n}{\nu}
    \,,
  \end{align*}
  by Karamata's theorem again, and for $G_n(x)=\log(x/u_n)\ind\{x>u_n\}$ that
  \begin{align*}
    \frac{\EE[G_n(X_0)]}{G_{\infty,n}'(0)}
    \sim
    \frac{
    \PP[X_0>u_n]
    }
    {f_{X_0}(u_n)}
    \frac{\nu+1}{\nu}
    \ \sim \
    u_n
    \frac{1}{\nu}
    \frac{\nu+1}{\nu}
    \,.
  \end{align*}
  We finally check that $n^{\kappa_0+\delta/2-1/(\nu\land 2)-d}/G_{\infty,n}'(0)\to 0$.
  For this it suffices to consider
  \begin{align*}
    \frac{n^{\kappa_0+\delta/2-1/(\nu\land 2)-d}}{f_{X_0}(u_n)}
   \  \sim\
    L(u_n)
    n^{\kappa_0+\delta/2-1/(\nu\land 2)-d}
    u_n^{\nu+1}
    \,,
  \end{align*}
  where $L:x\mapsto x^{-\nu-1}/f_{X_0}(x)$ is slowly varying.
  Note that the term in the last display converges to zero since $u_n^{\nu+1}=o(n^{-\kappa_0-\delta+d+1/(\nu\land 2)})$ and $L(u_n)=o(n^{\delta/2})$.
  Therefore, we may apply Corollary~\ref{cor:sub_ord_univariateI_optim} to yield the desired results.
\end{proof}
For the proof of Corollary~\ref{cor:light_det} we require the following auxiliary result.
\begin{lemma}
\label{lem:gronwall}
  If $\omega\in C^1(\RR)$ is positive and satisfies, for some $\beta>0$, that $x^{1-\beta}\omega'(x)/\omega(x)$ converges to 0 as $x\to\infty$, then for all $p\in\mathbb{R}$ it holds that  $x^{p}\omega(x)\exp(-x^{\beta}/\beta)\to 0$ as $x\to\infty$.
\end{lemma}
\begin{proof}
  Since
  $x^{1-\beta}\omega'(x)/\omega(x)\to 0$ as $x\to\infty$,
  there exists $x_0>0$ with $\omega'(x)\leq \frac{1}{2} x^{\beta-1} \omega(x)$ for all $x> x_0$.
  By Gr{\"o}nwall's lemma it holds, for $x>x_0$,
 \begin{align*}
  \omega(x)
  \ \le \
  \omega(x_0)
  \exp(-x_0^{\beta}/(2\beta))
  \exp(x^{\beta}/(2\beta))
  \,.
\end{align*}
Therefore, it holds that for any $p\in\mathbb{R}$, as $x\to\infty$,
  \begin{align*}
  x^{p}
  \omega(x)
  \exp(-x^{\beta}/\beta)
  \ \lesssim \
  x^{p}
  \exp(-x^{\beta}/(2\beta))
  \ \to\ 0
  \,,
\end{align*}
which immediately yields the result since $\omega$ is assumed to be non-negative.
\end{proof}
\begin{proof}[Proof of Corollary~\ref{cor:light_det}]
  We proceed as in the proof of Corollary~\ref{cor:heavy_det}.
  There we checked that Assumption~\ref{asu:G_n} holds. We next establish the limit
  \begin{equation}
  \label{eqn:hopital}
  \lim_{n\to\infty}
    \frac
    {u_n^{1-\beta} f_{X_0}(u_n)}
    {\PP[X_0>u_n]} = 1.
  \end{equation}
  Let $\phi:x\mapsto x^{1-\beta} f_{X_0}(x)$ and $\psi:x\mapsto \PP[X_0>x]$. The functions $\phi$ and $\psi$ are differentiable in a neighborhood of infinity, converge to 0 at infinity, and $\psi'(x)=-f_{X_0}(x)<0$ for $x$ large enough. Moreover $x^{1-\beta}\omega'(x)/\omega(x)\to 0$ as $x\to\infty$ since $\beta>0$ and $\omega'$ is either identically zero far from the origin or regularly varying at infinity, so that
  \[
  \lim_{x\to\infty} \frac{\phi'(x)}{\psi'(x)} = \lim_{x\to\infty} -\frac{(1-\beta) x^{-\beta} f_{X_0}(x) + x^{1-\beta} f'_{X_0}(x)}{f_{X_0}(x)} = \lim_{x\to\infty} \left( 1-x^{1-\beta}\frac{\omega'(x)}{\omega(x)} \right)=1.
  \]
  Then~\eqref{eqn:hopital} follows from l'H\^{o}pital's rule, that is,
  \[
  \lim_{n\to\infty}
    \frac
    {u_n^{1-\beta} f_{X_0}(u_n)}
    {\PP[X_0>u_n]} = \lim_{n\to\infty}
    \frac
    {\phi(u_n)}
    {\psi(u_n)} = \lim_{n\to\infty}
    \frac
    {\phi'(u_n)}
    {\psi'(u_n)} = 1.
  \]
  Therefore $\PP[X_0>u_n]\sim u_n^{1-\beta}f_{X_0}(u_n)$.
  Next, note that by partial integration it holds
  \begin{align*}
  \EE[G_n(X_0)] &=\int_{u_n}^\infty
    \log(x/u_n)
    f_{X_0}(x)
    \,\mathrm{d}x \\
    &= -
    \lim_{x\to\infty}
    \log(x/u_n)\PP[X_0>x]
    \ +\
    \int_{u_n}^\infty
    \frac{\PP[X_0>x]}{x}
    \,\mathrm{d}x \ =\  \int_{u_n}^\infty
    \frac{\PP[X_0>x]}{x}
    \,\mathrm{d}x
  \end{align*}
  since the boundary term vanishes due to Lemma~\ref{lem:gronwall}. Similarly we get
  \begin{align*}
  G_{\infty,n}'(0)=-\int_{u_n}^\infty
    \log(x/u_n)
    f'_{X_0}(x)
    \,\mathrm{d}x
    \ =\
    \int_{u_n}^\infty
    \frac{f_{X_0}(x)}{x}
    \,\mathrm{d}x
    \,.
  \end{align*}
  Note that
  \begin{align*}
    \lim_{n\to\infty}
    \frac{u_n^{-\beta} \PP[X_0>u_n]}
    {
    \int_{u_n}^{\infty}
    \frac{\PP[X_0>x]}{x}
    \,\mathrm{d}x
    }
    \ = \
    \lim_{n\to\infty}
    \frac{\beta u_n^{-\beta-1}\PP[X_0>u_n]+u_n^{-\beta}f_{X_0}(u_n)}{
    \PP[X_0>u_n]/u_n
    }
    \ = \ 1
    \,,
  \end{align*}
  and
  \begin{align*}
    \lim_{n\to\infty}
    \frac{u_n^{-1} \PP[X_0>u_n]}
    {
    \int_{u_n}^{\infty}
    \frac{f_{X_0}(x)}{x}
    \,\mathrm{d}x
    }
    \ = \
    \lim_{n\to\infty}
    \frac{u_n^{-2}\PP[X_0>u_n]+u_n^{-1}f_{X_0}(u_n)}{
    f_{X_0}(u_n)
    /u_n
    }
    \ = \ 1
    \,,
  \end{align*}
  where we used twice that $\PP[X_0>u_n]\sim u_n^{1-\beta}f_{X_0}(u_n)$ and l'H\^{o}pital's rule in a way similar to what was done above. This yields $\EE[G_n(X_0)]/G_{\infty,n}'(0)\sim u_n^{1-\beta}$.
  Finally we check that $n^{\kappa_0+\delta/2-d-1/2}/G_{\infty,n}'(0)\to 0$.
  Since $G_{\infty,n}'(0)\sim u_n^{-\beta} f_{X_0}(u_n)$, this boils down to
  \begin{align}
  \label{eq:omega_vanish}
  \frac{
  n^{\kappa_0+\delta/2-d-1/2}
  }{u_n^{-\beta}\omega(u_n)\exp(-u_n^{\beta}/\beta)}
  \ = \
  \frac{n^{-\delta/2}}{u_n^{-\beta}\omega(u_n)}
  \ \to \ 0
  \end{align}
  due to the regular variation property of $\omega$.
  Then, as before, applying Corollary~\ref{cor:sub_ord_univariateI_optim} yields the results.
\end{proof}

\section{Derandomization device}

We state and prove a lemma which is the main tool for obtaining convergence of PoT estimators with random thresholds (Corollaries~\ref{cor:heavy_rand} and~\ref{cor:light_rand}) from their convergence with deterministic thresholds (Corollaries~\ref{cor:heavy_det} and~\ref{cor:light_det}).
\begin{lemma}
  \label{lem:derand}
  Let $(X_t)$ be a stationary time series, and
  let $g$ be an increasing, continuously differentiable function whose derivative $g'$ is regularly varying at infinity.
  For $x>0$ define
  \begin{align*}
    \xi(x)
    \ := \
    \frac{xf_{X_0}(x)}{\PP[X_0>x]}
    \,.
  \end{align*}
  Assume the following:
  \begin{itemize}
    \item
    There exists a threshold sequence $(u_n)$ with $u_n\to\infty$, and a rate sequence $(r_n)$ with $r_n\to\infty$, such that $r_n/\xi(u_n)\to\infty$,
 and there exists a constant $c>0$ such that
      \begin{align}
        \label{lem:derand:c1}
        \begin{split}
        u_n
        \cdot
        g'(u_n)
        \frac
        {
        \PP[X_0>u_n]
        }
        {
\EE
\left[
  \,
\left(
      g(X_0)
       -
      g(u_n)
      \right)
      \mathbbm{1}
      \left\{
        X_0>u_n
      \right\}
\right]
        }
        \
        \sim
        \
          c
          \cdot
          \xi(u_n)
          \,,
        \end{split}
      \end{align}
      and,
      for all $t\in\mathbb{R}$,
\begin{align}
  \label{asu:derand:iii}
        r_n
        \left(
        \frac{\PP[X_0>u_n]}{\PP[X_0>(1+t/r_n)\cdot u_n]}
        \ - \
        1
        \right)
        \ \sim\
        t
        \cdot
        \xi(u_n)
        \,.
      \end{align}
\item
    There exists $\delta_\epsilon\in(0,1)$ such that for all
    sequences $(\epsilon_n)$ converging to $0$ and satisfying
    \begin{align}
      \label{cond:epsilon}
      |\epsilon_n|= \mathcal{O}(r_n^{\delta_\epsilon-1})
    \end{align}
    the following holds:
\begin{align}
  \label{lem:derand:clt}
  \begin{split}
    &
      \frac
      {
    r_n/\xi(u_n)
      }
      {n}
      \\&
      \cdot \
      \sum_{t=1}^{n}
     \left[
      \left(
      \frac
      {
      \left(
      g(X_t)
       -
      g(u_n)
      \right)
      \mathbbm{1}
      \left\{
        X_t>u_n
      \right\}
    }
      {
\EE
\left[
  \,
\left(
      g(X_0)
       -
      g(u_n)
      \right)
      \mathbbm{1}
      \left\{
        X_0>u_n
      \right\}
\right]
      }
      \ - \
      1
      \right)
      \,,
      \left(
      \frac
      {
            \mathbbm{1}
      \left\{
        X_t>u_n(1+\epsilon_n)
      \right\}
    }
      {
      \PP[
        X_0>u_n(1+\epsilon_n)
      ]
      }
      \ - \
      1
      \right)
    \right]
  \end{split}
      \end{align}
      converges in distribution, as $n\to\infty$, to a pair of non-degenerate continuous random variables $[\Phi, \Psi]$.
        \item
There exists a positive sequence $(\tau_n)$
such that
\begin{align}
\label{cond:tau:1}
\tau_n \ = \  \mathcal{O}(r_n^{\delta_\epsilon - 1})
\,,
\end{align}
    \begin{align}
      \label{cond:tau:2}
      \tau_nr_n\to \infty
      \,,
    \end{align}
and
      \begin{align}
        \label{asu:derand:iv}
        \frac{\PP[X_0>u_n]}{\PP[X_0>u_n(1\pm\tau_n)]}
        \ \to \
        1
        \,,
      \end{align}
      as $n\to\infty$.
  \end{itemize}
  If additionally $n\PP[X_0>u_n]\to\infty$, then under the above assumptions, for
  \begin{align*}
    k
    \ = \
    k(n,u_n)
    \ := \
        \lfloor n\PP[X_0>u_n]\rfloor
        \,,
  \end{align*}
  it holds that
      \begin{align*}
        &
     \left[
        \frac{r_n}{\xi(u_n)}
     \frac{1}{k}
      \sum_{i=1}^{k}
      \left(
      \frac
      {
        g(X_{n-i+1:n})
      \ - \
        g(X_{n-k:n})
    }
      {
\EE
\left[
        \left(
          \,
      g(X_0)
      -
      g(u_n)
        \right)
      \mathbbm{1}
      \left\{
        X_0>u_n
      \right\}
\right]
      }
      \frac{k}{n}
      \ - \ 1
      \right)
      \,
      , \
      r_n
      \left(
      \frac
      {
        X_{n-k :n}
      }
      {u_n}
    \ - \
    1
      \right)
    \right]
      \end{align*}
      converges in distribution to
      $
      \left[
      \Phi
      \ - \
      c\cdot
      \Psi
      \,,
      \Psi
      \right]
      $ as $n\to\infty$.
\end{lemma}
To prove this result, we need the following lemma, which establishes a connection between random thresholds and their deterministic counterparts.
\begin{lemma}\emph{(Koenker-type inversion lemma)}
  \label{lem:koen}
  Let
  $k(u_n,n):= \lfloor n\PP[X_0>u_n]  \rfloor$, and assume that Conditions~\eqref{asu:derand:iii} and~\eqref{lem:derand:clt} from Lemma~\ref{lem:derand} hold. Then:
  \begin{enumerate}[label=(\roman*)]
    \item
    The inequality
  \begin{align*}
      r_n
      \left(
        \frac
        {X_{n-k:n}}
        {u_n}
        \ - \
        1
      \right)
      &
      \
      \le
      \
      s
      \intertext{is equivalent to}
      r_n
      \left(
        \frac
        {1}
        {n}
        \sum_{t=1}^{n}
        \frac
        {
\mathbbm{1}
        \left\{
          X_t>u_n(1+s/r_n)
        \right\}
        }
        {
          \PP[
          X_0>u_n(1+s/r_n)
          ]
        }
        \ - \
        1
              \right)
              &
              \ \le \
              r_n
              \left(
                \frac
        {
          \PP[X_0>u_n]
        }
        {
          \PP[
          X_0>u_n(1+s/r_n)
          ]
        }
        \ - \
        1
              \right)
              \,.
  \end{align*}
\item
As $n\to\infty$, it holds that
\begin{align*}
&
  r_n
  \left[
      \frac
      {
    1/\xi(u_n)
      }
      {n}
      \sum_{t=1}^{n}
      \left(
      \frac
      {
      \left(
      g(X_t)
       -
      g(u_n)
      \right)
      \mathbbm{1}
      \left\{
        X_t>u_n
      \right\}
    }
      {
\EE
\left[
  \,
\left(
      g(X_0)
       -
      g(u_n)
      \right)
      \mathbbm{1}
      \left\{
        X_0>u_n
      \right\}
\right]
      }
      \ - \
      1
      \right)
      \,,
      \left(
        \frac
        {X_{n-k:n}}
        {u_n}
        \ - \
        1
      \right)
  \right]
  \\&
      \
      \stackrel{\mathrm{d}}{\longrightarrow}
      \
      [
      \Phi
      \,,
      \Psi
      ].
\end{align*}
  \end{enumerate}
\end{lemma}
\begin{proof}
  \textbf{Proof of Part (i):}
  For any $x\in\mathbb{R}$, it holds that
  \begin{align*}
    X_{\lceil n \alpha \rceil:n}
    \ \le \ x
    \qquad
    \text{if and only if}
    \qquad
    \alpha
    \ \le \
        \frac
        {1}
        {n}
        \sum_{t=1}^{n}
\mathbbm{1}
        \left\{
          X_t\le x
        \right\}
        \,.
  \end{align*}
  This equivalence follows from a general property of quantile functions: for any distribution function $G$ and its associated quantile function $Q$, we have $Q(\alpha)\le x$ if and only if $\alpha\le G(x)$, where $Q$ is defined as the left-continuous inverse of $G$.
  \\
  \textbf{Proof of Part (ii):}
  Let $s_1,s_2\in\RR$ and define the event
  \[
  E_n(s_1)=\left\{ \frac
      {
    r_n/\xi(u_n)
      }
      {n}
      \sum_{t=1}^{n}
      \left(
      \frac
      {
      \left(
      g(X_t)
       -
      g(u_n)
      \right)
      \mathbbm{1}
      \left\{
        X_t>u_n
      \right\}
    }
      {
\EE
\left[
  \,
\left(
      g(X_0)
       -
      g(u_n)
      \right)
      \mathbbm{1}
      \left\{
        X_0>u_n
      \right\}
\right]
      }
      \ - \
      1
      \right)
      \ \le \
      s_1
      \right\}.
  \]
  By Part \textit{(i)}, together with Conditions~\eqref{asu:derand:iii} and~\eqref{lem:derand:clt} and the continuity of $\Psi$, we obtain
  \begin{align*}
    &
    \PP
    \left[E_n(s_1),\,
      r_n
      \left(
        \frac
        {X_{n-k:n}}
        {u_n}
        \ - \
        1
      \right)
      \
      \le
      \
      s_2
    \right]
    \\&
    \ = \
       \PP
    \left[
    E_n(s_1),\,
r_n
      \left(
        \frac
        {1}
        {n}
        \sum_{t=1}^{n}
        \frac
        {
\mathbbm{1}
        \left\{
          X_t>u_n(1+s_2/r_n)
        \right\}
        }
        {
          \PP[
          X_0>u_n(1+s_2/r_n)
          ]
        }
        \ - \
        1
              \right)
              \ \le \
              r_n
              \left(
                \frac
        {
          \PP[X_0>u_n]
        }
        {
          \PP[
          X_0>u_n(1+s_2/r_n)
          ]
        }
        \ - \
        1
              \right)
    \right]
    \\&
    \ = \
 \PP
    \left[
    E_n(s_1),\,
r_n
      \left(
        \frac
        {1}
        {n}
        \sum_{t=1}^{n}
        \frac
        {
\mathbbm{1}
        \left\{
          X_t>u_n(1+s_2/r_n)
        \right\}
        }
        {
          \PP[
          X_0>u_n(1+s_2/r_n)
          ]
        }
        \ - \
        1
              \right)
              \ \le \
              s_2
              \cdot \xi(u_n)
              (1 + o(1))
    \right]
    \\&
    \ \to \
 \PP
    \left[
    \Phi \le s_1\,,
    \Psi \le s_2
    \right]
    \,,
  \end{align*}
  from which the convergence in distribution follows by Slutsky's lemma (note that the quantity
  $\epsilon_n:=s_2/r_n$ indeed satisfies Condition~\eqref{cond:epsilon}).
    \end{proof}

\begin{proof}[Proof of Lemma~\ref{lem:derand}]
  Define
  \begin{align*}
        \widehat{e}_{g,n}(k,u_n)
    &
    \ := \
    \frac
    {1}
    {k}
    \sum_{t=1}^{n}
    \left(
    g(X_{t})
    \ - \
    g(u_n)
    \right)
    \mathbbm{1}
    \left\{
      X_t>u_n
    \right\}
    \,.
  \end{align*}
  Note that
  \begin{align*}
    \widehat{e}_{g,n}(k,X_{n-k:n})
    &
    \ = \
\frac
    {1}
    {k}
    \sum_{t=1}^{n}
    \left(
    g(X_{t})
    \ - \
    g(X_{n-k:n})
    \right)
    \ind\{X_{t}>X_{n-k:n}\}
    \,,
    \end{align*}
    and therefore
  \begin{align}
    \label{derand:1}
    \widehat{e}_{g,n}(k,X_{n-k:n})
    \ - \
    \widehat{e}_{g,n}(k,u_n)
    &
    \ = \
    g(u_n)
    \ - \
    g(X_{n-k:n})
    \ + \
    R_1
    \ + \
    R_2
    \,,
  \end{align}
  where
  \begin{align*}
    R_1
    &
    \ := \
    \left(
    g(X_{n-k:n})
    \ - \
    g(u_n)
    \right)
    \left(
    1
    \ - \
    \frac{1}{k}
    \sum_{t=1}^n
    \ind\{X_t>X_{n-k:n}\}
    \right)
    \,,
    \\
    R_2
    &
    \ := \
\frac
    {1}
    {k}
    \sum_{t=1}^{n}
    \left(
      g(X_t)
      \ - \
      g(u_n)
    \right)
    \left(
    \mathbbm{1}
    \left\{
      X_t>X_{n-k:n}
    \right\}
    -
    \mathbbm{1}
    \left\{
      X_t>u_n
    \right\}
    \right)
    \,.
  \end{align*}
  We will show that $R_j=o_{\PP}(g(u_n)-g(X_{n-k:n}))$ for $j=1,2$, that is,
  \begin{align}
    \label{R:vanish:1}
    \frac{R_j}{g(u_n)-g(X_{n-k:n})}
    \ \stackrel{\PP}{\longrightarrow} \
    0
    \,.
  \end{align}
  Combining \eqref{R:vanish:1} with \eqref{derand:1}, it will follow that
\begin{align}
  \label{derand:eq:P}
     \widehat{e}_{g,n}(k,X_{n-k:n})
    \ - \
    \widehat{e}_{g,n}(k,u_n)
    \ = \
    (g(u_n)
    \ - \
    g(X_{n-k:n}))
    ( 1 \ + \
    o_\PP(1))
    \,.
    \end{align}
    We begin with the analysis of $R_1$.
 By Lemma~\ref{lem:koen}.(ii), we have
    \begin{align}
      \label{eq:derand:rand}
      r_n
      \left(
      \frac{X_{n-k:n}}{u_n}
      \ - \
      1
      \right)
      \ \stackrel{\mathrm{d}}{\longrightarrow} \
      \Psi.
    \end{align}
    Let $(\tau_n)$ be as in Condition~\eqref{cond:tau:1} and \eqref{cond:tau:2}.
    Then,
    we obtain
    \begin{align}
      \label{eq:rand:1}
      \frac{1}{\tau_n}
      \left(
      \frac{X_{n-k:n}}{u_n}
      \ - \
      1
      \right)
      \ = \
      \frac{1}{r_n\tau_n}
      r_n
      \left(
      \frac{X_{n-k:n}}{u_n}
      \ - \
      1
      \right)
      \ \stackrel{\PP}{\longrightarrow} \
      0\,,
    \end{align}
    since $r_n\tau_n\to\infty$ as $n\to\infty$.
    Consequently, with arbitrarily high probability for large $n$, we have
    \begin{align}
      \label{eq:4444}
      \left(
      1
      \ - \
      \tau_n
      \right)
      u_n
      \ \le \
      X_{n-k:n}
      \ \le \
      \left(
      1
      \ + \
      \tau_n
      \right)
      u_n
      \,.
    \end{align}
    This implies that
\begin{align}
  \label{2210}
  \begin{split}
  &
  \left|
  \frac{R_1}{g(u_n)-g(X_{n-k:n})}
  \right|
  \ = \
  \left|
  \left(
    \frac{1}{k}
    \sum_{t=1}^n
    \ind\{X_t>X_{n-k:n}\}
  \right)
    \ - \
    1
  \right|
  \\&
  \ \le \
 \left|
 \left(
    \frac{1}{k}
      \sum_{t=1}^{n}
    \mathbbm{1}
    \left\{
      X_t>
      \left(
      1
      \ - \
      \tau_n
      \right)
      u_n
    \right\}
 \right)
    \ - \
    1
    \right|
    \
    \lor
    \
    \left|
    \left(
    \frac{1}{k}
      \sum_{t=1}^{n}
    \mathbbm{1}
    \left\{
      X_t>
      \left(
      1
      \ + \
      \tau_n
      \right)
      u_n
    \right\}
    \right)
    \ - \
    1
    \right|
       \,,
  \end{split}
    \end{align}
  with arbitrarily high probability as $n\to\infty$.
    By Assumption~\eqref{asu:derand:iv}
    and definition $k=\lfloor n\PP[X_0>u_n]\rfloor$, we obtain
    \begin{align}
    \nonumber
      &
    \left|
    \left(
      \frac
      {1}
      {k}
      \sum_{t=1}^{n}
    \mathbbm{1}
    \left\{
      X_t>u_n(1\pm\tau_n)
    \right\}
    \right)
    \ - \
    1
    \right|
        \\
        \nonumber
        &
        \ \le \
        \frac
        {n\PP[X_0>u_n]}
        {\lfloor n\PP[X_0>u_n]\rfloor}
        \cdot
        \frac
        {\PP[X_0>u_n(1\pm \tau_n)]}
        {\PP[X_0>u_n]}
        \cdot
        \frac{1}{r_n/\xi(u_n)}
        \cdot
    \left|
      \frac
        {r_n/\xi(u_n)}
      {n}
      \sum_{t=1}^{n}
    \left(
      \frac{
    \mathbbm{1}
    \left\{
      X_t>u_n(1\pm\tau_n)
    \right\}
      }{
        \PP[X_0>u_n(1\pm\tau_n)]
      }
    \ - \
    1
    \right)
    \right|
        \\
        \nonumber
        &
    \qquad + \
    \left|
        \frac{n\PP[X_0>u_n]}
        {\lfloor n\PP[X_0>u_n]\rfloor}
        \cdot
        \frac
        {\PP[X_0>u_n(1\pm \tau_n)]}
        {\PP[X_0>u_n]}
        \ - \
        1
            \right|
           \\
            \label{vanish:2}
            &
    \ \stackrel{\PP}{\longrightarrow} \
    0
    \,,
    \end{align}
    as $n\to\infty$, since each term on the right-hand side converges to zero in probability due to
    \begin{align*}
\frac
        {n\PP[X_0>u_n]}
        {\lfloor n\PP[X_0>u_n]\rfloor}
        \,,\,
        \frac
        {\PP[X_0>u_n(1\pm \tau_n)]}
        {\PP[X_0>u_n]}
        \ \to \
        1
        \qquad\text{and}\qquad
        \frac{1}{r_n/\xi(u_n)}
        \ \to \
        0\,,
    \end{align*}
    by $n\PP[X_0>u_n]\to\infty$, Assumption~\eqref{asu:derand:iv}
    and $r_n/\xi(u_n)\to\infty$, and
    \begin{align*}
    \left|
      \frac
      {r_n/\xi(u_n)}
      {n}
      \sum_{t=1}^{n}
      \left( \frac{
    \mathbbm{1}
    \left\{
      X_t>u_n(1\pm\tau_n)
    \right\}
      }{
      \PP[X_0>u_n(1\pm\tau_n)]
      }
    \ - \
    1 \right)
    \right|
    \ \stackrel{\mathrm{d}}{\longrightarrow} \
    |\Psi|
    \,,
    \end{align*}
    which is guaranteed by \eqref{lem:derand:clt} and the fact that
    $(\tau_n)$ also satisfies \eqref{cond:tau:1}. Using \eqref{2210} and \eqref{vanish:2} we get \eqref{R:vanish:1} for $j=1$.

    We now analyze $R_2$. Since $g$ is increasing, we can write
  \begin{align*}
    &
    \left|
      g(u_n)
      \ - \
      g(X_t)
    \right| \left|
    \mathbbm{1}
    \left\{
      X_t>X_{n-k:n}
    \right\}
    -
    \mathbbm{1}
    \left\{
      X_t>u_n
    \right\}
    \right|
    \\&
    \ = \
\left(
      g(u_n)
      \ - \
      g(X_t)
    \right)
    \ind\{
    u_n > X_t > X_{n-k:n}
    \}
    \\&
    \qquad + \
    \left(
      g(X_t)
      \ - \
      g(u_n)
    \right)
    \ind\{
   X_{n-k:n} \ge X_t > u_n
    \}
    \,.
    \intertext{By monotonicity of $g$, we may further bound this as}
    &
    \ \le \
\left(
      g(u_n)
      \ - \
      g(X_{n-k:n})
    \right)
    \ind\{
    u_n>X_t > X_{n-k:n}
    \}
    \\&
    \qquad + \
    \left(
      g(X_{n-k:n})
      \ - \
      g(u_n)
    \right)
    \ind\{
    X_{n-k:n}>X_t > u_n
    \}
    \\&
    \ = \
    \left(
      g(X_{n-k:n})
      \ - \
      g(u_n)
    \right)
    \left(
    \ind\{
    X_{n-k:n}>X_t > u_n
    \}
    \ - \
    \ind\{
    u_n>X_t > X_{n-k:n}
    \}
    \right)
    \,.
  \end{align*}
  It follows, rewriting $R_2$ and carefully considering the sign in the absolute value, that
  \begin{align}
    \nonumber
    &
    |R_2| \ = \
      \left|
      \frac
    {1}
    {k}
    \sum_{t=1}^{n}
    \left(
      g(X_t)
      \ - \
      g(u_n)
    \right)
    \left(
    \mathbbm{1}
    \left\{
      X_t>X_{n-k:n}
    \right\}
    -
    \mathbbm{1}
    \left\{
      X_t>u_n
    \right\}
    \right)
      \right|
      \\
      \nonumber
      &
    \ \leq \ \left(
      g(X_{n-k:n})
      \ - \
      g(u_n)
    \right) \cdot
    \frac{1}{k}
    \sum_{t=1}^n
    \left( \ind\{
    X_{n-k:n} > X_t > u_n
    \} - \ind\{
    u_n > X_t > X_{n-k:n}
    \} \right)
    \\
    \nonumber
    &
      \ \le \
      \left(
    g(X_{n-k:n})
    \ - \
    g(u_n)
      \right)
    \cdot
      \frac
      {1}
      {k}
      \sum_{t=1}^{n} \Big(
    \mathbbm{1}
    \left\{
      X_t>u_n
    \right\}
    \ - \
      \mathbbm{1}
    \left\{
      X_t>X_{n-k:n}
      \right\} \Big)
      \\
      \label{R:vanish:2}
      &
    \ \le \
    \left|
    g(u_n)
    \ - \
    g(X_{n-k:n})
    \right|
    \cdot
    \left|
      \frac
      {1}
      {k}
      \sum_{t=1}^{n} \Big(
    \mathbbm{1}
    \left\{
      X_t>X_{n-k:n}
    \right\}
    \ - \
    \mathbbm{1}
    \left\{
      X_t>u_n
    \right\} \Big)
    \right|
    \,.
  \end{align}
  We now aim to show that
  \begin{align}
    \label{R:vanish}
    \left|
      \frac
      {1}
      {k}
      \sum_{t=1}^{n} \Big(
    \mathbbm{1}
    \left\{
      X_t>X_{n-k:n}
    \right\}
    \ - \
    \mathbbm{1}
    \left\{
      X_t>u_n
    \right\} \Big)
    \right|
    \ \stackrel{\PP}{\longrightarrow} \
    0
    \,,
  \end{align}
  which, when combined with \eqref{R:vanish:2}, establishes \eqref{R:vanish:1} for $j=2$.
  To prove this convergence, note that by~\eqref{eq:4444},
  with arbitrarily high probability as $n\to\infty$, the following holds:
    \begin{align}
      \label{vanish:21}
    \begin{split}
      &
    \left|
      \frac
      {1}
      {k}
      \sum_{t=1}^{n} \Big(
    \mathbbm{1}
    \left\{
      X_t>X_{n-k:n}
    \right\}
    \ - \
    \mathbbm{1}
    \left\{
      X_t>u_n
    \right\} \Big)
    \right|
      \\&
    \ \le \
    \left|
    \left(
      \frac
      {1}
      {k}
      \sum_{t=1}^{n}
    \mathbbm{1}
    \left\{
      X_t>X_{n-k:n}
    \right\}
    \right)
    \
     -
    \
    1
    \right|
    \ + \
    \left|
    \left(
      \frac
      {1}
      {k}
      \sum_{t=1}^{n}
    \mathbbm{1}
    \left\{
      X_t>u_n
    \right\}
    \right)
    \ - \
    1
    \right|
      \\&
    \ \le \
 \left|
 \left(
    \frac{1}{k}
      \sum_{t=1}^{n}
    \mathbbm{1}
    \left\{
      X_t>
      \left(
      1
      \ - \
      \tau_n
      \right)
      u_n
    \right\}
 \right)
    \ - \
    1
    \right|
    \ \lor \
    \left|
    \left(
    \frac{1}{k}
      \sum_{t=1}^{n}
    \mathbbm{1}
    \left\{
      X_t>
      \left(
      1
      \ + \
      \tau_n
      \right)
      u_n
    \right\}
    \right)
    \ - \
    1
    \right|
    \\&
    \qquad + \
    \left|
    \left(
      \frac
      {1}
      {k}
      \sum_{t=1}^{n}
    \mathbbm{1}
    \left\{
      X_t>u_n
    \right\}
    \right)
    \ - \
    1
    \right|
    \,.
    \end{split}
    \end{align}
    Applying Inequalities~\eqref{vanish:2} and~\eqref{vanish:21}, we obtain the desired convergence in~\eqref{R:vanish}, and hence conclude~\eqref{R:vanish:1}, and therefore \eqref{derand:eq:P}.
    We now show that
\begin{align}
  \label{eq:g:1}
    g(X_{n-k:n})
    \ - \
    g(u_n)
    \ = \
    u_n
    g'
    (u_n)
\left(
      \frac
      {
    X_{n-k:n}
      }
      {u_n}
      \ - \
      1
    \right)
    \ + \
    o_{\PP}
    \left(
      \frac
      {u_n\cdot g'(u_n)}
      {r_n}
    \right)
    \,,
\end{align}
and, using this together with~\eqref{derand:eq:P},~\eqref{eq:derand:rand}, and \eqref{eq:g:1}, we will conclude that
  \begin{align}
  \label{eq:g:grand_derand}
    \widehat{e}_{g,n}(k,X_{n-k:n})
    \ = \
    \widehat{e}_{g,n}(k,u_n)
    \ - \
    u_n
    \cdot
    g'
    (u_n)
    \left(
      \frac
      {
    X_{n-k:n}
      }
      {u_n}
      \ - \
      1
    \right)
    \ + \
o_{\PP}
    \left(
      \frac
      {u_n\cdot g'(u_n)}
      {r_n}
    \right)
    \,.
  \end{align}
  To establish Equation~\eqref{eq:g:1}, observe that
\begin{align}
  \label{eq:1:int:g}
  \begin{split}
  &
    g(X_{n-k:n})
    \ - \
    g(u_n)
\ - \
    u_n
    g'
    (u_n)
\left(
      \frac
      {
    X_{n-k:n}
      }
      {u_n}
      \ - \
      1
    \right)
    \\&
    \ = \
    u_n g'(u_n)
    \int_1^{X_{n-k:n}/u_n} \bigg(
    \frac{g'(xu_n)}{g'(u_n)}
    \ - \ 1 \bigg)
    \,\mathrm{d}x
    \,.
    \end{split}
\end{align}
Since $g'$ is regularly varying and $X_{n-k:n}/u_n\stackrel{\PP}{\longrightarrow}1$
by \eqref{eq:rand:1}, it follows from \cite[Proposition~B.1.10]{dehaanExtremeValueTheory2006} that
\begin{align*}
    \int_1^{X_{n-k:n}/u_n} \left(
    \frac{g'(xu_n)}{g'(u_n)}
    \ - \ 1 \right)
    \,\mathrm{d}x
    \ = \
    o_\PP
    \left(  \left|
  \frac{X_{n-k:n}}{u_n}
  \ - \ 1 \right|
    \right)
    \ = \
    o_\PP
    \left(
    \frac{1}{r_n}
    \right)
    \,,
\end{align*}
where the final equality follows from \eqref{eq:derand:rand}.
Combining this with \eqref{eq:1:int:g}, we obtain the desired result in \eqref{eq:g:1}.
  We now prove the stated result using Equation~\eqref{eq:g:grand_derand}.
  To this end, note that by the stationarity of the time series, $k/n \sim \PP[X_0>u_n]$, and Assumption~\eqref{lem:derand:c1} it follows that
  \begin{align}
    \label{expect:1}
  \EE[\widehat{e}_{g,n}(k,u_n)]
    &
    \ = \
  \frac{n}{k}
  \EE
  [
  \left(
  g(X_0)-g(u_n)
  \right)
  \ind\{X_0>u_n\}
  ]
      \\
      \nonumber
      &
    \ \sim \
    \frac{
  \EE
  [
  \left(
  g(X_0)-g(u_n)
  \right)
  \ind\{X_0>u_n\}
  ]
    }{\PP[X_0>u_n]}
    \\
    \nonumber
    &
    \ \sim \
    \frac{u_ng'(u_n)}{c\cdot \xi(u_n)}
    \,.
  \end{align}
  Consequently, we can write
  \begin{align*}
    &
  \frac{
  r_n
  }{\xi(u_n)}
\left(
  \frac
  {
    \widehat{e}_{g,n}(k,u_n)
  }
  {
  \EE[\widehat{e}_{g,n}(k,u_n)]
  }
  \ - \
1
\right)
    \\&
    \ = \
\frac{
    r_n/\xi(u_n)
  }{n}
    \sum_{t=1}^n
    \left(
  \frac
    {
    \left(
      g(X_t)-g(u_n)
    \right)
      \ind\{X_t>u_n\}
    }
  {
      \EE[
      \left(
      g(X_0)-g(u_n)
      \right)
      \ind\{X_0>u_n\}
      ]
  }
  \ - \
1
    \right)
    \,.
  \end{align*}
  Substituting the approximation for
  $
    \widehat{e}_{g,n}(k,X_{n-k:n})
  $
  from
  \eqref{eq:g:grand_derand}
  and combining with the joint convergence established in
Lemma~\ref{lem:koen}.(ii), we conclude
\begin{align*}
  &
  \frac{
  r_n
  }{\xi(u_n)}
\left(
  \frac
  {
    \widehat{e}_{g,n}(k,X_{n-k:n})
  }
  {
  \EE[\widehat{e}_{g,n}(k,u_n)]
  }
  \ - \
1
\right)
  \\&
    \
    \ = \
    \
  \frac{
  r_n
  }{\xi(u_n)}
\left(
  \frac
  {
    \widehat{e}_{g,n}(k,u_n)
  }
  {
  \EE[\widehat{e}_{g,n}(k,u_n)]
  }
  \ - \
1
\right)
\ - \
c\cdot
r_n
      \left(
      \frac
      {
    X_{n-k:n}
      }
      {u_n}
      \ - \
      1
    \right)
    \ + \
o_{\PP}
  (
  1)
  \,.
  \end{align*}
  Therefore, by Lemma~\ref{lem:koen}.(ii), it follows that
\begin{align*}
  &
  \left[
  \frac{r_n}{\xi(u_n)}
    \left(
      \frac
      {
        \widehat{e}_{g,n}(k,X_{n-k:n})
    }
      {
        \EE
        \left[
      \widehat{e}_{g,n}(k,u_n)
        \right]
      }
        \ - \
        1
    \right)
    , \
    r_n
    \left(
      \frac
      {X_{n-k:n}}
      {u_n}
      \ - \
      1
    \right)
  \right]
  \
  \stackrel{\mathrm{d}}{\longrightarrow}
  \
  \left[
    \Phi
    \ - \
    c\cdot
    \Psi
    ,
    \Psi
  \right]
  \,.
\end{align*}
To complete the proof, we use Equation~\eqref{expect:1} and the fact that
\[
 \widehat{e}_{g,n}(k,X_{n-k:n}) = \frac{1}{k}
    \sum_{i=1}^{k}
    g(X_{n-i+1:n})-g(X_{n-k:n})
\]
to express
\begin{align*}
      \frac
      {
        \widehat{e}_{g,n}(k,X_{n-k:n})
    }
      {
        \EE
        \left[
      \widehat{e}_{g,n}(k,u_n)
        \right]
      }
        \ - \
        1
    \ = \
    \frac{1}{k}
    \sum_{i=1}^{k}
    \left(
    \frac{g(X_{n-i+1:n})-g(X_{n-k:n})}{\EE[(g(X_0)-g(u_n))\ind\{X_0>u_n\}]}
    \frac{k}{n}
    \ - \ 1
    \right)
\end{align*}
from which the stated convergence result follows immediately.
\end{proof}
\section{Verifying the conditions of Lemma~\ref{lem:derand} in the settings of Corollary~\ref{cor:heavy_det} and~\ref{cor:light_det}}
In the next lemma we check the assumptions of Lemma~\ref{lem:derand} in the settings of Corollary~\ref{cor:heavy_det} and Corollary~\ref{cor:light_det}. Applying Lemma~\ref{lem:derand} then directly results in Corollary~\ref{cor:heavy_rand} and~\ref{cor:light_rand}.
\begin{lemma}
  \label{lem:derand:poly}
  Let $g=\log$. Then the assumptions of Lemma~\ref{lem:derand} hold in the settings of (i) Corollary~\ref{cor:heavy_det} and (ii) Corollary~\ref{cor:light_det}, respectively, with the following parameters:
  \begin{enumerate}[label=(\roman*)]
    \item In Corollary~\ref{cor:heavy_det}, $r_n=n^{1-1/(\nu\land 2)-d} u_n$,
\begin{align*}
    \xi(u_n)
    \ = \
    \frac{u_nf_{X_0}(u_n)}{\PP[X_0>u_n]}
    \ \to \
    \nu
    \,,
\end{align*}
and the joint limit is
    \begin{align*}
    (\Phi, \Psi)
    \ = \
    \left( \frac{\nu}{1+\nu}
    Z_{2\land\nu}, \ Z_{2\land\nu} \right)
    \,,
  \end{align*}
  with constant $c=1$ and $(\tau_n)$ any sequence satisfying \eqref{cond:tau:1} and \eqref{cond:tau:2}.
    \item In Corollary~\ref{cor:heavy_det}, $r_n=n^{1/2-d} u_n$,
\begin{align*}
    \xi(u_n)
    \ = \
    \frac{u_nf_{X_0}(u_n)}{\PP[X_0>u_n]}
    \ \sim \
   u_n^\beta
   \,,
\end{align*}
and the joint limit is $(\Phi, \Psi) = (Z_2, Z_2)$,
  with constant $c=1$ and $(\tau_n)$ any sequence
  satisfying \eqref{cond:tau:1} and \eqref{cond:tau:2}.
  \end{enumerate}
 \end{lemma}

 \begin{proof}[Proof of Lemma~\ref{lem:derand:poly}]
We begin by showing all the assumptions of Lemma~\ref{lem:derand} except for the joint convergence Assumption~\eqref{lem:derand:clt}. We do this separately for both settings.

In Setting (i), it follows from Karamata's theorem~\cite[Theorem~B.1.5]{dehaanExtremeValueTheory2006}, that
  \[
    \xi(u_n)
    \ = \
    \frac{u_nf_{X_0}(u_n)}{\PP[X_0>u_n]}
    \ \to \
    \nu
    \,.
  \]
  This implies $r_n/\xi(u_n)\to\infty$. Note now that, by Karamata's theorem, we have
  \begin{align}
  \nonumber
\EE
\left[
    \log
    \left(
    X_0/u_n
    \right)_+
\right]
    &
    \ = \
    u_n
    f_{X_0}(u_n)
    \int_{1}^{\infty}
    \log(x)
    \frac{
    f_{X_0}(u_n x)
    }{
    f_{X_0}(u_n)
    }
    \,\mathrm{d}x
    \\
    \nonumber
    &
    \ \sim \
    \nu
    \PP[X_0>u_n]
    \int_{1}^{\infty}
    \log(x)
    \frac{1}{x^{1+\nu}}
        \,\mathrm{d}x
        \label{i_know}
        \\&
       \ = \
        \frac{
    \PP[X_0>u_n]
        }{\nu}
        \,.
  \end{align}
  Since $g'(u_n)=1/u_n$, and using Equation~\eqref{i_know}, we find
    that
  \begin{align*}
    u_n
    g'(u_n)
    \frac{\PP[X_0>u_n]}{\EE[(g(X_0)-g(u_n))\ind\{X_0>u_n\}]} = \frac{\PP[X_0>u_n]}{\EE
\left[
    \log
    \left(
    X_0/u_n
    \right)_+
\right]}
    \ \sim \
    \nu
    \ \sim \
    \xi(u_n)
    \,,
  \end{align*}
  which shows that Assumption~\eqref{lem:derand:c1} holds with constant $c=1$. Applying the mean value theorem and the regular variation property of $f_{X_0}$, along with the local uniformity property of regularly varying functions,
  we conclude that for each $t\in\mathbb{R}$,
  there exists a sequence $(\theta_n)=(\theta_n(t)) \subset (0,1)$ such that
  \begin{align*}
    &
 r_n
    \left(
    \frac{\PP[X_0>u_n]}{\PP[X_0>(1+t/r_n)\cdot u_n]}
    \ - \ 1
    \right)
    \\&
    \ = \
    \frac{r_n}{\PP[X_0>(1+t/r_n)u_n]}
    \frac{u_nt}{r_n}
    f_{X_0}((1+\theta_nt/r_n)u_n)
    \\&
    \ = \
    t
    \cdot
    \frac{u_n}{u_n(1+t/r_n)}
    \frac{
    u_n(1+t/r_n)\cdot  f_{X_0}(u_n(1+t/r_n))
    }{\PP[X_0>u_n(1+t/r_n)]}
    \cdot
    \frac{
    f_{X_0}(u_n(1+\theta_nt/r_n))
    }{
    f_{X_0}(u_n(1+t/r_n))
    }
    \\&
    \ \sim \
    t\cdot \xi(u_n(1+t/r_n))
    \\&
    \ \sim\
    t\cdot \xi(u_n)
    \,,
  \end{align*}
  where the last approximation uses that $\xi(x)\to \nu$ as $x\to\infty$, as shown above.
  Thus, Assumption~\eqref{asu:derand:iii} is satisfied.
  Moreover, by the local uniformity of regularly varying functions,
  we have that for any sequence $(\tau_n)$ with $\tau_n\to 0$ as $n\to\infty$,
  \begin{align*}
    \frac{\PP[X_0>u_n]}{\PP[X_0>u_n(1\pm \tau_n)]}
    \ \to \
    1
    \,.
  \end{align*}
  Therefore, Assumption~\eqref{asu:derand:iv} is also satisfied.

Now we turn to Setting (ii).
  We use again
  \[
    \xi(u_n)
    \ = \
    \frac{u_nf_{X_0}(u_n)}{\PP[X_0>u_n]}
    \ \sim \
    u_n^{1-(1-\beta)}
    =
    u_n^{\beta}
    \,
  \]
  such that $\xi(u_n)/r_n \sim u_n^\beta/r_n \to 0 $, since $u_n$ grows only logarithmically while $r_n$ grows polynomially in $n$.
  As shown in the proof of Corollary~\ref{cor:light_det}, we have
  \begin{align*}
    u_n
    g'(u_n)
    \frac{\PP[X_0>u_n]}{\EE[(g(X_0)-g(u_n))\ind\{X_0>u_n\}]}
    \ =\
    \frac{\PP[X_0>u_n]}{
    \int_{u_n}^{\infty}
    \frac{\PP[X_0>x]}{x}
    \,
    \mathrm{d}x
    }
    \ \sim\
    u_n^\beta
    \ \sim \
    \xi(u_n)
    \,,
  \end{align*}
  which confirms that Assumption~\eqref{lem:derand:c1}
  holds with constant $c=1$. Applying the mean value theorem, we find that for each $t\in\mathbb{R}$, there exists a sequence $(\theta_n(t))\subset(0,1)$ such that
\begin{align*}
    r_n
    \left(
    \frac{\PP[X_0>u_n]}
    {\PP[X_0>(1+t/r_n)u_n]}
    \ - \
    1
    \right)
    &
    \ = \
    \frac{r_n}{\PP[X_0>(1+t/r_n)u_n]}
    \frac{u_nt}{r_n}
    f_{X_0}((1+\theta_nt/r_n)u_n)
    \\&
    \ \sim \
    \frac{u_n^{\beta}
    }
    {f_{X_0}((1+t/r_n)u_n)}
    t
    f_{X_0}((1+\theta_nt/r_n)u_n)
    \\&
    \ \sim\
    t
    u_n^\beta
    \ \sim\
    t\xi(u_n)
    \,.
      \end{align*}
          This confirms that Assumption~\eqref{asu:derand:iii} holds in Case (ii).
      For any sequence $(\tau_n)$ satisfying Conditions~\eqref{cond:tau:1} and \eqref{cond:tau:2}, we have
      \begin{align*}
        \frac{\PP[X_0>u_n]}{\PP[X_0>u_n(1\pm \tau_n)]}
        &
        \ \sim \
        \frac{f_{X_0}(u_n)}{f_{X_0}(u_n(1\pm \tau_n))}
        \frac{1}{(1\pm \tau_n)^{1-\beta}}
                \\&
                \ \sim \
                \frac{\omega(u_n)}{\omega(u_n(1\pm \tau_n))}
                \exp(-u_n^\beta(1 - (1\pm\tau_n)^\beta)/\beta)
        \ \to \
        1
        \,,
      \end{align*}
      as $n\to\infty$. This holds because $\tau_n$ decreases polynomially in $n$ (since $r_n$ grows polynomially), whereas $u_n$ grows only logarithmically.
       Thus Assumption~\eqref{asu:derand:iv} is satisfied.

 It remains to establish the joint convergence stated in~\eqref{lem:derand:clt}. To this end, we apply the Cram\'er-Wold device with coefficients $\lambda_1,\lambda_2\in\mathbb{R}$. Our goal is to show that for any sequence $(\epsilon_n)$,
   satisfying \eqref{cond:epsilon},
 \begin{align}
   \nonumber
    &
     \frac{r_n/\xi(u_n)}{n}
      \sum_{t=1}^n
     \left[ \lambda_1
    \left(
    \frac{
\left(
     g(X_t)
     -
     g(u_n)
\right)_+
            }{
\EE
\left[
\left(
     g(X_0)
     -
     g(u_n)
\right)_+
    \right]
    }
    \ - \ 1
    \right)
    \ + \
    \lambda_2
    \left(
    \frac{\ind\{X_t>u_n(1+\epsilon_n)\}}{
    \PP[
    X_0>u_n(1+\epsilon_n)
    ]
    }
    \ - \ 1
     \right) \right]
     \intertext{converges in distribution to}
    &
    \label{goal:wold}
    \begin{cases}
    Z_{\nu\land 2}
    \left(
    \lambda_1
    \dfrac{\nu}{1+\nu}
    \ + \
    \lambda_2
    \right)
      &\qquad\text{in Case (i), and}\\
    Z_{2}
    \left(
    \lambda_1
    \ + \
    \lambda_2
    \right)
      &\qquad\text{in Case (ii).}
    \end{cases}
  \end{align}
  The proof consists in reducing~\eqref{goal:wold} to an application of Theorem~\ref{thm:approx_optimal}, and it is organized in two consecutive steps. Define the function
\begin{align*}
  G_n(x)
  \ = \
  \lambda_1
  \cdot
  \frac{
  \left(
  g(x)
  -
  g(u_n)
  \right)_+
  \PP[X_0>u_n]
  }{
  \EE[
  \left(
  g(X_0)
  -
  g(u_n)
  \right)_+
  ]
  }
  \ + \
  \lambda_2
  \cdot
  \ind\{x>u_n\}
  \,.
\end{align*}
{\bf Step 1:} Observe that by construction
  \begin{align*}
    \EE[G_n(X_0)]
    \ = \
    \left(
    \lambda_1
    \ + \
    \lambda_2
    \right)
    \PP[X_0>u_n]
    \,,
  \end{align*}
  which yields
  \begin{align*}
    &   \frac{1}{\PP[X_0>u_n]} \sum_{t=1}^n
    \left(
    G_n(X_t)
    \ - \
    \EE[G_n(X_0)]
    \right)=    \sum_{t=1}^n
    \left(
    \frac{G_n(X_t)}{\PP[X_0>u_n]}
    \ - \
    (\lambda_1+\lambda_2)
    \right).
  \end{align*}
  We now write the decomposition
      \begin{align*}
  &
    \sum_{t=1}^n
    \left(
    \frac{
    G_n(X_t)
    }{
    \PP[X_0>u_n]
    }
    \ - \
    (\lambda_1+\lambda_2)
    \right)
    \\&
    \ = \
    \sum_{t=1}^n
     \left[ \lambda_1
    \left(
    \frac{
\left(
     g(X_t)
     -
     g(u_n)
\right)_+
            }{
\EE
\left[
\left(
     g(X_0)
     -
     g(u_n)
\right)_+
    \right]
    }
    \ - \ 1
    \right)
    \ + \
    \lambda_2
    \left(
    \frac{\ind\{X_t>u_n(1+\epsilon_n)\}}{
    \PP[
    X_0>u_n(1+\epsilon_n)
    ]
    }
    \ - \ 1
     \right) \right] \\
     &\ + \
     \frac{n}{r_n/\xi(u_n)}
    \lambda_2
    R
    \,,
      \end{align*}
      where
      \begin{align*}
        R
        \ := \
\frac{r_n/\xi(u_n)}{n}
    \sum_{t=1}^n
    \left(
    \frac
    {
    \ind\{X_t>u_n\}
    }{
    \PP[
    X_0>u_n
    ]
    }
    \ - \
    \frac
    {
    \ind\{X_t>u_n(1+\epsilon_n)\}
    }{
    \PP[
    X_0>u_n(1+\epsilon_n)
    ]
    }
    \right)
        \ = \
        R_1
        \ - \
        R_2
    \,,
      \end{align*}
      and $R_1, R_2$ are defined as
\begin{align*}
  R_1
  &
  \ := \
  \frac{
  \PP[X_0>u_n]
-
  \PP[X_0>u_n(1+\epsilon_n)]
  }
  {\PP[X_0>u_n]}
  \frac{r_n/\xi(u_n)}{n}
  \\&
  \qquad\cdot
  \sum_{t=1}^n
  \left(
  \frac{
  \ind\{X_t>u_n\}
  \ - \
  \ind\{X_t>u_n(1+\epsilon_n)\}
  }
  {
  \EE[
  \ind\{X_0>u_n\}
  \ - \
  \ind\{X_0>u_n(1+\epsilon_n)\}
  ]
  }
  \ - \
  1
  \right)
  \,,
  \\
  R_2
  &
  \ := \
  \frac{
  \PP[X_0>u_n]
  \
-
  \
  \PP[X_0>u_n(1+\epsilon_n)]}
  {\PP[X_0>u_n]}
  \frac{r_n/\xi(u_n)}{n}
  \sum_{t=1}^n
  \left(
  \frac{
  \ind\{X_t>u_n(1+\epsilon_n)\}
  }
  {
  \PP[
  X_0>u_n(1+\epsilon_n)
  ]
  }
  \ - \
  1
  \right)
  \,.
\end{align*}
If we can show that $R_1,R_2\stackrel{\PP}{\longrightarrow} 0$, then the convergence in Equation~\eqref{goal:wold} follows from the convergence of
\[
\frac{r_n/\xi(u_n)}{n \PP[X_0>u_n]} \sum_{t=1}^n
    \left(
    G_n(X_t)
    \ - \
    \EE[G_n(X_0)]
    \right).
\]
{\bf Step 2:} Obviously $r_n=n^{1-(d+1/\alpha)}$ with
\begin{align*}
  \alpha
  \ = \
    \begin{cases}
      \nu\land 2
          &\qquad\text{in Case (i), and}\\
          2
      &\qquad\text{in Case (ii).}
    \end{cases}
\end{align*}
We aim to combine Theorems~\ref{thm:approx_optimal} and Theorem~\ref{thm:clt_partsum} to show that
  \begin{align}
    \nonumber
    &
    \frac{r_n/\xi(u_n)}{n \PP[X_0>u_n]} \sum_{t=1}^n
    \left(
    G_n(X_t)
    \ - \
    \EE[G_n(X_0)]
    \right) \\
    \label{eq:28891}
    &= n^{1-(d+1/\alpha)}
    \cdot \frac{1}
    {
    f_{X_0}(u_n)
    }
    \cdot \frac{1}{n}
    \sum_{t=1}^n
    \left(
    G_n(X_t)
    \ - \
    \EE[G_n(X_0)]
    \right)
  \end{align}
  converges in distribution to the limit announced in~\eqref{goal:wold}. To apply this result it is sufficient to show that
  \begin{align}
    \label{eq:still_to_show}
    \begin{split}
    G_{\infty,n}'(0)
  \ = \
    \begin{cases}
      f_{X_0}(u_n)
    \left(
    \lambda_1
    \dfrac{\nu}{1+\nu}
    \ + \
    \lambda_2
    \ + \
              o(1)
    \right)
              &\qquad\text{in Case (i),}\\[10pt]
              f_{X_0}(u_n)
              (
              \lambda_1
              \ + \
              \lambda_2
              \ + \
              o(1)
              )
      &\qquad\text{in Case (ii).}
    \end{cases}
    \end{split}
  \end{align}
  Indeed, from this, Theorem~\ref{thm:approx_optimal} will imply
  that in Case (i),
 \begin{align*}
    &
    \norm{
    n^{-d-1/\alpha}
    \frac{1}{
    f_{X_0}(u_n)
    }
    \sum_{t=1}^{n}
    (G_n(X_t)
    \ - \
    \EE[G_n(X_0)])
    \ - \
    \left(
    \lambda_1
    \frac{\nu}{1+\nu}
    \ + \
    \lambda_2
    \ + \
    o(1)
    \right)
    n^{-d-1/\alpha}
    \sum_{t=1}^{n}
    X_t
    }_{L^{r_0}(\PP)}
    \\&
    \ \lesssim \
    \frac{
   n^{\kappa_0+\delta/2-d-1/\alpha}
       }{
    f_{X_0}(u_n)
    }
\ \to \
0
\,,
  \end{align*}
  as $n\to\infty$.
  This conclusion follows from $f_{X_0}\in\mathrm{RV}_{-\nu-1}$ and
  $u_n=\operatorname{o}(n^{(d+1/\alpha-\kappa_0-\delta)/(\nu+1)})$.
  In Case (ii), we similarly obtain
\begin{align*}
    &
    \norm{
    n^{-d-1/2}
    \frac{1}{
    f_{X_0}(u_n)
    }
    \sum_{t=1}^{n}
    (G_n(X_t)
    \ - \
    \EE[G_n(X_0)])
    \ - \
    \left(
    \lambda_1
    \ + \
    \lambda_2
    \ + \
    o(1)
    \right)
    n^{-d-1/2}
    \sum_{t=1}^{n}
    X_t
    }_{L^{r_0}(\PP)}
    \\&
    \ \lesssim \
    \frac{
  n^{\kappa_0+\delta/2-d-1/2}
       }{
    f_{X_0}(u_n)
    }
    \to
    0
    \,,
    \end{align*}
  as $n\to\infty$, due to \eqref{eq:omega_vanish}.
  \eqref{eq:28891} and the announced result.

We now provide details on both steps.
\subsection*{Details Step 1: Analysis of $R_1$}
We apply Theorem~\ref{thm:approx_optimal} using the function
\begin{align*}
  G_n^{R_1}
  (x)
  \ = \
  \ind\{x>u_n\}
  \ - \
  \ind\{x>u_n(1+\epsilon_n)\}
  \,,
\end{align*}
which satisfies the bound
\begin{align*}
|G_n^{R_1}(x)|\le \ind\{x>u_n(1-(-\epsilon_n)_+)\}\,.
\end{align*}
This ensures that Assumption~\ref{asu:G_n} is met, with the threshold $u_n$ replaced by $u_n(1-(-\epsilon_n)_+)$.
We then decompose $R_1$ as
\begin{align*}
  R_1
  &
  \ = \
  R_{1,1}
  \ + \
  R_{1,2}
  \,,
\end{align*}
where
\begin{align*}
  R_{1,1}
  &
  \ := \
  \frac{
  1}{\PP[X_0>u_n]}
  \frac{r_n/\xi(u_n)}{n} \cdot
  \sum_{t=1}^n
  \left(
  G_n^{R_1}(X_t)
  \ - \
  \EE[ G_n^{R_1}(X_0) ]
  \ - \
  (G_{n,\infty}^{R_1})'(0)
  X_t
  \right)
  \,,
  \\
R_{1,2}
  &
  \ := \
  \frac{
  1}{\PP[X_0>u_n]}
  \frac{r_n/\xi(u_n)}{n} \cdot
  (G_{n,\infty}^{R_1})'(0)
  \sum_{t=1}^n
  X_t
  \,.
\end{align*}
First we analyze $R_{1,1}$. Note that by construction
\begin{align*}
  \frac{
  1}{\PP[X_0>u_n]}
  \frac{r_n/\xi(u_n)}{n}
  =
  \frac{n^{-d-1/\alpha}}{f_{X_0}(u_n)}
  \,.
\end{align*}
Applying Theorem~\ref{thm:approx_optimal}, we find that
\begin{align*}
  &
  \norm{
  R_{1,1}
  }_{L^{r_0}(\PP)}
  \\&
  \ = \
  \frac{
  1}{f_{X_0}(u_n)}
  \norm{
  n^{-d-1/\alpha}
  \sum_{t=1}^n
  \left(
  G_n^{R_1}(X_t)
  \ - \
  \EE[ G_n^{R_1}(X_0) ]
  \ - \
  (G_{n,\infty}^{R_1})'(0)
  X_t
  \right)
  }_{L^{r_0}(\PP)}
  \\&
  \ \le \
  \frac{n^{\kappa_0+\delta/2-d-1/\alpha}}{f_{X_0}(u_n)}
  \ \to \
  0
  \,,
\end{align*}
due to the assumed growth rate on $u_n$, as was shown during the outline of Step 2, so that $R_{1,1}\stackrel{\PP}{\longrightarrow} 0$.
Next we turn to $R_{1,2}$.
Since $r_n/(nu_n)=n^{-d-1/\alpha}$, it follows from
Theorem~\ref{thm:clt_partsum} that
\begin{align*}
  \frac{r_n}{nu_n}
  \sum_{t=1}^n
  X_t
  \stackrel{\mathrm{d}}{\longrightarrow}
  Z_{\alpha}\,.
\end{align*}
The remaining factor is
\begin{align*}
\frac{
  u_n
  (G_{n,\infty}^{R_1})'(0)
  }{\PP[X_0>u_n]\xi(u_n)}
  \,,
\end{align*}
which we will show
tends to zero as $n\to\infty$.
The stated convergence of $R_{1,2}$ then follows by Slutsky's lemma.
We now use Lemma~\ref{lem:swap}(ii) in order to evaluate
\begin{align*}
  &
\frac{
  u_n
  (G_{n,\infty}^{R_1})'(0)
  }{\PP[X_0>u_n]\xi(u_n)}
  \ = \
  \frac{f_{X_0}(u_n(1+\epsilon_n))-f_{X_0}(u_n)}{f_{X_0}(u_n)}
  \ = \
  \frac{f_{X_0}(u_n(1+\epsilon_n))}{f_{X_0}(u_n)}
  \ - \
  1
  \,.
\end{align*}
In Case (i), we assume $f_{X_0}\in\mathrm{RV}_{-\nu-1}$, so by local uniformity of regular variation
\begin{align*}
  \frac{f_{X_0}(u_n(1+\epsilon_n))}{f_{X_0}(u_n)}
  \ = \
  (1+\epsilon_n)^{-\nu-1}
  \ + \
  o(1)
  \ \to \
  1
  \,,
\end{align*}
as $n\to\infty$, since $\epsilon_n\to 0$.
In Case (ii), we have
\begin{align}
\label{helper123}
  \frac{f_{X_0}(u_n(1+\epsilon_n))}{f_{X_0}(u_n)}
  \ = \
  \frac{\omega(u_n(1+\epsilon_n))}{\omega(u_n)}
  \exp(u_n^\beta(1 - (1+\epsilon_n)^\beta)/\beta)
  \ \to \
  1
  \,,
\end{align}
since $\epsilon_n$ decays polynomially in $n$ while $u_n$ grows only logarithmically, and $\omega$ is either constant or regularly varying at infinity.
This shows $R_{1,2}\stackrel{\PP}{\longrightarrow} 0$, and thus $R_{1}\stackrel{\PP}{\longrightarrow}0$.
\subsection*{Details Step 1: Analysis of $R_2$}
The analysis of $R_2$ is
more straightforward, as the conditions of Theorem~\ref{thm:approx_optimal} are easily verified for the function
\begin{align*}
  G_n^{R_2}(x)
  \ = \
  \ind\{x>u_n(1+\varepsilon_n)\}
  \,.
\end{align*}
Recall first of all that we have just shown the convergence $f_{X_0}(u_n(1+\epsilon_n))/f_{X_0}(u_n)\to 1$. Moreover, in Case (i), since $f_{X_0}\in \mathrm{RV}_{-\nu-1}$,
  it follows that $x\mapsto \PP[X_0>x]\in \mathrm{RV}_{-\nu}$.
  Therefore,
  \begin{align*}
  \frac{
  \PP[X_0>u_n(1+\epsilon_n)]
  \
-
  \
  \PP[X_0>u_n]
  }{\PP[X_0>u_n]}
    \ = \
    (1+\epsilon_n)^{-\nu}
    \ - \ 1
    \ + \ o(1)
  \ \to\ 0
  \,.
  \end{align*}
  In Case (ii), we apply the mean value theorem: there exists $(\theta_n)\subset(0,1)$ such that
 \begin{align*}
   &
  \frac{
  \PP[X_0>u_n(1+\epsilon_n)]
  \
-
  \
  \PP[X_0>u_n]
  }{\PP[X_0>u_n]}
    \ = \
    -u_n\epsilon_n
    \frac{f_{X_0}(u_n)}{\PP[X_0>u_n]}
    \frac{f_{X_0}(u_n(1+\theta_n\epsilon_n))}{f_{X_0}(u_n)}
    \ \to \ 0\,,
  \end{align*}
  as $n\to\infty$, since $u_nf_{X_0}/\PP[X_0>u_n]\sim u_n^{\beta}$ with $\epsilon_n u_n^\beta\to 0$, and the remaining term converges to 1 due to \eqref{helper123}. This entails $\PP[X_0>u_n(1+\epsilon_n)]/\PP[X_0>u_n]\to 1$. Then, applying Theorem~\ref{thm:approx_optimal} and Theorem~\ref{thm:clt_partsum}, we obtain that
\begin{align*}
  &
  n^{1-(d+1/\alpha)} \cdot \frac{\PP[X_0>u_n(1+\varepsilon_n)]}{f_{X_0}(u_n(1+\varepsilon_n))} \cdot \frac{1}{n}
  \sum_{t=1}^n
  \left(
  \frac{\ind\{X_t>u_n(1+\epsilon_n)\}}{\PP[X_0>u_n(1+\epsilon_n)]}
  \ - \
  1
  \right)
\end{align*}
  converges in distribution to a non-degenerate random variable, and so does
  \begin{align*}
  &
  n^{1-(d+1/\alpha)} \cdot \frac{\PP[X_0>u_n]}{f_{X_0}(u_n)} \cdot \frac{1}{n}
  \sum_{t=1}^n
  \left(
  \frac{\ind\{X_t>u_n(1+\epsilon_n)\}}{\PP[X_0>u_n(1+\epsilon_n)]}
  \ - \
  1
  \right) \\
  &=
\frac{r_n/\xi(u_n)}{n}
  \sum_{t=1}^n
  \left(
  \frac{\ind\{X_t>u_n(1+\epsilon_n)\}}{\PP[X_0>u_n(1+\epsilon_n)]}
  \ - \
  1
  \right) \\
  &= \frac
  {\PP[X_0>u_n]}
  {
  \PP[X_0>u_n(1+\epsilon_n)]
  \
-
  \
  \PP[X_0>u_n]
  }
  R_2.
\end{align*}
We conclude that in both settings it holds $R_2\stackrel{\PP}{\longrightarrow}0$.

We will now continue the analysis separated into the settings, where
it suffices to check the validity of \eqref{eq:still_to_show} in order to complete the proof.
\subsection*{Details Step 2}
We begin by verifying Equation~\eqref{eq:still_to_show} in Setting (i).
To simplify the notation, we write
  \[
  (\log(x)-\log(u_n))\ind\{x>u_n\}
  \
  =
  \
    \max\{
  \log(x/u_n),0
    \}
    \ =: \
  \log(x/u_n)_+.
  \]
  It then follows that
  \begin{align*}
    G_{\infty,n}'(0)
    &
    \ = \
    -\int_\RR
    G_n(x)
    f_{X_0}'(x)
    \,\mathrm{d}x
    \\&
    \ = \
    -\lambda_1
\frac{
    \int_{u_n}^{\infty}
    \log
    \left(
    x/u_n
    \right)
    f_{X_0}'(x)
    \,\mathrm{d}x
    }{
\EE
\left[
    \log
    \left(
    X_0/u_n
    \right)_+
\right]
    }
    \PP[X_{0}>u_n]
    \ + \
    \lambda_2
    \cdot
    f_{X_0}(u_n)
    \,.
  \end{align*}
 Using integration by parts and the regular variation of $f_{X_0}\in\mathrm{RV}_{-\nu-1}$,
 we obtain
  \begin{align*}
    -\int_{u_n}^{\infty}
    \log
    \left(
    x/u_n
    \right)
    f_{X_0}'(x)
    \,\mathrm{d}x
    &
    \ = \
    \left[
    -f_{X_0}(x)
    \log(x/u_n)
    \right]^{x\to\infty}_{x=u_n}
    \ + \
\int_{u_n}^{\infty}
\frac{1}{x}
    f_{X_0}(x)
    \,\mathrm{d}x \\
    &\ = \
    \int_{u_n}^{\infty}
\frac{1}{x}
    f_{X_0}(x)
    \,\mathrm{d}x
    \,.
    \intertext{Changing variables,
    and using the uniform convergence of $f_{X_0}(u_nx)/f_{X_0}(u_n)$ to 1 in neighborhoods of infinity, this becomes:}
   -\int_{u_n}^{\infty}
    \log
    \left(
    x/u_n
    \right)
    f_{X_0}'(x)
    \,\mathrm{d}x
    &
           \ = \
           f_{X_0}(u_n)
\int_{1}^{\infty}
\frac{1}{x}
\frac{
           f_{X_0}(u_nx)
           }{f_{X_0}(u_n)}
           \,\mathrm{d}x
           \\&
           \ \sim \
           f_{X_0}(u_n)
\int_{1}^{\infty}
           \frac{1}{x^{2+\nu}}
           \,\mathrm{d}x
           \\&
           \ = \
           f_{X_0}(u_n)
           \frac{1}{1+\nu}
           \,.
         \end{align*}
  Recalling~\eqref{i_know}, it follows that
  \begin{align*}
    G_{\infty,n}'(0)
    \
    \ = \
    \
        f_{X_0}(u_n)
    \left(
    \lambda_1
    \frac{\nu}{1+\nu}
    \ + \
    \lambda_2
    \ + \
    o(1)
    \right)
    \,,
  \end{align*}
  which verifies Equation~\eqref{eq:still_to_show} in Setting (i).

To prove the same equation in Case (ii),
first note that in the proof of Corollary~\ref{cor:light_det} we showed
\begin{align*}
\frac{\EE[G_n(X_0)]}{G_{\infty,n}'(0)} = \frac
{
\int_{u_n}^{\infty}
\log(x/u_n)f_{X_0}(x)\,\mathrm{d}x
}
{
-
\int_{u_n}^{\infty}
\log(x/u_n)f'_{X_0}(x)\,\mathrm{d}x
}
\
\sim
\
\frac{\PP[X_0>u_n]}{f_{X_0}(u_n)}
\
\sim
\
u_n^{1-\beta}
\,.
\end{align*}
From this it follows immediately that
  \begin{align*}
    G_{\infty,n}'(0)
    \
    =
    \
    u_n^{\beta-1} \EE[G_n(X_0)]
    \
    &=
    \
    u_n^{\beta-1} \PP[X_0>u_n] (\lambda_1 + \lambda_2 + o(1)) \\
    &
    \
    =
    \
    f_{X_0}(u_n)(\lambda_1 + \lambda_2 + o(1))
    \, ,
  \end{align*}
  that is, \eqref{eq:still_to_show}.
Having shown \eqref{eq:still_to_show} in both cases, we conclude the proof.
  \end{proof}


%
%
%
%
\end{document}